\documentclass[letterpaper,11pt,oneside,reqno]{amsart}
\usepackage{amsfonts,amsmath, amssymb,amsthm,amscd}
\usepackage{bm}
\usepackage{mathrsfs}
\usepackage{verbatim}
\usepackage{hyperref}
\usepackage{graphicx}
\usepackage[utf8]{inputenc}
\usepackage{xcolor}
\usepackage{mathtools}
\usepackage{dsfont}
\usepackage{fullpage}
\usepackage{tikz}
\usepackage{setspace}
\usepackage{upgreek}

\begin{document}

%%%%%%%%%%%%%%%%%%%%%%SHORTCUTS%%%%%%%%%%%%%%%%%%%%%%%%%%%%%%%%%%%%%%%%%%%%%%%%%%%%%%%%%%%%%%
\newcommand{\e}{\varepsilon}
\newcommand{\ve}{\mathcal{E}}
\newcommand{\EE}{\ensuremath{\mathbb{E}}}
\newcommand{\qq}[1]{(q;q)_{#1}}
\newcommand{\PP}{\ensuremath{\mathbb{P}}}
\newcommand{\R}{\ensuremath{\mathbb{R}}}
\newcommand{\Rplus}{\ensuremath{\mathbb{R}_{+}}}
\newcommand{\C}{\ensuremath{\mathbb{C}}}
\newcommand{\Z}{\ensuremath{\mathbb{Z}}}
\newcommand{\N}{\ensuremath{\mathbb{N}}}
\newcommand{\Zgzero}{\ensuremath{\mathbb{Z}_{>0}}}
\newcommand{\Zgeqzero}{\ensuremath{\mathbb{Z}_{\geqslant 0}}}
\newcommand{\Zleqzero}{\ensuremath{\mathbb{Z}_{\leq 0}}}
\newcommand{\Q}{\ensuremath{\mathbb{Q}}}
\newcommand{\Y}{\ensuremath{\mathbb{Y}}}
\newcommand{\I}{\ensuremath{\mathbf{i}}}
\newcommand{\Real}{\ensuremath{\mathfrak{Re}}}
\newcommand{\Imag}{\ensuremath{\mathfrak{Im}}}
\newcommand{\subs}{\ensuremath{\mathbf{Subs}}}
\newcommand{\Res}[1]{\underset{{#1}}{\mathbf{Res}}}
\newcommand{\Resfrac}[1]{\mathbf{Res}_{{#1}}}
\newcommand{\Sub}[1]{\underset{{#1}}{\mathbf{Sub}}}
\newcommand{\Sym}{\ensuremath{\mathbf{Sym}}}
\newcommand{\phidist}{\ensuremath{\boldsymbol{\varphi}}}
\newcommand{\mskyl}{\ensuremath{\mathfrak{m}}}
\newcommand{\prech}{\ensuremath{\prec}}
\newcommand{\precv}{\ensuremath{\prec_{\mathrm{v}}}}
\newcommand{\gap}{\ensuremath{\mathrm{gap}}}
\newcommand{\diagq}{\ensuremath{a_{\circ}}}
\newcommand{\diag}{\ensuremath{\alpha_{\circ}}}
\newcommand{\whitenoise}{\ensuremath{\mathscr{\dot{W}}}}
\newcommand{\dist}{\textrm{dist}}
\def \Ai {{\rm Ai}}
\def \Pf {{\rm Pf}}
\def \sgn {{\rm sgn}}
\newcommand{\var}{{\rm var}}
\newcommand{\U}{\ensuremath{\mathcal{U}^{\mathlarger{\llcorner}}}}
\newcommand{\Udiag}{\ensuremath{\mathcal{U}^\angle}}
\newcommand{\rhodiag}{\ensuremath{\rho^{\swarrow}}}
\newcommand{\rholeft}{\ensuremath{\rho^{\leftarrow}}}
\newcommand{\rhoup}{\ensuremath{\rho^{\uparrow}}}
\newcommand{\Proj}{\ensuremath{\textrm{Proj}}}
\newcommand{\Projtilde}{\ensuremath{\widetilde{\textrm{Proj}}}}
\newcommand{\La}{\Lambda}
\newcommand{\la}{\lambda}
\newcommand{\ta}{\theta}

\newcommand{\kernel}{\mathsf{K}}
\newcommand{\bel}[1]{b_{#1}^{\mathrm{el}}}
\newcommand{\fkernel}{\mathsf{f}}
\newcommand{\gkernel}{\mathsf{g}}
\newcommand{\hkernel}{\mathsf{h}}
\newcommand{\scaling}[1]{ \mathfrak{s}(#1)}
\newcommand{\G}{G}
\newcommand{\p}{\mathsf{p}}
\newcommand{\q}{\mathsf{q}}
\newcommand{\eps}{\epsilon}
\newcommand{\ratealpha}{\upalpha}
\newcommand{\rategamma}{\upgamma}
\newcommand{\NN}{N}
\newcommand{\ssum}{\sum^{*}}

\newcommand{\PMM}{\ensuremath{\mathbb{PMM}}}
\newcommand{\PSM}{\ensuremath{\mathbb{PSM}}}
\newcommand{\PHL}{\ensuremath{\mathbb{PHL}}}
\newcommand{\PMP}{\ensuremath{\mathbb{PMP}}}
\newcommand{\EPMM}{\ensuremath{\mathbb{E}^{\mathrm{PMM}}}}
\newcommand{\EPSM}{\ensuremath{\mathbb{E}^{\mathrm{PSM}}}}
\newcommand{\EPHL}{\ensuremath{\mathbb{E}^{\mathrm{PHL}}}}
\newcommand{\EPMP}{\ensuremath{\mathbb{E}^{\mathrm{PMP}}}}
%\newcommand{\PQWM}{\ensuremath{\mathbb{PQWM}}}
%\newcommand{\PQWP}{\ensuremath{\mathbb{PQWP}}}
%\newcommand{\EQWM}{\ensuremath{\mathbb{E}^{\mathrm{QWM}}}}
%\newcommand{\EQWP}{\ensuremath{\mathbb{E}^{\mathrm{QWP}}}}
%\newcommand{\PWM}{\ensuremath{\mathbb{PWM}}}
%\newcommand{\PWP}{\ensuremath{\mathbb{PWP}}}
%\newcommand{\EWM}{\ensuremath{\mathbb{E}^{\mathrm{WM}}}}
%\newcommand{\EWP}{\ensuremath{\mathbb{E}^{\mathrm{WP}}}}
%\newcommand{\PMPasc}{\ensuremath{\mathbb{PMP}_{\mathrm{asc}}}}

%%%%%THEOREMS%%%%%%%%%%%%%%%%%%%%%%%%%%%%%%%%%%%%%%%%%%%%%%%%%%%%%%%%%%%%%%%%%%%%%%%%%%%%
\newtheorem{theorem}{Theorem}[section]
\newtheorem*{theorem*}{Theorem}
\newtheorem{theoremintro}{Theorem}
\renewcommand*{\thetheoremintro}{\Alph{theoremintro}}
\newtheorem{conj}[theorem]{Conjecture}
\newtheorem{lemma}[theorem]{Lemma}
\newtheorem{proposition}[theorem]{Proposition}
\newtheorem{corollary}[theorem]{Corollary}
\newtheorem{claim}[theorem]{Claim}
\newtheorem{prop}{Proposition}

\theoremstyle{definition}
\newtheorem{remark}[theorem]{Remark}
\newtheorem{example}[theorem]{Example}
\newtheorem{definition}[theorem]{Definition}
\newtheorem{definitions}[theorem]{Definitions}

\def\note#1{\marginpar{\raggedright\footnotesize #1}}
\def\change#1{{\color{green}\note{change}#1}}
\def\note#1{\textup{\textsf{\color{blue}((#1))}}}

%%%%%%%%TIKZ MACROS%%%%%%%%%%%%%%%%%%%%%%%%%%%%%%%%%%%%%%%%%%%%%%%%%%%%%%%%%%%%%%%%%%%%%%%%%
\usetikzlibrary{shapes.multipart}
\usetikzlibrary{patterns}
\usetikzlibrary{shapes.multipart}
\usetikzlibrary{arrows}
\usetikzlibrary{decorations.markings}
\usepgflibrary{decorations.shapes}
\usetikzlibrary{decorations.shapes}
\usepgflibrary{shapes.symbols}
\usetikzlibrary{shapes.symbols}
\usetikzlibrary{decorations.pathreplacing}

\tikzstyle{fleche}=[>=stealth', postaction={decorate}, thick]
\tikzstyle{axis}=[->, >=stealth', thick, gray]
\tikzstyle{path}=[->, >=stealth', thick]
\tikzstyle{grille}=[dotted, gray]
\newcommand{\pathrr}{\raisebox{-6pt}{\begin{tikzpicture}[scale=0.3]
		\draw[thick] (-1,0) -- (1,0);
		\draw[dotted] (0,-1) -- (0,1);
		\end{tikzpicture}}}
\newcommand{\pathru}{\raisebox{-6pt}{\begin{tikzpicture}[scale=0.3]
		\draw[thick] (-1,0) -- (0,0) -- (0,1);
		\draw[dotted] (0,-1) -- (0,0) -- (1,0);
		\end{tikzpicture}}}
\newcommand{\pathuu}{\raisebox{-6pt}{\begin{tikzpicture}[scale=0.3]
		\draw[dotted] (-1,0) -- (1,0);
		\draw[thick] (0,-1) -- (0,1);
		\end{tikzpicture}}}
\newcommand{\pathur}{\raisebox{-6pt}{\begin{tikzpicture}[scale=0.3]
		\draw[dotted] (-1,0) -- (0,0) -- (0,1);
		\draw[thick] (0,-1) -- (0,0) -- (1,0);
		\end{tikzpicture}}}
\newcommand{\pathbrr}{\raisebox{2pt}{\begin{tikzpicture}[scale=0.4]
		\draw[path] (-1,0) -- (0,0);
		\draw[dotted] (0,0) -- (0,1);
		\end{tikzpicture}}}
\newcommand{\pathbuu}{\raisebox{2pt}{\begin{tikzpicture}[scale=0.4]
		\draw[path] (0,0) -- (0,0.1);
		\draw[thick] (0,0) -- (0,1);
		\draw[dotted] (-1,0) -- (0,0);
		\end{tikzpicture}}}
\newcommand{\pathbru}{\raisebox{2pt}{\begin{tikzpicture}[scale=0.4]
		\draw[thick] (-1,0) -- (0,0);
		\draw[thick] (0,0) -- (0,1);
		\end{tikzpicture}}}
\newcommand{\pathbur}{\raisebox{2pt}{\begin{tikzpicture}[scale=0.4]
		\draw[dotted] (0,0) -- (0,1);
		\draw[dotted] (-1,0) -- (0,0);
		\end{tikzpicture}}}

%%%%%%%%%%%%PREAMBLE%%%%%%%%%%%%%%%%%%%%%%%%%%%%%%%%%%%%%%%%%%%%%%%%%%%%%%%%%%%%%%%%%%%%%%%%%%%%%
\title{Stochastic six-vertex model in a half-quadrant and half-line open ASEP}

\author[G. Barraquand]{Guillaume Barraquand}
\address{G. Barraquand,
	Columbia University,
	Department of Mathematics,
	2990 Broadway,
	New York, NY 10027, USA.}
\email{barraquand@math.columbia.edu}
\author[A. Borodin]{Alexei Borodin}
\address{A. Borodin, Department of Mathematics, MIT, Cambridge, USA, and
	Institute for Information Transmission Problems, Moscow, Russia.}
\email{borodin@math.mit.edu}
\author[I. Corwin]{Ivan Corwin}
\address{I. Corwin, Columbia University,
	Department of Mathematics,
	2990 Broadway,
	New York, NY 10027, USA.}
\email{ivan.corwin@gmail.com}
\author[M. Wheeler]{Michael Wheeler}
\address{M. Wheeler, School of Mathematics and Statistics, University of Melbourne, Parkville,
	Victoria 3010, Australia.}
\email{wheelerm@unimelb.edu.au}

\begin{abstract}
We consider the asymmetric simple exclusion process (ASEP) on the positive integers with an open boundary condition. We show that, when starting devoid of particles and for a certain boundary condition, the height function at the origin fluctuates asymptotically (in large time $\tau$) according to the Tracy-Widom GOE distribution on the $\tau^{1/3}$ scale. This is the first example of KPZ asymptotics for a half-space system outside the class of free-fermionic/determinantal/Pfaffian models.

Our main tool in this analysis is a new class of probability measures on Young diagrams that we call half-space  Macdonald processes, as well as two surprising relations. The first relates a special (Hall-Littlewood) case of these measures to the half-space stochastic six-vertex model (which further limits to ASEP) using a Yang-Baxter graphical argument. The second relates certain averages under these measures to their half-space (or Pfaffian) Schur process analogs via a refined Littlewood identity.
\end{abstract}

\maketitle

\setcounter{tocdepth}{1}

\tableofcontents

\numberwithin{equation}{section}

%%%%%%%%%%%%%%%%%%%%%%%%%%%%%%%%%%%%%%%%%%%%%%%%%%%%%%%%%%%%%%%%%%%%%%%%%%%%%%%%%%%%%%%%%%%%%%%

\section{Introduction}
The large scale statistics of random complex systems are often qualitatively independent from much of the microscopic details of the system, so that the laws of large universality classes can be probed via exactly solvable examples. In particular, particle systems in one spatial dimension  -- modeling, for instance, phenomena in non-equilibrium transport, traffic jams, and interface growth models -- are believed to lie, under mild hypotheses, in the Kardar-Parisi-Zhang (KPZ) universality class \cite{kardar1986dynamic, corwin2012kardar, halpin2015kpz, spohn2016kardar}. The large scale statistics of such one-dimensional particle systems have been extensively studied in infinite volume. The case of particle systems connected to boundary reservoirs is physically relevant \cite{spohn1983long}, yet less mathematically tractable.

In this paper, we study the asymmetric simple exclusion process (ASEP) on the positive integers with an open boundary at the origin in contact with a reservoir of particles kept at a constant density. It is expected (for instance due to known results for TASEP, a degeneration of ASEP) that a phase transition happens depending on the local density imposed by the boundary reservoir at the origin, between a maximal current phase and a low density phase. The critical case happens when the boundary imposes an average density $1/2$ of particles at the origin. In this paper, we study the statistics of the number of particles in the system when started empty, and with boundary conditions tuned to this critical point. We prove that after a very long time $\tau$, the random variable scales (around its law of large numbers centering) like $\tau^{1/3}$ and converges in this scale weakly to the GOE Tracy-Widom distribution (Theorem \ref{th:GOElimitintro}).
Further, our results also shed light on the distribution of the solution to the KPZ equation with Neumann boundary condition (Theorem \ref{th:KPZintro}), which arises as a limit of the height function of weakly asymmetric half-line ASEP around this critical point \cite{corwin2016open, parekh2017kpz}.

This is the first proof of (KPZ / random-matrix-theoretic) asymptotics in a non free-fermionic half-space model. Free-fermionic full-space systems have been well-studied via robust mathematical approaches, in particular the Schur processes \cite{okounkov2003correlation}. These are determinantal systems, meaning that correlation functions are written as determinants of a large matrix. The half-space analog of such systems are Pfaffian Schur processes \cite{rains2000correlation, baik2001algebraic, borodin2005eynard, sasamoto2004fluctuations}, whose correlation functions are given via Pfaffians. The full and half-space TASEP (where jumps only go in one direction) and a small handful of other models fit into the free-fermionic framework. 

ASEP and many other important models do not fit into the free-fermionic framework. 
In the last decade, starting from the work of Tracy and Widom on ASEP (on the full line) \cite{tracy2009asymptotics}, many KPZ type limit theorems have been obtained for non-free fermionic models in a full-space. These results have helped refine and expand the notion of KPZ universality. Some attempts have been made to study similar half-space systems, but until now no method has yielded rigorous distributional asymptotics without a Pfaffian structure. Among the existing works on non free-fermionic half-space systems, \cite{o2014geometric} studied the Log-Gamma directed polymer in a half-quadrant using properties of the geometric RSK algorithm on symmetric matrices, and conjectured integral formulas, but these are presently not amenable for asymptotic analysis. Using coordinate Bethe ansatz, \cite{tracy2013asymmetric} derived integral formulas for the transition probabilities in half-line ASEP for certain specific boundary conditions, but these formulas are not amenable to asymptotic analysis either.

Inspired by recent developments relating the integrability of ASEP on the full line to the stochastic six-vertex model \cite{borodin2014spectral, borodin2016stochasticsix, borodin2016asep, aggarwal2016phase}, we study half-line ASEP as a scaling limit of a stochastic six-vertex model in a half-quadrant with a boundary condition corresponding to off-diagonally symmetric alternating sign matrices considered in \cite{kuperberg2002symmetry}. Our analysis of the half-quadrant stochastic six-vertex model relies upon two surprising relations.

The first relation is between the half-quadrant stochastic six-vertex model and a family of measures on sequences of partitions that we call the half-space Hall-Littlewood processes.  These measures (see Definition \ref{def:Macdonaldmeasure}) are half-space variants of Macdonald processes (introduced in \cite{borodin2014macdonald} for the full space case) that generalize Pfaffian Schur processes by replacing Schur functions by Macdonald symmetric functions which rely on two parameters $q,t$. Using a graphical interpretation of the Yang-Baxter and reflection equations, we show that the height function in the half-quadrant stochastic six-vertex model has the same law as an observable of the half-space Hall-Littlewood process, i.e. the degeneration of half-space Macdonald process for $q=0$ (Theorem \ref{th:matchingHL6V}).

The second relation is between certain expectations of observables under the half-space Hall-Littlewood and Schur processes. 
Extracting statistical information from half-space  Macdonald processes is, in general, a difficult task -- see \cite{barraquand2018half} for an approach using Macdonald operators in the spirit of \cite{borodin2014macdonald}. In this paper, an important technical tool that will considerably simplify our analysis is a refined Littlewood summation identity (Proposition \ref{prop:refinedLittlewood}) for Macdonald symmetric polynomials, which was conjectured in \cite{betea2015refined} and proved in \cite{rains2014multivariate}. This allows us to relate certain observables of Macdonald measures for different values of $q$ and consequently connect the Pfaffian Schur process (case $q=t$) and the half-space Hall-Littlewood process (case $q=0$). The outcome is finally an identity between the $t$-deformed Laplace transform of the current in half-line ASEP and a multiplicative functional of a Pfaffian point process with an explicit correlation kernel (Proposition \ref{prop:Fredholmcurrent}), which can be analyzed (using Pfaffian point process methods) asymptotically in several interesting limit regimes. 

In statistical mechanics, a lot of effort went into obtaining determinant representations for partition functions or correlation functions for the XXZ spin chain/six-vertex model, despite the fact that the model is not free-fermionic (see e.g. the review  \cite{kitanine2009algebraic}). Our approach achieves this general goal for half-line ASEP and the half-space stochastic six-vertex model: we uncover a ``hidden''  fermionic structure and hence compute observables as Pfaffians.

\subsection*{ASEP with an open boundary}
The asymmetric simple exclusion process (ASEP) on a half-line with an open boundary at the origin is an interacting particle system on $\Z_{>0}$ where each site is occupied by at most one particle. Formally, this is a continuous time Markov process on the state of particle configurations (see Definition \ref{def:halflineASEP}). Each particle jumps by one to the right at rate $\p$ and to the left at rate $\q$, with $\q<\p$, provided the target site is empty. At the origin, we have a reservoir of particles that injects a particle at site $1$ (whenever it is empty) at rate $\ratealpha$ and removes a particle from site $1$ (whenever it is occupied) at rate $\rategamma$ -- see Figure \ref{fig:generalASEP}.
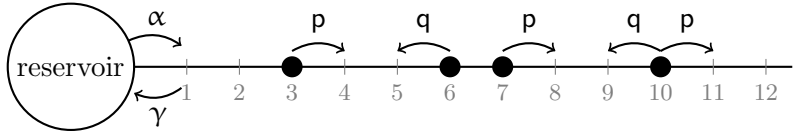
\begin{figure}
\begin{tikzpicture}[scale=0.7]
\draw[thick] (-1.2, 0) circle(1.2);
\draw (-1.2,0) node{reservoir};
\draw[thick] (0, 0) -- (12.5, 0);
\foreach \x in {1, ..., 12} {
	\draw[gray] (\x, 0.15) -- (\x, -0.15) node[anchor=north]{\footnotesize $\x$};
}

\fill[thick] (3, 0) circle(0.2);
\fill[thick] (6, 0) circle(0.2);
\fill[thick] (7, 0) circle(0.2);
\fill[thick] (10, 0) circle(0.2);
\draw[thick, ->] (3, 0.3)  to[bend left] node[midway, above]{$\p$} (4, 0.3);
\draw[thick, ->] (6, 0.3)  to[bend right] node[midway, above]{$\q$} (5, 0.3);
\draw[thick, ->] (7, 0.3) to[bend left] node[midway, above]{$\p$} (8, 0.3);
\draw[thick, ->] (10, 0.3) to[bend left] node[midway, above]{$\p$} (11, 0.3);
\draw[thick, ->] (10, 0.3) to[bend right] node[midway, above]{$\q$} (9, 0.3);
\draw[thick, ->] (-0.1, 0.5) to[bend left] node[midway, above]{$\ratealpha$} (0.9, 0.4);
\draw[thick, <-] (0, -0.5) to[bend right] node[midway, below]{$\rategamma$} (0.9, -0.4);
\end{tikzpicture}
\caption{Jump rates in the half-line ASEP. In this paper we will study more precisely the case when $\p=1, \ \q=t$, $\ratealpha=1/2, \ \rategamma=t/2$. }
\label{fig:generalASEP}
\end{figure}

It was proved in \cite{liggett1975ergodic} that when
\begin{equation}
 \frac{\ratealpha}{\p} + \frac{\rategamma}{\q} =1,
 \label{eq:Liggettscondition}
\end{equation}
there exist stationary measures for this process. Moreover, assuming \eqref{eq:Liggettscondition},  there is a phase transition as $\varrho = \frac{\ratealpha}{\p}$ varies. This parameter $\varrho\in(0,1)$ corresponds to the density of particles that the reservoir imposes at site $1$. When $\varrho<1/2 $, the system admits stationary measures that are  product i.i.d. Bernoulli  with parameter $\varrho$.
When $\varrho \geqslant 1/2$, stationary measures are spatially correlated and more complicated: there is a rarefaction fan with a density of particles $\varrho$ near the origin and density $1/2$ at $+\infty$.

When $\q=0$, the process becomes the totally asymmetric simple exclusion process (TASEP), and this phase transition is much better understood.
It was shown \cite{baik2001asymptotics, baik2017pfaffian} (see also \cite{sasamoto2004fluctuations}) that starting from an empty configuration,  the total number of particles in the system $\NN(\tau)$ at time $\tau$ has Gaussian fluctuations on the $\tau^{1/2}$ scale when $\varrho<1/2$, but has Tracy-Widom GSE\footnote{G(U/O/S)E  stands for Gaussian (Unitary/Orthogonal/Symplectic) Ensemble, and Tracy-Widom distributions were introduced in \cite{tracy1994level, tracy1996orthogonal} as the limiting distributions of the fluctuations of the largest eigenvalues of these ensembles. See Definition \ref{def:GOEdistribution} in the GOE case.} fluctuations on the scale $\tau^{1/3}$ when $\varrho>1/2$ and has Tracy-Widom GOE fluctuations on the scale  $\tau^{1/3}$ in the critical case $\varrho=1/2$. We refer to \cite[Section 6.1]{baik2017pfaffian} for a heuristic explanation of this phase transition using last passage percolation.

Let us go back to the asymmetric case. We expect that modulo a rescaling of time by $\p-\q$, the total  number of particles in half-line ASEP undergoes the same phase transition with the same limiting statistics (based on the fact that the full line TASEP and ASEP asymptotics are identical modulo such time rescaling). Before stating our results, let us mention some of the progress made to uncover the integrability of the model. The stationary distributions in the open ASEP with one or two boundaries can be computed via the matrix product ansatz  \cite{derrida1993exact, grosskinsky2004phase}. This realization led to a number of results (generally in the physics literature) such as the derivation of hydrodynamic limit, understanding of phase diagrams, large deviation principles. There exists an abundant literature on the subject (see for instance  \cite{duhart2014semi} and references therein).
An alternative understanding of stationary measures for half-line ASEP was proposed in
\cite{sasamoto2012combinatorics} using staircase tableaux (see also \cite{uchiyama2004asymmetric, corteel2007tableaux, corteel2010staircase, corteel2011tableaux}). For half-line ASEP with a finite constant number of particles (i.e. with closed boundary conditions $\ratealpha=\rategamma=0$), \cite{tracy2013bose} were able to express transition probabilities  using coordinate Bethe ansatz. For half-line ASEP with general boundary condition, \cite{tracy2013asymmetric} derived integral formulas for the transition probabilities, combining the formulas in the closed boundary case with the analysis of the reservoir. These formulas are explicit only when $\alpha$ or $\gamma$ equal zero, and are presently not amenable for asymptotic analysis in any case.

Let us now state our main result. Without loss of generality\footnote{multiplying all jump parameters by a constant correspond to a division of time by the same multiplicative factor.}, we can assume that $\p=1$. The left jump rate $\q$ will be denoted $t$, as it will coincide  in our analysis with the deformation parameter $t$ of Macdonald symmetric functions. Hence we will denote  time rather by the letter $\tau$ or $T$. Our main result is a limit theorem about the fluctuations of the current in ASEP in the critical case.
\begin{theoremintro}[Theorem \ref{th:GOElimit}]
Under the assumption \eqref{eq:Liggettscondition} (existence of stationary measures) and for $\varrho=1/2$ (critical density), that is when the jump rates are given by
$$ \p=1, \ \q=t, \ \ratealpha=1/2, \ \rategamma=t/2,$$
we have for any $t\in [0,1)$   the weak convergence
$$ \frac{\frac{T}{4}-\NN\left(\frac{T}{1-t} \right)}{2^{-4/3}T^{1/3}} \xRightarrow[T\to +\infty]{} \mathcal{L}_{\rm GOE},  $$
where $\NN(\tau)$ denotes the total number of particles in half-line ASEP at time $\tau$ and $\mathcal{L}_{\rm GOE}$ is the Tracy-Widom GOE distribution (see Definition \ref{def:GOEdistribution}).
\label{th:GOElimitintro}
\end{theoremintro}
When $t=0$, ASEP becomes TASEP, and the result was proved as Theorem 1.3 in \cite{baik2017pfaffian} (see also \cite{baik2001asymptotics, sasamoto2004fluctuations} for very similar results in the context of last passage percolation in a half-space with with geometric weights).

For half-line TASEP, the current fluctuations are also known for $\varrho\neq1/2$ \cite{baik2017pfaffian} (equivalently $\alpha\neq 1/2$). Thus, we expect that current fluctuations of half-line open ASEP are Tracy-Widom GSE distributed when $\varrho>1/2$ and Gaussian when $\varrho<1/2$. Extending Theorem \ref{th:GOElimitintro} to other values of $\varrho$ does not seem to be immediately accessible from the techniques developed in this paper.
\subsection*{Stochastic six-vertex model in a half-quadrant} We will approach the half-line ASEP  through a scaling limit of another integrable model, the stochastic six-vertex model in a half-quadrant, which we believe is also interesting in its own right. On the whole line $\Z$, this approach was recently used for studying ASEP in \cite{aggarwal2016phase, aggarwal2018current, borodin2016asep, corwin2017transversal}. Our half-space model is closely related to off-diagonally symmetric alternating sign matrices, whose weighted enumeration was computed in \cite{kuperberg2002symmetry}.
\begin{figure}
\begin{center}
\begin{tikzpicture}[scale=0.8]
\foreach \x in {1, ..., 7} {
\draw[dotted] (0,\x) -- (\x,\x) -- (\x, 8);
\draw[path] (0,\x) --(0.1,\x);
\draw[gray] (\x, \x-1) node{$\x$};
\draw[gray] (-1,\x) node{$\x$};
}
\draw[path] (0,1) --(1,1);
\draw[path] (0,2) --(1,2) -- (1,3) -- (2,3) -- (2,4) -- (4,4);
\draw[path] (0,3) --(1,3) -- (1,4) -- (2,4) -- (2,5) -- (3,5) -- (3,6) -- (4,6) -- (5,6) -- (5,7)--(7,7) ;
\draw[path] (2,2) -- (2,3) -- (3,3);
\draw[path] (5,5) -- (5,6) -- (6,6);
\draw[path] (0,4) --(1,4) -- (1,5) -- (2,5) -- (2,6) -- (3,6) -- (3,7) -- (5,7) -- (5,8);
\draw[path] (0,5) --(1,5) -- (1,6) -- (2,6) -- (2,7) -- (3,7) -- (3,8);
\draw[path] (0,6) --(1,6) -- (1,7) -- (2,7) -- (2,8) ;
\draw[path] (0,7) --(1,7) -- (1,8)  ;
\draw[path] (2,2)--(2,2.1);
\draw[path] (5,5)--(5,5.1);
\end{tikzpicture}
\end{center}
\caption{Sample configuration of the stochastic six-vertex model in a half-quadrant.}
\label{fig:example6v}
\end{figure}
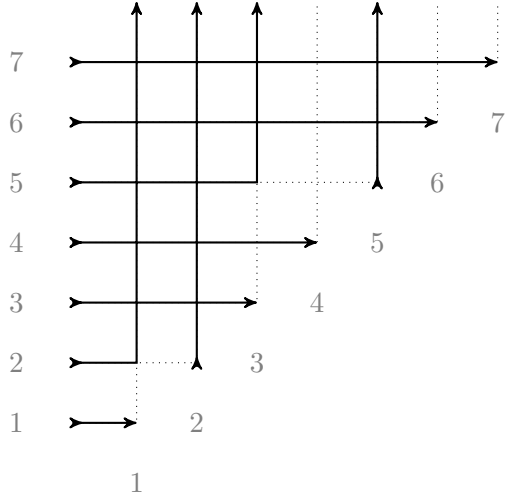
Consider the half-quadrant $ \lbrace (x,y)\in \Z_{>0}^2: x\leqslant y \rbrace $.
The stochastic six-vertex model in the half-quadrant is a probability measure on collections of up-right paths (see Figure \ref{fig:example6v})
 determined by
the Boltzmann weights  (see Section \ref{sec:def6v} for a more precise definition)
$$ \PP\left(\pathrr\right) = \frac{1-a_xa_y}{1-ta_xa_y}, \ \ \PP\left(\pathru\right) = \frac{(1-t)a_xa_y}{1-ta_xa_y},\ \ \PP\left(\pathuu\right) = \frac{t(1-a_xa_y)}{1-ta_xa_y}, \ \ \PP\left(\pathur\right) =\frac{1-t}{1-ta_xa_y}, $$
and boundary condition
\begin{equation}
 \PP\left(\pathbrr\hspace{0.2cm}\right) =  \PP\left(\pathbuu\hspace{0.2cm}\right) = 1,\ \ \ \PP\left(\pathbru\hspace{0.2cm}\right) =  \PP\left(\pathbur\hspace{0.2cm}\right) =0.
 \label{eq:boundaryconditionintro}
\end{equation}
Since these weights are stochastic, in the sense that
$$ \PP\left(\pathrr\right) + \PP\left(\pathru\right) = \PP\left(\pathur\right)+\PP\left(\pathuu\right)=1, $$
the measure on paths can be constructed in a Markovian way starting from the left boundary of the quadrant.
%We define $\mathfrak{h}(x,y)$ as the number of paths that cross one of the vertices $(i,y)$ for $1\leqslant i\leqslant x$.
When the parameters $a_x$ go to $1$, the paths will turn at almost every vertex point, and hence follow a straight staircase path in the $\pi/4$ direction. If one scales those parameters as $a_x\equiv 1-\e$ and rescale time by $\e^{-1}$, one can interpret the horizontal positions of paths in a finite neighborhood of the diagonal as a particle system that converges to half-line  ASEP as $\e$ goes to zero (see Proposition \ref{prop:cv6vtoASEP}). The assumption that $\varrho=1/2$ in Theorem \eqref{th:GOElimitintro} comes from the specific choice of boundary condition \eqref{eq:boundaryconditionintro}. More general boundary conditions do not seem to be related to the half-space Hall-Littlewood measure considered in this paper.  It is likely that a result similar to Theorem \ref{th:GOElimitintro} can also be proved for the six-vertex model in a half-quadrant by our methods (see Remark \ref{rem:asymptotics6v}), but we do not pursue that here.

\subsection*{KPZ equation on the positive reals}
The KPZ equation on $\R_{\geqslant 0}$ with Neumann boundary condition is the (ill-posed) stochastic partial differential equation
\begin{equation}
\begin{cases}
\partial_{\tau} \mathscr H = \frac{1}{2}\Delta  \mathscr H + \frac{1}{2}\big(\partial_x \mathscr H\big)^2 + \dot W,\\
\partial_x \mathscr H (\tau, x)\big\vert_{x=0} = A \qquad (\forall \tau > 0),
\end{cases}
\label{eq:KPZequation}
\end{equation}
where $W$ is a space-time white noise. We say that $\mathscr H$ solves this equation in the Cole-Hopf sense with narrow-wedge initial condition when $\mathscr H= \log \mathscr Z$ and $\mathscr Z$ is a mild solution (see Definition \ref{def.mild}) to the multiplicative stochastic heat equation with Robin boundary condition
\begin{equation} \label{eq:mSHEintro}
\begin{cases}
\partial_{\tau} \mathscr Z = \tfrac12 \Delta \mathscr Z + \mathscr Z \dot W\;,\\
\partial_x \mathscr Z (\tau, x)\big\vert_{x=0} = A  \ \mathscr Z(\tau, 0)  \qquad (\forall \tau > 0),
\end{cases}
\end{equation}
with delta initial condition. Based on the convergence of weakly asymmetric ASEP to the KPZ equation from \cite{corwin2016open}, we expect that a certain limit of half-line ASEP height function (denoted $\mathcal{H}(\tau)$ in the next Theorem) weakly converges as $t$ goes to $1$ and has the same
distribution as $\mathscr{H}(\tau, 0)$ where $\mathscr H$ solves \eqref{eq:KPZequation} with $A=-1/2$.  We explain however in Section \ref{sec:KPZ} that the results of \cite{corwin2016open}  do not directly apply to the boundary and initial conditions that we are considering (\cite{corwin2016open} assumes $A\geqslant 0$ and H\"older continuous initial conditions). After the posting of the first version of this paper, \cite[Theorem 1.2]{parekh2017kpz} extended the convergence result from \cite{corwin2016open} and showed that the random variable denoted $\mathcal{H}(\tau)$ in the next theorem has the same
distribution as $\mathscr{H}(\tau, 0)$ where $\mathscr H$ solves \eqref{eq:KPZequation} with $A=-1/2$.
\begin{theoremintro}[Theorem \ref{th.LaplaceKPZ} and Corollary \ref{prop:multiplicativefunctional}]
 Under the scalings
$$ t= e^{-\epsilon}, \ \ \tau=\frac{\e^{-3}\tilde{\tau}}{1-t} \approx \e^{-4} \tilde{\tau}, \ \  $$
the random variable
\begin{equation*}
\mathcal{U}_{\e}(\tilde \tau)  = \frac{ t^{\big( \NN(\tau)- \e^{-3}\tilde\tau/4\big) }}{1-t^2}
\end{equation*}
weakly converges as $\e\to 0$ to a positive random variable $\mathcal{U}(\tilde \tau)$. Furthermore, if
$$ \mathcal{H}(\tau) := \log\big( 4\ \mathcal{U}(8\tau) \big) -\frac{\tau}{24},$$
then for any $z>0$,
$$\EE\left[ \exp\left( \frac{-z}{4} \exp\left(\frac{\tau}{24} + \mathcal H (\tau)\right)\right) \right]   = \EE\left[ \prod_{i=1}^{+\infty} \frac{1}{\sqrt{1+z \exp\left((\tau/2)^{1/3}  \mathfrak{a}_i\right)}} \right],$$
where  $\lbrace \mathfrak{a}_i\rbrace_{i=1}^{\infty} $ forms the GOE point process (i.e. the sequence of rescaled eigenvalues of a matrix from the GOE, see Definition \ref{def:GOEdistribution}).
\label{th:KPZintro}
\end{theoremintro}

\begin{remark}
One can deduce immediately from the above theorem that as $\tau$ goes to infinity,
$$ \lim_{\tau\to+\infty} \PP\left( \frac{\mathcal{H}(\tau) + \frac{\tau}{24}}{2^{-1/3}\tau^{1/3} } \leqslant x \right) = \PP\big( \mathfrak{a}_1 \leqslant x \big)  = F_{\rm GOE}(x).$$
This is the half-space analogue of Corollary 1.3 in \cite{amir2011probability} (see also \cite{sasamoto2010exact, calabrese2010free, dotsenko2010replica}) where a similar limit theorem was proved for the solution to the KPZ equation on $\R$ (the scaling there is exactly the same but the limit distribution is the Tracy-Widom GUE distribution instead of the GOE).
\end{remark}
\begin{remark}
The KPZ equation \eqref{eq:KPZequation} with boundary parameter $A$ is considered in the case $A=+\infty $ in the physics paper \cite{gueudre2012directed}, where  large-time Tracy-Widom $GSE$ asymptotics are obtained via a non-rigorous replica method. The paper \cite{borodin2016directed} studies the case $A=0$, though the results are also based on a non-rigorous replica method and some partially conjectural combinatorial simplifications.  It is not yet clear whether the conjectural results from the latter work are compatible with the assumption that ASEP and TASEP have the same fluctuations. Indeed, if the height function in  ASEP and TASEP satisfies the same limit theorem modulo a rescaling of time by the asymmetry, one would expect that when the parameters of ASEP are scaled so as to obtain $A=0$ in the KPZ equation limit, the large time fluctuations of the height function would be related to a crossover distribution between Tracy-Widom GOE and GSE distributions as in \cite{baik2001asymptotics, baik2017pfaffian}  (rather than the GSE Tracy-Widom distribution in \cite{borodin2016directed}, which would arise only when the density of particles enforced by the boundary near the origin is strictly greater than $1/2$).
\end{remark}

\subsection*{Outline of the paper}

In Section \ref{sec:Macdonald}, we define half-space Macdonald processes and explain how  a refined Littlewood summation identity for Macdonald symmetric polynomials (see Proposition \ref{prop:refinedLittlewood}) allows us to relate certain observables of half-space Macdonald measures for different values of $q$. More precisely, in Section \ref{sec:Fredholm} we express  certain observables of the half-space Hall-Littlewood process (case $q=0$) as Fredholm Pfaffians involving the correlation kernel of the Pfaffian Schur process (case $q=t$). In Section \ref{sec:6v}, we show that half-space Hall-Littlewood processes are related to the  stochastic six-vertex model in a half-quadrant, horizontal sections of the latter being marginals of the former. This type of connection between Macdonald measures and (higher-spin) vertex model goes back to \cite{borodin2016stochastic}, but our  proof is a half-space variant of the corresponding full-space result in \cite{borodin2016between}. The main ingredient is a $t$-boson representation of Hall-Littlewood polynomials, and graphical  interpretations of Yang-Baxter and reflection equations from \cite{wheeler2016refined}. Under a scaling limit, the height function of the stochastic six-vertex model converges to ASEP, so that we obtain in Section \ref{sec:ASEP} a Fredholm Pfaffian formula characterizing the distribution of the current in half-line ASEP. We  exploit this formula in two asymptotic regimes.
In Section \ref{sec:GOElimit}, we perform asymptotic analysis on these formulas to prove the convergence of the current in ASEP to the GOE Tracy-Widom distribution (Theorem \ref{th:GOElimit}). In Section \ref{sec:KPZ} we discuss the convergence of ASEP height function to the KPZ equation based on \cite{corwin2016open}, and prove that the height function of ASEP at the origin converges in the weak asymmetry regime to a multiplicative functional of the GOE point process (Theorem \ref{th.LaplaceKPZ} and Corollary \ref{prop:multiplicativefunctional}).

\subsection*{Ackowledgements}  A. B. was partially supported by the National Science Foundation grant
DMS-1607901 and by Fellowships of the Radcliffe Institute for Advanced Study and the Simons Foundation.  I.C. was partially supported by the NSF through DMS-1208998, the Clay Mathematics Institute through a Clay Research Fellowship, the Poincar\'e Institute through the Poincar\'e chair,  and the Packard Foundation through a Packard Fellowship for Science and Engineering. G.B. was partially supported by the Packard
Foundation through I.C.'s fellowship. M. W. was partially  supported by the Australian Research Council grant DE160100958.

\section{Half-space Macdonald Processes}
\label{sec:Macdonald}
A {\it partition} $\lambda$ is a non-increasing sequence of non-negative integers $\lambda_1\geqslant \lambda_2\geqslant \cdots$ only finitely many of which are nonzero. We will sometimes use the notation $\lambda=  1^{m_1}2^{m_2}\dots$ for a partition $\la$ where $m_j$ is the multiplicity of the integer $j$ in the sequence of $\la_i$'s.  We denote by $\Y$ the set of all partitions. The {\it length} of $\lambda$ is the number of non-zero elements and is denoted $\ell(\lambda)$. The transpose $\lambda'$ of a partition is defined by $\lambda'_i = \sharp\{j:\lambda_j\geqslant i\}$. In particular, $\lambda'_1=\ell(\la)$.
%The empty partition (or partition of all zeros) is denoted by $\varnothing$.
A partition can be identified with a {\it Young diagram}. For a box $\Box$ in a Young diagram, $\ell(\Box)$ is equal to the number of boxes in the diagram below it (the leg length) and $a(\Box)$ is equal to the number of boxes in the diagram to the right of it (the arm length) -- see Figure \ref{fig:diagram}. A partition is called  {\it even} if all $\lambda_i$ are even. We write $\mu\subseteq \lambda$ if $\mu_i\leqslant \lambda_i$ for all $i$ and call $\lambda/\mu$ a {\it skew Young diagram}. A partition $\mu$ {\it interlaces} with $\lambda$ (denoted by $\mu \prec \lambda$) if for all $i$, $\lambda_i \geqslant \mu_i \geqslant \lambda_{i+1}$. In the language of Young diagrams, this means that $\lambda$ can be obtained from $\mu$ by adding a {\it horizontal strip} in which at most one box is added per column.

\begin{figure}
\begin{tikzpicture}[scale=0.55]
\foreach \y [count=\xi] in {5,3,1,1,0}
\foreach \x in {0, ..., \y}
\draw (\x, -\xi) -- ++(1,0) -- ++(0,1) -- ++(-1,0)--cycle;
\fill[black] (1, -1) -- ++(1,0) -- ++(0,1) -- ++(-1,0)--cycle;
\draw[<->, >=stealth'] (2, -0.5) -- (6, -0.5) node[midway, fill=white, fill opacity=0.4, text opacity=1, anchor=north]{arm};
\draw[<->, >=stealth'] (1.5, -1) -- (1.5, -4) node[midway, fill=white, fill opacity=0.4, text opacity=1, anchor=west]{leg};

\begin{scope}[xshift=7cm]
\foreach \y [count=\xi] in {4,3,1,1,0,0}
\foreach \x in {0, ..., \y}
\draw (\x, -\xi) -- ++(1,0) -- ++(0,1) -- ++(-1,0)--cycle;
\end{scope}
\end{tikzpicture}
\caption{The left-most diagram corresponds with partition $\lambda = (6,4,2,2,1)$.
The black box has arm length $ a(\blacksquare)=4 $ and leg length $\ell(\blacksquare)=3$. The next diagram is $\lambda$'s transpose $\lambda' = (5,4,2,2,1,1)$.}
\label{fig:diagram}
\end{figure}
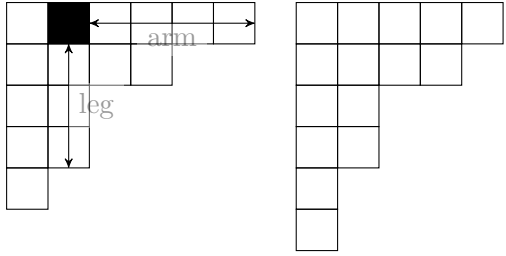

Let $\Sym$ be the ring of symmetric functions in countably many independent variables (see \cite[Chapter I, Section 2]{macdonald1995symmetric}).
The {\it skew Macdonald $P$ ($Q$) functions} $P_{\lambda/\mu}$ ($Q_{\lambda/\mu}$) (introduced in \cite[Chapter VI]{macdonald1995symmetric}) are symmetric functions indexed by skew partitions $\lambda/\mu$ and have coefficients that are rational functions of two auxiliary parameters $q,t$ which we will assume to be in $[0,1)$ throughout the paper. Macdonald symmetric functions become Hall-Littlewood symmetric functions when $q=0$, and Schur functions when $q=t$.

For $\lambda \in \Y$, define symmetric functions
\begin{equation}
\ve_{\lambda} = \sum_{\mu' \textrm{ even}} \bel{\mu} \ Q_{\lambda/\mu}
\end{equation}
where ``el'' stands for ``even leg'', and $\bel{\mu}\in \mathbb{Q}[q,t]$ is given by
\begin{equation}
\bel{\mu} = \prod_{\substack{\Box\in \mu\\\ell(\Box)\textrm{ even}}} b_\mu(\Box), \qquad b_\mu(\Box) = \begin{cases} \dfrac{1-q^a t^{\ell+1}}{1-q^{a+1}t^{\ell}}, &\Box\in \mu,\\ 1,&\Box\notin \mu, \end{cases}
\end{equation}
with $\ell=\ell(\Box)$ and $a=a(\Box)$ in the definition of $b_{\mu}(\Box)$. All summations over partitions will always be over the set $\Y$, sometimes subject to specified additional constraints, e.g. $\mu'$ even.
%It is sometimes useful in probabilistic applications to consider more general specializations of the ring of symmetric functions (see\cite{borodin2014macdonald}), but this is not necessary for the present paper.

We now recall or derive certain identities involving Macdonald symmetric functions which will be utilized in the remainder of the paper. The skew Cauchy identity \cite[Section VI.7]{macdonald1995symmetric} holds for two sets of formal variables $x$ and $y$:
\begin{equation}\label{eq:SkewCauchy}
\sum_{\kappa} P_{\kappa/\nu}(x)Q_{\kappa/\lambda}(y) =\Pi(x;y)\sum_{\tau}Q_{\nu/\tau}(y)P_{\lambda/\tau}(x)
\end{equation}
where $\Pi(x;y)$ is given by \cite[(2.5) in Ch. VI]{macdonald1995symmetric}
\begin{equation}\label{eq:PI}
\Pi(x;y) := \sum_{\kappa\in\Y} P_{\kappa}(x)Q_{\kappa}(y) = \prod_{i,j}  \phi(x_i y_j)  \qquad\text{where}\qquad
\phi(u) = \frac{(tu;q)_{\infty}}{(u;q)_{\infty}},
\end{equation}
and $(a;q)_n = (1-a)(1-qa)\cdots (1-q^{n-1}a)$ is the $q$-Pochhammer symbol (with infinite product form when $n=\infty$). Macdonald $P$ and $Q$ functions satisfy a branching rule whereby \cite[Section VI.7]{macdonald1995symmetric}
\begin{equation}\label{eq:branching}
\sum_{\mu} P_{\nu/\mu}(x) P_{\mu/\lambda}(y) = P_{\nu/\lambda}(x,y) \qquad\text{and} \qquad \sum_{\mu} Q_{\nu/\mu}(x) Q_{\mu/\lambda}(y) = Q_{\nu/\lambda}(x,y).
\end{equation}
They also satisfy a Littlewood identity \cite[Section VI.7, Ex. 4(i)]{macdonald1995symmetric} \begin{equation}\label{eq:Littleweood}
\sum_{\nu' \textrm{ even}} \bel{\nu}P_{\nu}(x) = \prod_{i<j} \phi(x_ix_j)  =:\Phi(x).
\end{equation}
Turning to the $\ve_{\lambda}$ function, it follows from the definition along with the branching rule (\ref{eq:branching}) that
\begin{equation}\label{eq:brancghingve}
\sum_{\mu} Q_{\lambda/\mu}(x) \ve_{\mu}(y) = \ve_{\lambda}(x,y).
\end{equation}
We also have a skew version of the Littlewood identity.
\begin{proposition}
For a set of formal variables $x$,
\begin{equation}
\sum_{\nu' \textrm{ even}} \bel{\nu} P_{\nu/\la}(x) = \Phi(x) \sum_{\mu' \textrm{ even}} \bel{\mu} Q_{\la/\mu}(x) = \Phi(x) \ve_{\la}(x).
\label{eq:skewLittlewood}
\end{equation}
\end{proposition}
\begin{proof}
Consider first the one variable case. $P_{\nu/\la}(x)$ and $Q_{\la/\mu}(x)$ are zero unless $ \mu\prec\la\prec\nu$. For a given $\la\in \Y$, there are unique $\mu$ and $\nu$ such that $\mu'$ and $\nu'$ are even and  $ \mu\prec\la\prec\nu$. It was proved in \cite[Eq. (4) page 350]{macdonald1995symmetric} by an explicit computation that
$$ \bel{\nu} P_{\nu/\la}(x) = \bel{\mu} Q_{\la/\mu}(x).$$
Since in the single variable case $\Phi(x)=1$, identity \eqref{eq:skewLittlewood} is established in that case. To deduce the result for more general $x=(x_1, \dots, x_n)$, denote $F_{\la}(x) = \sum_{\nu' \textrm{ even}} \bel{\nu} P_{\nu/\la}(x) $, so that
\begin{align*}
F_{\la}(x_1, \dots, x_n) &= \sum_{\nu' \textrm{ even}} \bel{\nu} P_{\nu/\la}(x_1, \dots, x_n)\\
&= \sum_{\nu' \textrm{ even}} \bel{\nu} \sum_{\tau} P_{\nu/\tau}(x_n)P_{\tau/\la}(x_1, \dots, x_{n-1})\\
&= \sum_{\tau} \sum_{\mu' \textrm{ even}} \bel{\mu} Q_{\tau/\mu}(x_n)P_{\tau/\la}(x_1, \dots, x_{n-1})\\
&= \Pi(x_1, \dots, x_{n-1};x_1) \sum_{\mu' \textrm{ even}} \bel{\mu}\sum_{\tau}  Q_{\la/\tau}(x_n)P_{\mu/\tau}(x_1, \dots, x_{n-1})\\
&= \Pi(x_1, \dots, x_{n-1};x_1)\sum_{\tau}  Q_{\la/\tau}(x_n)F_{\tau}(x_1, \dots, x_{n-1}),
\end{align*}
where in the second equality we have used the branching rule \eqref{eq:branching}, in the third equality we have used the single variable case established above, and in the fourth  equality we have used the skew-Cauchy identity \eqref{eq:SkewCauchy}. We may now iterate the relation
$$ F_{\la}(x_1, \dots, x_n)  =  \Pi(x_1, \dots, x_{n-1};x_1)\sum_{\tau}  Q_{\la/\tau}(x_n)F_{\tau}(x_1, \dots, x_{n-1}), $$
use the branching rule \eqref{eq:branching}, and deduce that
$$ F_{\la}(x) = \prod_{i<j} \phi(x_ix_j)  \sum_{\mu' \textrm{ even}} \bel{\mu} Q_{\la/\mu}(x),$$
as desired.
\end{proof}
Now combining the skew Cauchy and skew Littlewood identities, we obtain
\begin{equation}
\sum_{\mu } \ve_{\mu}(x) P_{\mu/\lambda}(y) = \Pi(x,y)\Phi(y)\, \ve_{\lambda}(x,y),
\label{eq:generalsummation}
\end{equation}
and in particular
\begin{equation}
\sum_{\mu } \ve_{\mu}(x) P_{\mu}(y) = \Pi(x,y)\Phi(y).
\label{eq:pfaffianmacdonaldsummation}
\end{equation}

A {\it specialization} $\rho$ of $\Sym$ is an algebra homomorphism of $\Sym\to \C$. For instance,  evaluating symmetric functions at a fixed finite set of variables defines such a homomorphism -- see \cite[Section 2.2.1]{borodin2014macdonald} for a more detailed discussion. We denote the application of $\rho$ to $f\in \Sym$ as $f(\rho)$, thus extending the usual notation for the evaluation at a set of variables. The {\it trivial} specialization $\rho=\varnothing$ takes the value 1 for the constant function $1\in \Sym$ and $0$ for all homogeneous functions $f\in\Sym$ of higher degree. The {\it union} of two specializations $\rho_1,\rho_2$ is defined via the relation
$$
p_k(\rho_1,\rho_2) = p_{k}(\rho_1)+p_{k}(\rho_2)
$$
and extended to all of $\Sym$ by linearity. Here $p_k(x) = x_1^k+x_2^k+\cdots$ are the Newton power sum symmetric functions. Notationally, we may write the union of $\rho_1,\rho_2$ by putting a comma between them. We say a specialization $\rho$ is {\it Macdonald nonnegative} if for every skew diagram $\lambda/\mu$, $P_{\lambda/\mu}(\rho)\geqslant 0$.

\begin{definition} For two Macdonald-nonnegative specialization $\rho^+, \rho^-$, we define the \emph{half-space Macdonald measure} as a probability measure $\PMM$ on Young diagrams $\la\in \Y$ such that
$$\PMM(\la) = \frac{1}{\Pi(\rho^+; \rho^-)\Phi(\rho^+)}\,P_{\la}(\rho^+)\ve_{\la}(\rho^-).$$
We will denote by $\EPMM$ the corresponding expectation. This is a well-defined probability measure thanks to \eqref{eq:pfaffianmacdonaldsummation}, provided the series converges. In this paper, we  deal with Macdonald-nonnegative specializations corresponding to evaluating functions into finitely many symmetric variables in $(0,1)$, so that  the sums \eqref{eq:SkewCauchy} and  \eqref{eq:skewLittlewood} always converge.
\label{def:Macdonaldmeasure}
\end{definition}

\begin{definition} For Macdonald nonnegative specializations $\rho^+_1, \dots, \rho^+_{n}$ and $\rho^-_1, \dots, \rho^-_{n}$, we define the \emph{half-space Macdonald process} as a probability measure $\PMP$ on sequences of Young diagrams
$$  \varnothing \subset \la^{(1)}\supset \mu^{(1)} \subset \la^{(2)}\supset \mu^{(2)} \subset \dots \mu^{(n-1)} \subset \la^{(n)} \supset \varnothing,$$
such that
\begin{equation}
\PMM\big( \la, \mu \big) = \frac{1}{Z(\rho)} P_{\la^{(1)}}(\rho^+_1)Q_{\la^{(1)}/\mu^{(1)}}(\rho^-_{1})  P_{\la^{(2)}/\mu^{(1)}}(\rho^+_2) Q_{\la^{(2)}/\mu^{(2)}}(\rho^-_{2})  \dots P_{\la^{(n)}/\mu^{(n-1)}}(\rho^+_{n}) \ve_{\la^{(n)}}(\rho^-_n),
\label{eq:defPMP}
\end{equation}
where
$$ Z(\rho) =  \Phi(\rho^+) \prod_{i\leqslant j}\Pi(\rho^+_i; \rho^-_j), $$
and $\rho^+ = (\rho^+_{1}, \dots, \rho^+_N)$.
We will denote by $\EPMP$ the corresponding expectation. One may check that this indeed defines a probability measure  by repeated application of the skew Cauchy and Littlewood identities \eqref{eq:SkewCauchy} and  \eqref{eq:skewLittlewood}, provided all series converge.
\label{def:Macdonaldprocesses}
\end{definition}

\begin{proposition}
Under the notations of Definition \ref{def:Macdonaldmeasure} and \ref{def:Macdonaldprocesses}, the marginal distribution of $\la^{(k)}$ under the half-space Macdonald process is the half-space  Macdonald measure with specializations
$$ \rho^+= (\rho^+_1, \dots, \rho^+_{k}), \ \ \rho^- = (\rho^-_{k}, \dots, \rho^-_{n}).$$
\end{proposition}
\begin{proof}
The result follows from  summing the probabilities  \eqref{eq:defPMP} over all $\mu^{(i)}$ and $\la^{(j)}$ for $j\neq k$, using  the branching rule \eqref{eq:branching}, the skew Cauchy identity  \eqref{eq:SkewCauchy} and the summation formula  \eqref{eq:generalsummation}.
\end{proof}

%\begin{definition}
%For two sets of parameters $a_1, \dots, a_n$ and $b_1, \dots, b_m$ in $(0,1)$, we define the Pfaffian Macdonald measure as a probability measure $\PMM$ on Young diagrams $\la\in \Y$ such that
%$$\PMM(\la) = \frac{1}{\Pi(a; b)\Phi(a)} \ve_{\la}(b)P_{\la}(a).$$
%We will denote $\EPMM$ the corresponding expectation. This probability measure is well-defined thanks to \eqref{eq:pfaffianmacdonaldsummation} and the fact that Madonald polynomials evaluated into nonnegative variables are nonnegative.
%\end{definition}

\begin{remark}
Half-space Macdonald measures/processes naturally generalize Pfaffian Schur measures/processes  introduced in \cite{borodin2005eynard, baik2001algebraic}, the latter corresponding to the degeneration when $q=t$. The term Pfaffian comes from the fact that the Pfaffian Schur process defines a  Pfaffian point process (Theorem 3.3 in \cite{borodin2005eynard}). However, half-space Macdonald processes do not correspond to any Pfaffian point process.
\end{remark}

\begin{definition}
We define the \emph{ascending half-space Macdonald process} as the half-space Macdonald process defined by specializations\footnote{we could  allow $\rho_n^-$ to be arbitrary to study  more general  boundary conditions.}  $\rho^-_1 = \dots = \rho^-_n=\varnothing$ and $\rho^+_i=(a_i)$ (specialization into a single variable $a_i$), where the variables $(a_1, \dots, a_n)$ are in $(0,1)$. This process is supported on sequences of interlaced partitions
$$ \la^{(1)}  \prec   \la^{(2)} \prec  \dots \prec  \la^{(n)}.$$
\label{def:ascendingprocess}
\end{definition}

In the rest of this section, we focus on the marginal of $\la^{(n)}$ for the ascending half-space  Macdonald process, i.e. the half-space Macdonald measure with specializations
$$ \rho^+=(a_1, \dots, a_n), \ \ \rho^-=\varnothing.$$

We will need one more identity, first conjectured in \cite{betea2015refined} and proved in \cite{rains2014multivariate}.
\begin{proposition}[\cite{rains2014multivariate}, Proposition 6.26] For an even integer $n$ and any $u\in \C$, Macdonald symmetric polynomials satisfy
\begin{equation}
%\sum_{\la'\textrm{ even}}   \prod_{i \textrm{ even}} \left( 1-uq^{\lambda_i}t^{n-i} \right)  \ \bel{\la} P_{\la}(x_1, \dots, x_n) = \Phi(x) \prod_{i<j}\frac{1-x_ix_j}{x_i-x_j} \Pf\left[ \frac{x_i-x_j}{1-x_ix_j} - u\frac{x_i-x_j}{1-tx_ix_j}\right].
\frac{1}{\Phi(x)}\sum_{\la'\textrm{ even}}  \ \  \prod_{i \textrm{ even}}\left(1-uq^{\lambda_i}t^{n-i} \right)\ \bel{\la}P_{\la}(x_1, \dots, x_n) =  \dfrac{\Pf\left[ \dfrac{x_i-x_j}{1-x_ix_j} - u\dfrac{x_i-x_j}{1-tx_ix_j}\right]}{\Pf\left[ \dfrac{x_i-x_j}{1-x_ix_j}\right]}.
\label{eq:refinedLittlewood}
\end{equation}
\label{prop:refinedLittlewood}
\end{proposition}
The most striking fact in the above identity \eqref{eq:refinedLittlewood} is that  the right-hand-side does not depend on $q$. This yields identities  that relate half-space Macdonald measures for different values of $q$. Consider a set of parameters $a_1, \dots, a_n \in (0,1)$ for $n$ even, and let us  denote by $\PSM$ and $\EPSM$ the half-space Macdonald measure and expectation associated with parameters $a_i$
%and $b_i\equiv 0$
in the Schur case $q=t$. It is important to note that $\PSM$ and $\EPSM$ do not depend on the parameter $q=t$. In this context \eqref{eq:refinedLittlewood} implies that for any $q,t\in[0,1)$ and $u\in \C$
\begin{equation}
\EPMM\left[ \prod_{i \textrm{ even}}\left( 1-uq^{\lambda_i}t^{n-i}\right)\right] = \EPSM\left[ \prod_{i\textrm{ even}} \left(1-ut^{\la_i+n-i}\right)\right].
%=  \EPHL\left[ \prod_{i\ \mathrm{ even}, \la_i=0} 1-ut^{n-i}\right]
\label{eq:relationMMSchurbis}
\end{equation}
Letting $u=-t^x$ and dividing both sides by $(-t^x; t^2)_{\infty}$, we rewrite this identity as
\begin{equation}
\EPMM\left[ \prod_{i \in 2\Z_{\geqslant 0}}\left(\frac{1+q^{\lambda_{n-i}}t^{x+i}}{1+t^{x+i}}\right)\right] = \EPSM\left[  \prod_{i\in 2\Z_{\geqslant 0}} \left(\frac{1+t^{x+\la_{n-i}+i}}{1+t^{x+i}}\right)\right],
% = \EPHL\left[ \frac{1}{(-t^{x+n-\ell(\la)}; t^2)_{\infty}}\right]
\label{eq:relationMMSchur}
\end{equation}
for any $x\in \R$, with the convention that $\la_{-m}=+\infty$ for $m>0$.

\begin{remark}
One could prove using \eqref{eq:relationMMSchur} a half-space analogue of Corollary 5.9 in \cite{borodin2016stochastic}: The distribution of the length of a half-space Macdonald random partition is asymptotically the same whatever are the values  of $q$ and $t$, in the sense of asymptotic equivalence as in \cite[Definition 5.2]{borodin2016stochastic}. We do not need this result for the present paper.
\label{rem:asymptoticequivalence}
\end{remark}

\section{Fredholm Pfaffian formulas}
\label{sec:Fredholm}

\subsection{Notations}

The Pfaffian of a skew-symmetric  $2k\times 2k$ matrix $A$ is defined by
\begin{equation}
\Pf(A) = \frac{1}{2^k k!} \sum_{\sigma\in\mathcal{S}_{2k}} \mathrm{sgn}(\sigma) a_{\sigma(1)\sigma(2)}a_{\sigma(3)\sigma(4)} \dots a_{\sigma(2k-1)\sigma(2k)},
\label{eq:defPfaffian}
\end{equation}
where $\mathrm{sgn}(\sigma)$ is the signature of the permutation $\sigma$.

Let $(\mathbb{X}, \mu) $ be a measure space. For a $2\times 2$-matrix valued  skew-symmetric kernel,
$$ \kernel(x,y) = \begin{pmatrix}
\kernel_{11}(x,y) & \kernel_{12}(x,y)\\
\kernel_{21}(x, y) & \kernel_{22}(x,y)
\end{pmatrix},\ \ x,y\in \mathbb{X},$$
we define the Fredholm Pfaffian (introduced in \cite[Section 8]{rains2000correlation}) by
\begin{equation}
\Pf\big[\mathsf{J}+\kernel\big]_{L^2(\mathbb{X}, \mu)} = 1+\sum_{k=1}^{\infty} \frac{1}{k!}
\int_{\mathbb{X}} \dots \int_{\mathbb{X}}  \Pf\Big( \kernel(x_i, x_j)\Big)_{i,j=1}^k \mathrm{d}\mu^{\otimes k}(x_1, \dots, x_k),
\label{eq:defFredholmPfaffian}
\end{equation}
provided the series converges, where $\mu^{\otimes k}$ is the product measure. The kernel $\mathsf{J}$ is defined by
$$ \mathsf{J}(x,y)  = \delta_{x=y}
\begin{pmatrix}
0 & 1\\-1 &0
\end{pmatrix} .$$
For a function $\fkernel: \mathbb{X} \to \R$, we define
\begin{align*}
\Pf\big[\mathsf{J}+\fkernel \cdot \kernel\big]_{L^2(\mathbb{X}, \mu)} &:= \Pf\big[\mathsf{J}+ \kernel\big]_{L^2(\mathbb{X}, \fkernel \mu)}\\
&= 1+\sum_{k=1}^{\infty} \frac{1}{k!}
\int_{\mathbb{X}} \dots \int_{\mathbb{X}}  \left(\prod_{i=1}^k\fkernel(x_i)\right)\ \Pf\Big( \kernel(x_i, x_j)\Big)_{i,j=1}^k \mathrm{d}\mu^{\otimes k}(x_1 \dots x_k),
\end{align*}
provided the series converge. We will use the notation the notation  $\Pf\big[\mathsf{J}+\kernel\big]_{\mathbb{L}^2(\mathbb{X})},$
when $\mathbb{X}$ is a subset of $\R^n$ equipped with Lebesgue measure, and the notation
$\Pf[\mathsf{J}+\kernel]_{\ell^2(\mathbb{X})}$
when  $\mathbb{X}$ is a discrete set equipped with the counting measure.
In order to study limits of Fredholm Pfaffians, we will need that the expansion in \eqref{eq:defFredholmPfaffian} is absolutely convergent, and for that we will use Hadamard's bound in the form of the next lemma.
\begin{lemma}[Lemma 2.5 in \cite{baik2017pfaffian}]
	Let $\kernel(x, y)$ a $2\times 2$ matrix valued skew symmetric kernel. Assume that there exist a constant $C>0$ and constants $0\leqslant a < b$ such that for $x,y \in \mathbb{X}\subset \R$,
	\begin{equation*}
	\vert \kernel_{11}(x, y) \vert <C e^{ax+ay}, \ \
	\vert \kernel_{12}(x, y)\vert  = \vert K_{21}(y,x)\vert  <C e^{ax-by}, \ \
	\vert \kernel_{22}(x, y)\vert  <C e^{-bx-by}.
	\end{equation*}
	Then, for all $k\in \Z_{>0}$, $x_1, \dots, x_k\in \mathbb{X}$,
	$$ \Big\vert \Pf\big( \kernel(x_i, x_j)\big)_{i,j=1}^k \Big\vert < (2k)^{k/2}C^k\prod_{i=1}^{k}e^{-(b-a) x_i}.$$
	\label{lem:hadamard}
\end{lemma}

\subsection{Pfaffian Schur measure correlation kernel}
Consider a Pfaffian Schur measure with parameters  $(a_1, \dots, a_n) \in (0,1)^n $, i.e. a measure of the form
$$\PSM(\la) = \prod_{i<j}(1-a_ia_j) s_{\la}(a_1, \dots, a_n) \mathds{1}_{\lambda' \  \mathrm{even}}, \ \ \ \la\in \Y.$$
We know from Theorem 3.3 in \cite{borodin2005eynard} that the random point configuration $\Lambda:=\lbrace \la_i-i\rbrace_{i \in \Z_{>0}}$ generates a Pfaffian point process on $\Z$.  This means that for $y_1, \dots, y_k\in \Z$,
$$ \PP\left( \lbrace y_1, \dots, y_k\rbrace \subset \Lambda \right) = \Pf\big[\kernel(y_i, y_j)\big]_{i,j=1}^k, $$
 where $\kernel( u,  v)$ is a $2 \times 2$ matrix-valued skew symmetric kernel
  $$ \kernel(u,v)= \begin{pmatrix}
  \kernel_{11}(u,v) & \kernel_{12}(u,v) \\
  \kernel_{21}(u,v) & \kernel_{22}(u,v)
\end{pmatrix}, \ \ u,v\in \Z, $$
with $\kernel_{12} = -\big(\kernel_{21}\big)^T$. We refer to \cite[Section 4.1]{baik2017pfaffian} for general background on Pfaffian point processes.  For the Pfaffian Schur process,  the kernel is   \cite[Theorem 3.3]{borodin2005eynard}
\begin{align*}
\kernel^{\mathrm{Schur}}_{11}(u,v)&= \frac{1}{(2\I\pi)^2} \iint \frac{z-w}{(z^2-1)(w^2-1)(zw-1)}
\prod_{j=1}^n\left( \frac{(z-a_j)(w-a_j)}{zw(1-a_jz)(1-a_jw)} \right) \frac{\mathrm{d}z}{z^u}\frac{\mathrm{d}w}{w^v},\\
\kernel^{\mathrm{Schur}}_{12}(u,v)&=\frac{1}{(2\I\pi)^2} \iint \frac{z-w}{(z^2-1)w(zw-1)}  \prod_{j=1}^n\left( \frac{(z-a_j)(w-a_j)}{zw(1-a_jz)(1-a_jw)} \right) \frac{\mathrm{d}z}{z^u}\frac{\mathrm{d}w}{w^v},\\
\kernel^{\mathrm{Schur}}_{22}(u,v)&= \frac{1}{(2\I\pi)^2} \iint \frac{z-w}{zw(zw-1)}\prod_{j=1}^n\left( \frac{(z-a_j)(w-a_j)}{zw(1-a_jz)(1-a_jw)} \right) \frac{\mathrm{d}z}{z^u}\frac{\mathrm{d}w}{w^v},
\end{align*}
where for $\kernel^{\mathrm{Schur}}_{11}$ and $\kernel^{\mathrm{Schur}}_{12}$, the contours for $z$ and $w$ are positively oriented circles around zero with radius between $1$ and the $\min \lbrace a_j^{-1}\rbrace $ (so that in particular $\vert z w\vert>1$ along these contours), while for $\kernel^{\mathrm{Schur}}_{22}$ the contours for $z$ and $w$ are positively oriented circles around zero with radius smaller than $1$ (so that $\vert zw\vert <1$ along these contours).

Notice that $-\ell(\la)$ is the leftmost hole in the point process $\Lambda$ (or leftmost point in $\Lambda^\complement:=\Z\setminus \Lambda$). The correlation kernel of $\Lambda^\complement$ is $\kernel '= \mathsf{J}-\kernel^{\rm Schur}$.
Indeed, the probability that a set $Y=\lbrace y_1, \dots, y_k\rbrace $ is included in $\Lambda^\complement$ is also the probability that there are no points of $\Lambda$ in $Y$, i.e. the gap probability, which is given by $\Pf[\mathsf{J}-\kernel^{\rm Schur}]_{\ell^2(Y)}$. For a finite set $Y$, the Fredholm Pfaffian on $\ell^2(Y)$ is simply the Pfaffian of the matrix indexed by elements of $Y$, i.e. $\Pf\left[\kernel'(y_i, y_j)\right]_{i,j=1}^{k}$, such that the correlation kernel  of the point process $\Lambda^\complement$ is $\kernel'$ as claimed.
In particular,
$$ \PP\left(-\ell(\la)>x\right) = \Pf\left[ \kernel^{\mathrm{Schur}}\right]_{\ell^2(-\infty, x]} = \Pf\left[ \mathsf{J}-\kernel^\prime\right]_{\ell^2( -\infty, x]}.$$
Shifting the point process by $n$, we obtain that the correlation kernel of $ n+\Lambda^{\complement} $ is $\kernel^{\complement}$ where
\begin{subequations}
\begin{align}
\kernel^{\complement}_{11}(u,v)&= \frac{1}{(2\I\pi)^2} \iint \frac{w-z}{(z^2-1)(w^2-1)(zw-1)}
\prod_{j=1}^n\left( \frac{(z-a_j)(w-a_j)}{(1-a_jz)(1-a_jw)} \right) \frac{\mathrm{d}z}{z^u}\frac{\mathrm{d}w}{w^v},\label{eq:K6V11}\\
\kernel^{\complement}_{12}(u,v)&= \mathds{1}_{u=v} + \frac{1}{(2\I\pi)^2} \iint \frac{w-z}{(z^2-1)w(zw-1)} \prod_{j=1}^n\left( \frac{(z-a_j)(w-a_j)}{(1-a_jz)(1-a_jw)} \right) \frac{\mathrm{d}z}{z^u}\frac{\mathrm{d}w}{w^v},\label{eq:K6V12}\\
\kernel^{\complement}_{22}(u,v)&= \frac{1}{(2\I\pi)^2} \iint \frac{w-z}{zw(zw-1)}\prod_{j=1}^n\left( \frac{(z-a_j)(w-a_j)}{(1-a_jz)(1-a_jw)} \right) \frac{\mathrm{d}z}{z^u}\frac{\mathrm{d}w}{w^v},\label{eq:K6V22}
\end{align}
\end{subequations}
where the contours are as before.
The point process $n+\Lambda^\complement$ is almost surely supported on the nonnegative integers, so that one can compute the Fredholm Pfaffian on $\ell^2[0,x]$ instead of $\ell^2( -\infty, x]$.
Thus, by \cite[Theorem 8.2]{rains2000correlation}, one may write that for any function $\fkernel :\Z_{\geqslant 0}\to \R$,
\begin{equation}
\EPSM\left[ \prod_{\lambda\in n+\Lambda^{\complement}}\left( 1+\fkernel(\lambda) \right) \right] = \Pf\left[ \mathsf{J}+\fkernel \cdot  \kernel^{\complement}  \right]_{\ell^2(\Z_{\geqslant 0})},
\label{eq:multiplicativefunctionalsSchur}
\end{equation}
whenever both sides admit absolutely convergent expansions.
In particular,
$$ \PSM\big(n-\ell(\la)>x\big)  =  \Pf\left[ \mathsf{J}-\kernel^{\complement}\right]_{\ell^2[0, x]}.$$

It will be more convenient to work with integral formulas where the contours  are all circles with radius less than $1$ (because we will later let $a_j\equiv a$ go to $1$ and there is a pole at $1/a$). When deforming the contours inside the unit circle, we pick some residues which yield the following formulas.
\begin{lemma}
Denoting $f(z)= \prod_{j=1}^n \frac{z-a_j}{1-a_jz} $, for $n$ even and $u,v\in \Z_{\geqslant 0}$, we have that
\begin{align*}
\kernel^{\complement}_{11}(u,v)&=  \frac{1}{(2\I\pi)^2} \iint \frac{(w-z)f(z)f(w)}{(z^2-1)(w^2-1)(zw-1)}
 \frac{\mathrm{d}z}{z^u}\frac{\mathrm{d}w}{w^v}
 + \frac{1}{2\I\pi}\int  \frac{f(w)}{w^2-1} \frac{\mathrm{d}w}{w^v}\mathds{1}_{u \in 2\Z}\\
 & \hspace{3cm}  -\frac{1}{2\I\pi}\int  \frac{f(w)}{w^2-1} \frac{\mathrm{d}w}{w^u}\mathds{1}_{v \in 2\Z}  + r(u,v),\\
 \kernel^{\complement}_{12}(u,v)&= \frac{1}{(2\I\pi)^2} \iint \frac{(w-z)f(z)f(w)}{(z^2-1)w(zw-1)}   \frac{\mathrm{d}z}{z^u}\frac{\mathrm{d}w}{w^v}  +  \frac{1}{2\I\pi} \int \frac{f(z)}{z^{v+1}}\mathrm{d}z\mathds{1}_{u\in 2\Z},\\
\kernel^{\complement}_{22}(u,v)&= \frac{1}{(2\I\pi)^2} \iint \frac{(w-z)f(z)f(w)}{zw(zw-1)} \frac{\mathrm{d}z}{z^u}\frac{\mathrm{d}w}{w^v},
\end{align*}
where the contours are all positively oriented circles around $0$ with radius smaller than $1$, and
\begin{equation}
 r(u,v) = \frac{1}{4}\big((-1)^u - (-1)^v\big) +\frac{1}{2}\sgn(v-u)\mathds{1}_{v-u\in 2\Z+1},
 \label{eq:defr}
\end{equation}
with the convention that
$$ \sgn(x) = \begin{cases} 1, &\text{ if }\  x>0, \\  -1,
&\text{ if }\ x<0, \\ 0, &\text{ if }\ x=0.
\end{cases}$$
\label{lem:residuecomputations}
\end{lemma}
\begin{proof}
Let $\mathcal{C}_{>1}$ be a contour defined by a (positively oriented) circle of radius larger than $1$ but arbitrarily close to $1$, let $\mathcal{C}_{<1}$ be a contour with radius smaller than $1$ but arbitrarily close to $1$, and $\mathcal{C}_{\ll 1}$ a contour with radius arbitrarily close to zero.

%We recall $f(z)=\left( \frac{z-a}{1-az}\right)^n$. Assume that $n$ is even and $u,v\in \Z_{\geqslant 0}$.

In the formula for $\kernel^{\complement}_{11}(u,v)$ in \eqref{eq:K6V11}, the integration variables $z$ and $w$ are such that $\vert z\vert, \vert w\vert, \vert zw\vert>1$.  We may first deform the $z$ contour from $\mathcal{C}_{>1}$ to $\mathcal{C}_{\ll 1}$ thus picking residues at $z=\pm1$ and $z=1/w$. This yields
\begin{multline*}
\kernel^{\complement}_{11}(u,v) =  \frac{1}{(2\I\pi)^2} \int_{\mathcal{C}_{>1}} \mathrm{d}w \int_{\mathcal{C}_{\ll 1}} \mathrm{d}z  \frac{(w-z)f(z)f(w)}{(z^2-1)(w^2-1)(zw-1)}
 \frac{1}{z^u}\frac{1}{w^v}\\
+ \frac{1}{2\I\pi}\int_{\mathcal{C}_{>1}} \mathrm{d}w \frac{1}{2(w^2-1)} \frac{f(w)}{w^v}\Big(f(1)+(-1)^uf(-1) \Big)
+ \int_{\mathcal{C}_{>1}} \mathrm{d}w \frac{1}{1-w^2}\frac{1}{w^{v-u}}.
\end{multline*}
Evaluating the residues at $\pm 1$ in the second integral, and taking into account that $n$ is even so that $f(1)=f(-1)=1$, we find
\begin{multline*}
\kernel^{\complement}_{11}(u,v) =  \frac{1}{(2\I\pi)^2} \int_{\mathcal{C}_{>1}} \mathrm{d}w \int_{\mathcal{C}_{\ll 1}} \mathrm{d}z  \frac{(w-z)f(z)f(w)}{(z^2-1)(w^2-1)(zw-1)}
 \frac{1}{z^u}\frac{1}{w^v}\\
+\frac{1}{2\I\pi} \int_{\mathcal{C}_{\ll 1}} \mathrm{d}w \frac{1}{w^2-1} \frac{f(w)}{w^v}\mathds{1}_{u \in 2\Z} + r(u,v),
\end{multline*}
where $r(u,v)$ is defined in \eqref{eq:defr}.
This term  $r(u,v)$ corresponds to taking residues in the variable $z$ and then  in the variables $w$ in  \eqref{eq:K6V11} for values of $(z,w)$ equal to
$$(1/w,1), (1/w,-1), (1/w, 0), (1, 1), (1, -1), (-1, 1), \text{ and } (-1, -1).$$
One readily checks that the sum of all these residues equals $r(u,v)$.

Deforming the contour for the variable $w$ in the first integral from $\mathcal{C}_{>1}$ to $\mathcal{C}_{\ll 1}$ we pick residues at $\pm 1$ which yield
\begin{multline*}
\kernel^{\complement}_{11}(u,v) =  \frac{1}{(2\I\pi)^2} \int_{\mathcal{C}_{\ll 1}} \mathrm{d}w \int_{\mathcal{C}_{\ll 1}} \mathrm{d}z  \frac{(w-z)f(z)f(w)}{(z^2-1)(w^2-1)(zw-1)}
\frac{1}{z^u}\frac{1}{w^v}\\
 -\frac{1}{2\I\pi}\int_{\mathcal{C}_{\ll 1}} \mathrm{d}w \frac{1}{w^2-1} \frac{f(w)}{w^u}\mathds{1}_{v \in 2\Z} + \frac{1}{2\I\pi}\int_{\mathcal{C}_{\ll 1}} \mathrm{d}w \frac{1}{w^2-1} \frac{f(w)}{w^v}\mathds{1}_{u \in 2\Z} + r(u,v).
\end{multline*}

In the formula for $\kernel^{\complement}_{12}(u,v)$ in \eqref{eq:K6V12}, we may deform the contour for $w$ to $\mathcal{C}_{< 1}$ without picking any residue. Then, deforming the contour for $z$ to $\mathcal{C}_{< 1}$ we pick residues at $z=1/w$ and $z=\pm 1$ which yields
\begin{multline*}
\kernel^{\complement}_{12}(u,v)= \mathds{1}_{u=v} + \frac{1}{(2\I\pi)^2} \int_{\mathcal{C}_{< 1}} \mathrm{d}w \int_{\mathcal{C}_{< 1}} \mathrm{d}z \frac{(w-z)f(z)f(w)}{(z^2-1)w(zw-1)}  \frac{\mathrm{d}z}{z^u}\frac{\mathrm{d}w}{w^v}\\
+\frac{1}{2\I\pi} \int_{\mathcal{C}_{<1}} \mathrm{d}w \frac{-1}{w^{v-u+1}} + \frac{1}{2\I\pi} \int_{\mathcal{C}_{<1}}  \frac{f(w)}{w^{v+1}}\mathds{1}_{u\in 2\Z},
\end{multline*}
which simplifies to
\begin{equation*}
\kernel^{\complement}_{12}(u,v)= \frac{1}{(2\I\pi)^2} \int_{\mathcal{C}_{< 1}} \mathrm{d}w \int_{\mathcal{C}_{< 1}} \mathrm{d}z \frac{(w-z)f(z)f(w)}{(z^2-1)w(zw-1)}   \frac{\mathrm{d}z}{z^u}\frac{\mathrm{d}w}{w^v}
 + \frac{1}{2\I\pi} \int_{\mathcal{C}_{<1}}  \frac{f(w)}{w^{v+1}}\mathds{1}_{u\in 2\Z}.
\end{equation*}
\end{proof}

\subsection{Hall-Littlewood observables}

We can use the relation between Pfaffian Schur and Hall-Littlewood measure \eqref{eq:relationMMSchurbis} to express certain observables of the latter using the correlation kernel of the former.
\begin{proposition} For any $x\in \R$, $n\in 2\Z_{>0}$, and any parameters $(a_1, \dots, a_n)\in(0,1)^n$,
\begin{equation}
\EPHL\left[\frac{1}{(-t^{x+n-\ell(\la)}, t^2)_{\infty}}\right] =  \Pf\left[  \mathsf{J}+\fkernel_x\cdot \kernel^{\complement}\right]_{\ell^2(\Z_{\geqslant 0})},
\label{eq:FredholmHL}
\end{equation}
where
\begin{equation}
 \fkernel_x(j) = \frac{(-t^{x+j+1}; t^2)_{\infty}}{(-t^{x+j}; t^2)_{\infty}}-1.
\label{eq:deffkernel}
\end{equation}
\label{prop:FredholmHL}
\end{proposition}
Before proving this proposition, it is worth noticing that $\fkernel_x(j)\in(-1,0)$ and $\fkernel_x(j)$ is increasing in $x$ and $j$. This will be of importance later when we consider scaling limits.
\begin{proof}
Letting $u=-t^x$ and $q=0$ in identity  \eqref{eq:relationMMSchurbis} implies that
\begin{equation}
\EPHL\left[\prod_{i\in 2\Z\cap[0, n-\ell(\la)-2]}  \left(1+t^{x+i}\right)\right] =  \EPSM\left[ \prod_{i\in 2\Z\cap [0,n]} \left(1+t^{x+\la_i-i+n}\right)\right].
\label{eq:relationHLSchur}
\end{equation}
The left-hand-side of \eqref{eq:relationHLSchur} can be rewritten
$$  (-t^x; t^2)_{\infty}\ \  \EPHL\left[\frac{1}{(-t^{x+n-\ell(\la)}, t^2)_{\infty}}\right],$$
and  the right-hand-side can be rewritten  as
$$ \EPSM\left[  \prod_{j\in J} \left( 1+t^{x+j} \right)\right], $$
where  $J= \lbrace \la_i-i+n\rbrace_{i\in 2\Z\cup[0,n]}$. Since $\la_1=\la_2, \la_3=\la_4, \dots$ under the Pfaffian Schur measure that we consider, the set $J$ is characterized by  $J \sqcup (J+1) = (n+\Lambda)\cap \Z_{\geqslant 0}$, where $\sqcup$ denotes the union  of disjoint sets and $\Lambda = \lbrace \la_i-i \rbrace_{i\in \Z_{>0}}$ as before. Let us call $P(x)$ the product inside the last expectation. We have
$$ P(x)P(x+1) = \prod_{j\in (n+\Lambda)\cap \Z_{\geqslant 0} }(1+t^{x+j})=:N(x).$$
This implies that
$$ P(x)  = \frac{N(x)N(x+2)N(x+4)\dots}{N(x+1)N(x+3)N(x+5)\dots} = \prod_{j\in (n+\Lambda)\cap \Z_{\geqslant 0}} \frac{(-t^{x+j}; t^2)_{\infty}}{(-t^{x+j+1}; t^2)_{\infty}}.$$
Thus we have shown that
\begin{equation}
\EPSM\left[ \prod_{i\in 2\Z\cap [0,n]} \left(1+t^{x+\la_i-i+n}\right)\right] = \EPSM\left[ \prod_{j\in (n+\Lambda)\cap \Z_{\geqslant 0}} \frac{(-t^{x+j}; t^2)_{\infty}}{(-t^{x+j+1}; t^2)_{\infty}}\right].
\label{eq:todivide}
\end{equation}
Notice the simplification
$$ \prod_{j\in \Z_{\geqslant 0}}\frac{(-t^{x+j}; t^2)_{\infty}}{(-t^{x+j+1}; t^2)_{\infty}}  = (-t^x; t^2)_{\infty}, $$
so that dividing  both sides of \eqref{eq:todivide} by that quantity,
$$\frac{1}{ (-t^x; t^2)_{\infty}} \EPSM\left[ \prod_{i\in 2\Z\cap [0,n]}\left( 1+t^{x+\la_i-i+n}\right)\right] = \EPSM\left[ \prod_{j\in n+\Lambda^{\complement}}  \frac{(-t^{x+j+1}; t^2)_{\infty}}{(-t^{x+j}; t^2)_{\infty}} \right].$$
At this point we have shown the following relation between observables of the half-space Schur and Hall-Littlewood processes:
$$  \EPHL\left[\frac{1}{(-t^{x+n-\ell(\la)}, t^2)_{\infty}}\right]  =  \EPSM\left[ \prod_{j\in n+\Lambda^{\complement}}  \frac{(-t^{x+j+1}; t^2)_{\infty}}{(-t^{x+j}; t^2)_{\infty}} \right].$$
To complete the proof, we note that the multiplicative functional of the Pfaffian Schur measure on the right-hand-side can be computed using \eqref{eq:multiplicativefunctionalsSchur} as in the statement of the proposition.
\end{proof}
%\begin{remark}
% Example 5.5 in \cite{borodin2016stochastic} shows that for any sequence of measures on partitions with at most $n$ parts -- and in particular for $\PHL$ with specialization into $a_1, \dots ,a_n$, the sequence of random variables $\lbrace \ell(\la)-n\rbrace$ is asymptotically equivalent to the sequence of functions
%$$x\mapsto  \EE\left[ \frac{1}{(-t^{x+n-\ell(\la)}; t^2)_{\infty}}\right].$$
%Thus, \eqref{eq:FredholmHL} relates the asymptotic distribution of $n-\ell(\la)$ under $\PHL$ with a multiplicative functional of $n+\Lambda^{\complement}$.
%\end{remark}

\section{Half-space Hall-Littlewood measures and stochastic six-vertex model in a half-quadrant}
\label{sec:6v}

\renewcommand{\leq}{\leqslant}
\renewcommand{\geq}{\geqslant}
\colorlet{lgray}{white!85!black}
\colorlet{lred}{white!65!red}

\newcommand{\bra}[1]{\left\langle #1\right|}
\newcommand{\ket}[1]{\left|#1\right\rangle}
\renewcommand{\tikz}[2]{\begin{tikzpicture}[scale=#1,baseline=(current bounding box.center),>=stealth] #2 \end{tikzpicture}}
\renewcommand{\gather}[1]{\begin{gathered} #1 \end{gathered}}
We will define a measure on lattice paths as in Figure \ref{fig:example6v}. There are two types of vertices, bulk and corner vertices.
\subsection{Definition of the model}
\label{sec:def6v} Consider the square lattice.
A {\it bulk vertex} is the crossing of a horizontal and a vertical line, as well as the four edges which surround the point of intersection. We refer to the four edges which constitute a vertex as its north, east, south and west edges, following standard compass orientation. Bulk vertices have the following generic form:

\begin{align}
\label{eq:vert}
\begin{tikzpicture}[scale=1.2,baseline=0]
\draw[dotted] (-1,0) -- node[circle,scale=0.7,fill=gray] {\color{white} $i_1$} (0,0) -- node[circle,scale=0.7,fill=gray] {\color{white} $j_1$} (1,0);
\draw[dotted] (0,-1) -- node[circle,scale=0.7,fill=gray] {\color{white} $i_2$} (0,0) -- node[circle,scale=0.7,fill=gray] {\color{white} $j_2$} (0,1);
\node[left] at (-1,0) {$a_y$};
\node[below] at (0,-1) {$a_x$};
\end{tikzpicture}\ ,
\quad\quad
\{i_1,i_2,j_1,j_2\} \in \{0,1\}.
\end{align}
The four indices placed around the vertex represent its {\it edge states.} We refer to $i_1,i_2$ as {\it incoming} states, while $j_1,j_2$ are called {\it outgoing.} These indices take values in $\{0,1\} = \{\text{empty},\text{occupied}\}$. Whenever an edge state is equal to $1$ we draw an up-oriented or right-oriented {\it path} on that edge, but leave the edge empty if its state is equal to $0$. A vertex centered at position $(x,y)$ also has two parameters $a_x,a_y$ associated to it and is assigned a Boltzmann weight that depends on its edge states and which is a function of the product  $a_x a_y$. Of the sixteen possible edge state configurations, only six receive non-zero Boltzmann weights: these are precisely the six that exhibit conservation of paths passing through the vertex. We list their weights below:
\begin{align}
\label{eq:6v}
\begin{array}{cccccc}
\tikz{0.7}{
\draw[dotted] (-1,0) -- (1,0);
\draw[dotted] (0,-1) -- (0,1);
\node[left] at (-1,0) {$a_y$};
\node[below] at (0,-1) {$a_x$};
}
&
\tikz{0.7}{
\draw[path] (-1,0) -- (1,0);
\draw[path] (0,-1) -- (0,1);
\node[left] at (-1,0) {$a_y$};
\node[below] at (0,-1) {$a_x$};
}
&
\tikz{0.7}{
\draw[path] (-1,0) -- (1,0);
\draw[dotted] (0,-1) -- (0,1);
\node[left] at (-1,0) {$a_y$};
\node[below] at (0,-1) {$a_x$};
}
&
\tikz{0.7}{
\draw[dotted] (-1,0) -- (1,0);
\draw[path] (0,-1) -- (0,1);
\node[left] at (-1,0) {$a_y$};
\node[below] at (0,-1) {$a_x$};
}
&
\tikz{0.7}{
\draw[path] (-1,0) -- (0,0) -- (0,1);
\draw[dotted] (0,-1) -- (0,0) -- (1,0);
\node[left] at (-1,0) {$a_y$};
\node[below] at (0,-1) {$a_x$};
}
&
\tikz{0.7}{
\draw[dotted] (-1,0) -- (0,0) -- (0,1);
\draw[path] (0,-1) -- (0,0) -- (1,0);
\node[left] at (-1,0) {$a_y$};
\node[below] at (0,-1) {$a_x$};
}
\\ \\
\ \ \ 1
&
\ \ \ 1
&
\ \ \ \dfrac{1-a_xa_y}{1-ta_xa_y}
&
\ \ \ \dfrac{t(1-a_xa_y)}{1-ta_xa_y}
&
\ \ \ \dfrac{(1-t)a_xa_y}{1-ta_xa_y}
&
\ \ \ \dfrac{1-t}{1-ta_xa_y}
\end{array}
\end{align}
The Boltzmann weights are stochastic in the following sense:
\begin{proposition}[Stochasticity]\label{prop:stoch}
Let $w_{a_x a_y} (i_1,i_2 ; j_1,j_2)$ be the weight of the vertex in \eqref{eq:vert}. The parameters $a_x,a_y,t$ can be chosen such that $0 \leq w_{a_x a_y}(i_1,i_2 ; j_1,j_2) \leq 1$ for all $\{i_1,i_2,j_1,j_2\}$, and for any fixed $i_1,i_2$, we have
\begin{align}
\label{eq:prob1}
\sum_{j_1,j_2 \in \{0,1\}}
w_{a_x a_y}(i_1,i_2 ; j_1,j_2)
\equiv
1.
\end{align}
\end{proposition}

\begin{proof}
The first property holds if we assume that $a_x,a_y,t$ are real and satisfy $0 \leq a_x,a_y,t < 1$. The four cases $\{i_1,i_2\} = \{0,0\}, \{1,1\}, \{0,1\}, \{1,0\}$ of \eqref{eq:prob1} can be easily checked, using the vertex weights \eqref{eq:6v}.
\end{proof}

\begin{figure}
\begin{tabular}{cc}
\begin{tikzpicture}[scale=0.7]
\foreach \x in {1, ..., 7} {
\draw[dotted] (0,\x) -- (\x,\x) -- (\x, 8);
\draw[gray] (\x, \x-1) node{$\x$};
\draw[gray] (-1,\x) node{$\x$};
}
\end{tikzpicture}
\quad
&
\quad
\begin{tikzpicture}[scale=0.7]
\foreach \x in {1, ..., 7} {
\draw[dotted] (0,\x) -- (\x,\x) -- (\x, 8);
\draw[path] (0,\x) --(0.1,\x);
\draw[gray] (\x, \x-1) node{$\x$};
\draw[gray] (-1,\x) node{$\x$};
}
\draw[path] (0,1) --(1,1);
\draw[path] (0,2) --(1,2) -- (1,3) -- (2,3) -- (2,4) -- (4,4);
\draw[path] (0,3) --(1,3) -- (1,4) -- (2,4) -- (2,5) -- (3,5) -- (3,6) -- (4,6) -- (5,6) -- (5,7)--(7,7) ;
\draw[path] (2,2) -- (2,3) -- (3,3);
\draw[path] (5,5) -- (5,6) -- (6,6);
\draw[path] (0,4) --(1,4) -- (1,5) -- (2,5) -- (2,6) -- (3,6) -- (3,7) -- (5,7) -- (5,8);
\draw[path] (0,5) --(1,5) -- (1,6) -- (2,6) -- (2,7) -- (3,7) -- (3,8);
\draw[path] (0,6) --(1,6) -- (1,7) -- (2,7) -- (2,8) ;
\draw[path] (0,7) --(1,7) -- (1,8)  ;
\draw[path] (2,2)--(2,2.1);
\draw[path] (5,5)--(5,5.1);
\end{tikzpicture}
\end{tabular}
\caption{Left panel: the half-quadrant. Right panel: sample configuration $\mathcal{C}$ of the stochastic six-vertex model in the half-quadrant, for which $\mathfrak{h}(7,7)=4$. The seventh path string is given by $\sigma_7(\mathcal{C})=(1,1,1,0,1,0,0)$.}
\label{fig:half_quad}
\end{figure}
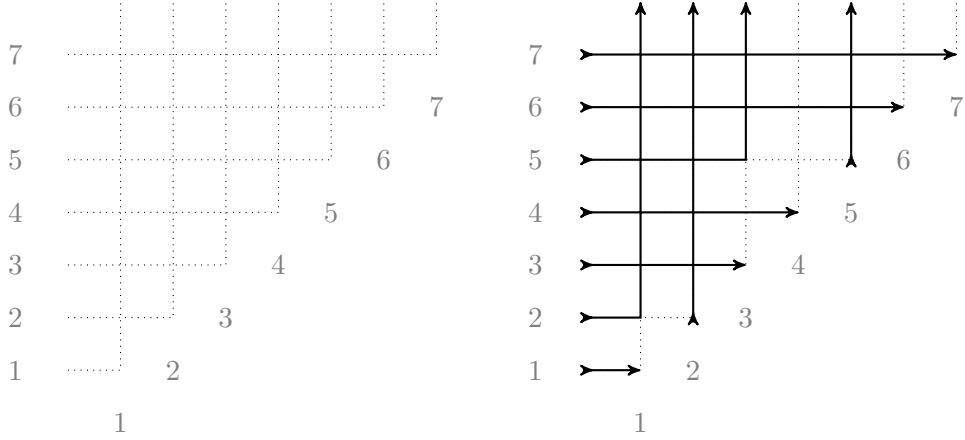

A {\it corner vertex} is a vertex formed by the union of a north and west edge (with the omission of south and east edges). They have the following form:
\begin{align}
\label{eq:corner_vert}
\begin{tikzpicture}[scale=1.2,baseline=0]
\draw[dotted] (-1,0) -- node[circle,scale=0.7,fill=gray] {\color{white} $i$} (0,0) -- node[circle,scale=0.7,fill=gray] {\color{white} $j$} (0,1);
\node[left] at (-1,0) {$a_x$};
\node[below] at (0,-0.5) {$a_x$};
\end{tikzpicture}\ ,
\quad\quad
\{i,j\} \in \{0,1\}.
\end{align}
There are two types of corner vertices to which we assign Boltzmann weight $1$: those in which a path enters from the left and is absorbed at the center of the vertex, with no path emerging from the top; and those in which no path enters from the left, with a path being created at the center of the vertex and then emerging from the top. The remaining two corner vertex configurations are disallowed and receive a weight of $0$:
\begin{align}
\label{eq:corner}
\begin{array}{cccc}
\tikz{0.7}{
\draw[dotted] (-1,0) -- (0,0);
\draw[path] (0,0) -- (0,1);
\node[left] at (-1,0) {$a_x$};
\node[below] at (0,-0.5) {$a_x$};
}
\quad&\quad
\tikz{0.7}{
\draw[path] (-1,0) -- (0,0);
\draw[dotted] (0,0) -- (0,1);
\node[left] at (-1,0) {$a_x$};
\node[below] at (0,-0.5) {$a_x$};
}
\quad&\quad
\tikz{0.7}{
\draw[dotted] (-1,0) -- (0,0);
\draw[dotted] (0,0) -- (0,1);
\node[left] at (-1,0) {$a_x$};
\node[below] at (0,-0.5) {$a_x$};
}
\quad&\quad
\tikz{0.7}{
\draw[path] (-1,0) -- (0,0);
\draw[path] (0,0) -- (0,1);
\node[left] at (-1,0) {$a_x$};
\node[below] at (0,-0.5) {$a_x$};
}
\\ \\
\ \ \ 1
\quad&\quad
\ \ \ 1
\quad&\quad
\ \ \ 0
\quad&\quad
\ \ \ 0
\end{array}
\end{align}
In contrast to bulk vertices, the Boltzmann weight assigned to a corner vertex is independent of the parameter $a_x$ which is attached to it.

The Boltzmann weights are chosen as in \eqref{eq:6v} and \eqref{eq:corner} because they satisfy the Yang--Baxter equation (bulk vertices) and a boundary Yang--Baxter or reflection equation (corner vertices).
\begin{proposition}[Yang--Baxter equation]
For any fixed $0 \leq i_1,i_2,i_3,j_1,j_2,j_3 \leq 1$, we have
\begin{multline*}
\sum_{0 \leq k_1,k_2,k_3 \leq 1}\
w_{a_z/a_y}(i_1,i_2 ; k_1,k_2)
w_{a_x a_z}(k_1,i_3 ; j_1,k_3)
w_{a_x a_y}(k_2,k_3 ; j_2,j_3)
\\
=
\sum_{0 \leq k_1,k_2,k_3 \leq 1}\
w_{a_x a_y}(i_2,i_3 ; k_2,k_3)
w_{a_x a_z}(i_1,k_3 ; k_1,j_3)
w_{a_z/a_y}(k_1,k_2 ; j_1,j_2),
\end{multline*}
where $a_x,a_y,a_z$ are three arbitrary parameters ($w_{a_z/a_y}$ is obtained from $w_{a_xa_y}$ by substituting $a_z/a_y$ in place of $a_xa_y$ in \eqref{eq:6v}). 
\label{prop:YangBaxtersimple}
\end{proposition}

\begin{proof}
This is a classical result in statistical mechanics \cite{baxter1982exactly}. It can also be checked by direct computation, although in this case there are $2^6$ individual equations to verify.
\end{proof}
\begin{proposition}[Reflection equation]
Let the weight of the corner vertex in \eqref{eq:corner_vert} be denoted by $c(i;j) = \delta_{i,1-j}$.
For any fixed $0 \leq i_1,i_2,j_1,j_2 \leq 1$, we have
\begin{multline}
\label{eq:reflect}
\sum_{0 \leq k_1,k_2,\ell_1,\ell_2 \leq 1}
w_{a_y/a_x}(i_1,i_2;k_1,k_2) c(k_1;\ell_1)
w_{a_x a_y}(k_2,\ell_1;\ell_2,j_1) c(\ell_2;j_2)
\\
=
\sum_{0 \leq k_1,k_2,\ell_1,\ell_2 \leq 1}
c(i_2;k_2) w_{a_x a_y}(i_1,k_2;k_1,\ell_2)
c(k_1;\ell_1) w_{a_y/a_x}(\ell_2,\ell_1;j_2,j_1),
\end{multline}
where $a_x,a_y$ are two arbitrary parameters ($w_{a_y/a_x}$ is obtained from $w_{a_xa_y}$ by substituting $a_y/a_x$ in place of $a_xa_y$ in \eqref{eq:6v}). 
\label{prop:boundaryYangBaxtersimple}
\end{proposition}
\begin{proof}
This result is due to Kuperberg \cite{kuperberg2002symmetry, sklyanin1988boundary}. One can eliminate redundant sums in \eqref{eq:reflect}, using the fact that $c(i;j) = \delta_{i,1-j}$. This gives
\begin{multline*}
\sum_{0 \leq k_1,k_2 \leq 1}
w_{a_y/a_x}(i_1,i_2;k_1,k_2) w_{a_x a_y}(k_2,1-k_1;1-j_2,j_1)
\\
=
\sum_{0 \leq \ell_1,\ell_2 \leq 1}
w_{a_x a_y}(i_1,1-i_2;1-\ell_1,\ell_2) w_{a_y/a_x}(\ell_2,\ell_1;j_2,j_1),
\end{multline*}
which can then be checked for the sixteen possible values of $\{i_1,i,_2,j_1,j_2\}$.
\end{proof}
Propositions \ref{prop:YangBaxtersimple} and \ref{prop:boundaryYangBaxtersimple} are the reasons behind the integrability of the stochastic six-vertex model and its symmetries, and this is why we mention them. We will not use Propositions \ref{prop:YangBaxtersimple} and \ref{prop:boundaryYangBaxtersimple} explicitly in the following, because these results are superseded by Propositions \ref{prop:graphicalCauchy} and \ref{prop:gzrelation} below, which apply to a slightly more general model.
\subsection{Markov process on the half-quadrant}

Let us consider the following subset of $\mathbb{Z}^2$:
\begin{align*}
\lbrace (x,y)\in \Z_{>0}^2: x\leqslant y \rbrace.
\end{align*}
We refer to this as the {\it half-quadrant.} There is a bulk vertex at each point $(x,y) \in \mathbb{Z}_{>0}^2$ such that $x<y$, while the points $(x,x) \in \mathbb{Z}_{>0}^2$ are occupied by corner vertices (see the left panel of Figure \ref{fig:half_quad}). We now study a discrete time Markov process of up-right paths on the half-quadrant. It is defined inductively as follows:
\begin{itemize}
\item Let there be an incoming path on the west edge of each vertex at $(1,j)$, $j \in \mathbb{Z}_{\geq 1}$.
\item Assume that for some $n \geq 2$, the incoming edge states of the vertices
$\{(x,y)\}_{x+y=n}$ are all determined. Choose the outgoing edge states of these vertices by sampling from the Bernoulli distribution imposed by the vertex weights \eqref{eq:6v}, \eqref{eq:corner}:
\begin{align*}
\PP\left(\pathrr\right) = \frac{1-a_xa_y}{1-ta_xa_y},
\ \ \PP\left(\pathru\right) = \frac{(1-t)a_xa_y}{1-ta_xa_y},
\ \ \PP\left(\pathuu\right) = \frac{t(1-a_xa_y)}{1-ta_xa_y},
\ \ \PP\left(\pathur\right) =\frac{1-t}{1-ta_xa_y},
\end{align*}
 and when $n$ is even, choose the configuration of the corner vertex according to
\begin{align*}
\PP\left(\pathbrr\hspace{0.2cm}\right) =  \PP\left(\pathbuu\hspace{0.2cm}\right) = 1,
\ \ \ \PP\left(\pathbru\hspace{0.2cm}\right) =  \PP\left(\pathbur\hspace{0.2cm}\right) =0.
\end{align*}
This determines the incoming states of the vertices $\{(x,y)\}_{x+y=n+1}$.
\item One can repeat this procedure to fill out the whole of the half-quadrant by induction on $n$.
\end{itemize}
This Markovian procedure defines the stochastic six-vertex model on the half-quadrant. Equivalently, one can think of this procedure as inducing a probability measure on random configurations of paths in the half-quadrant: the probability of the cylindric set of path that start off some fixed finite configuration near the origin is just the product of the Boltzmann weights of the vertices in that configuration. The latter point of view will be especially useful in what follows.

\subsection{Height function and path string distribution}

The height function $\mathfrak{h}$ is a random variable defined on the vertices of the half-quadrant. For all $(x,y)$ such that $x \leq y$, we define
\begin{align}
\mathfrak{h}(x,y) =
\text{number of paths that cross one of the vertices}\ (i,y)\ \text{for}\ 1 \leq i \leq x.
\label{eq:defheight6v}
\end{align}
More generally, we will be interested in the distribution of up-right paths which exit to the north of the $n$-th horizontal gridline, for some $n \geq 1$. This collection of paths forms a random binary string $(s_1,\dots,s_n)$, where each $s_k$ equals $1$ if the north edge of the vertex $(k, n)$ is occupied and $0$ else. We refer to it as the $n$-th {\it path string.} If $\mathcal{C}$ is a configuration of the stochastic six-vertex model on the half-quadrant, we let $\sigma_n(\mathcal{C})$ denote its $n$-th path string (see the right panel of Figure \ref{fig:diagram}). The problem of finding the distribution of $\sigma_n(\mathcal{C})$ is equivalent to calculating the partition functions on triangles in the half-quadrant:
\begin{align*}
\PP \Big[ \sigma_n(\mathcal{C}) = \{s_1,\dots,s_n\} \Big]
\ \ = \ \
\begin{tikzpicture}[scale=0.8,baseline=(current bounding box.center)]
\foreach \x in {1, ..., 5} {
\draw[dotted] (0,\x) -- (\x,\x) -- (\x, 6) node[circle,scale=0.65,fill=gray] {\color{white} $s_\x$};
\draw[path] (0,\x) --(1,\x);
%\node[above] at (\x,6) {$j_\x$};
\draw[gray] (\x, \x-1) node{$\x$};
\draw[gray] (-1,\x) node{$\x$};
}
\end{tikzpicture}
\end{align*}
where summation is implicit over all internal edges in the lattice shown above.

When one knows $\sigma_n(\mathcal{C}) = \{s_1,\dots,s_n\}$, one can clearly reconstruct the value of the height function everywhere along the $n$-th horizontal gridline, using the fact that
\begin{align}
\label{height-path_string}
\mathfrak{h}(x,n) = \sum_{k=1}^{x} s_k.
\end{align}
Now we state the main result of this section:
\begin{theorem}\label{th:matchingHL6V}
Let $\lambda^{(1)} \subset \cdots \subset \lambda^{(n)}$ be a sequence of partitions in the ascending half-space Hall--Littlewood process (i.e. the process of Definition \ref{def:ascendingprocess}, at $q=0$), with associated probability measure
\begin{align}
\label{asc_HL}
\PHL \left( \varnothing = \lambda^{(0)} \subset \lambda^{(1)} \subset \cdots \subset \lambda^{(n)} = \lambda \right)
=
\frac{
P_{\lambda^{(1)}/\lambda^{(0)}}(a_1)
\dots
P_{\lambda^{(n)}/\lambda^{(n-1)}}(a_n)
}
{\Phi(a_1,\dots,a_n)}
\ \bel{\lambda}\
\mathds{1}_{\lambda'\,\text{even}}.
\end{align}
Let $[\lambda^{(1)} \subset \cdots \subset \lambda^{(n)}]$ encode the {\it support} of the sequence $\lambda^{(1)} \subset \cdots \subset \lambda^{(n)}$, defined as  the vector obtained by taking the difference in lengths of adjacent partitions:
\begin{align*}
[\lambda^{(1)} \subset \cdots \subset \lambda^{(n)}]
:=
\Big(\ell\left(\lambda^{(i)}\right) - \ell\left(\lambda^{(i-1)}\right)\Big)_{1 \leq i \leq n}.
\end{align*}
The following equivalence of distributions holds:
\begin{align}
\label{equiv_distrib}
\PHL \left( [\lambda^{(1)} \subset \cdots \subset \lambda^{(n)}] = (s_1,\dots,s_n) \right)
=
\PP \Big[ \sigma_n(\mathcal{C}) = (s_1,\dots,s_n) \Big],
\end{align}
where the right hand side is the path-string distribution in the stochastic six-vertex model.
\end{theorem}

In view of \eqref{height-path_string}, this theorem straight away yields the following corollary:
\begin{corollary}
Let $\big(\ell(\lambda^{(i)})\big)_{1 \leq i \leq n}$ be  the lengths of partitions in an ascending half-space Hall--Littlewood process \eqref{asc_HL}, and $\big(\mathfrak{h}(i,n)\big)_{1 \leq i \leq n}$ be the values of the height function along the $n$-th horizontal line in the half space six-vertex model. These two random vectors are equally distributed:
\begin{align*}
\PHL\left(
\big(\ell(\lambda^{(i)})\big)_{1 \leq i \leq n}
=
\big(k_i\big)_{1 \leq i \leq n} \right)
=
\PP\left(
\big(\mathfrak{h}(i,n)\big)_{1 \leq i \leq n}
=
\big(k_i\big)_{1 \leq i \leq n}
\right).
\end{align*}
\label{cor:equaldistHL6V}
\end{corollary}

The rest of the section is devoted to the proof of Theorem \ref{th:matchingHL6V}. It proceeds along parallel lines to a proof in \cite{borodin2016between}, relating the distribution of lengths of partitions in an (ordinary) ascending Hall--Littlewood process to the distribution of the six-vertex model height function in the full quadrant $\mathbb{Z}_{\geq 0}^2$. Both the proof in \cite{borodin2016between} and in the present paper are extensions of ideas that were developed in \cite{wheeler2016refined}, where integrability in a model of $t$-deformed bosons was used to prove refined Cauchy and Littlewood-type summation identities involving Hall--Littlewood polynomials. Similar ideas in the context of slightly more general higher spin six-vertex model were developed independently  in  \cite{borodin2017family, borodin2016higher}.

\begin{remark}
One could prove a slightly more general version of Theorem \ref{th:matchingHL6V} relating not-necessarily ascending half-space Hall-Littlewood process with the height distribution in a six-vertex model in a more complicated domain (jagged domain) following the same lines as in the proof of Theorem 5.6 in \cite{borodin2016between}. This is not useful for our present purposes, and we do not pursue it.
\end{remark}
\subsection{An integrable model of $t$-bosons}

Following \cite{borodin2016between,wheeler2016refined}, we consider another integrable model of up-right paths, in which horizontal edges of the lattice can be occupied by at most one path, but no restriction is imposed on the number of paths that traverse a vertical edge. Assuming conservation of lattice paths through a vertex, four types of vertices are possible. We indicate these below, along with their associated Boltzmann weights:
%

%\usepackage{pgfplots}
%\begin{tikzpicture}
%\draw[decorate, decoration={border,segment length=1mm,amplitude=5mm,angle=-135}, lgray] (0,0) --( 0,2);
%\draw[pattern=north west lines] (1,0) rectangle  (4,1);
%\end{tikzpicture}

\begin{align}
\label{eq:black-vertices}
\begin{array}{cccc}
\begin{tikzpicture}[scale=0.8,>=stealth]
\draw[lgray,ultra thick] (-1,0) -- (1,0);
\draw[lgray,line width=10pt] (0,-1) -- (0,1);
\node[below] at (0,-1) {$m$};
\draw[ultra thick,->,rounded corners] (-0.075,-1) -- (-0.075,1);
\draw[ultra thick,->,rounded corners] (0.075,-1) -- (0.075,1);
\node[above] at (0,1) {$m$};
\end{tikzpicture}
\quad\quad\quad
&
\begin{tikzpicture}[scale=0.8,>=stealth]
\draw[lgray,ultra thick] (-1,0) -- (1,0);
\draw[lgray,line width=10pt] (0,-1) -- (0,1);
\node[below] at (0,-1) {$m$};
\draw[ultra thick,->,rounded corners] (-0.075,-1) -- (-0.075,1);
\draw[ultra thick,->,rounded corners] (0.075,-1) -- (0.075,0) -- (1,0);
\node[above] at (0,1) {$m-1$};
\end{tikzpicture}
\quad\quad\quad
&
\begin{tikzpicture}[scale=0.8,>=stealth]
\draw[lgray,ultra thick] (-1,0) -- (1,0);
\draw[lgray,line width=10pt] (0,-1) -- (0,1);
\node[below] at (0,-1) {$m$};
\draw[ultra thick,->,rounded corners] (-1,0) -- (-0.15,0) -- (-0.15,1);
\draw[ultra thick,->,rounded corners] (0,-1) -- (0,1);
\draw[ultra thick,->,rounded corners] (0.15,-1) -- (0.15,1);
\node[above] at (0,1) {$m+1$};
\end{tikzpicture}
\quad\quad\quad
&
\begin{tikzpicture}[scale=0.8,>=stealth]
\draw[lgray,ultra thick] (-1,0) -- (1,0);
\draw[lgray,line width=10pt] (0,-1) -- (0,1);
\node[below] at (0,-1) {$m$};
\draw[ultra thick,->,rounded corners] (-1,0) -- (-0.15,0) -- (-0.15,1);
\draw[ultra thick,->,rounded corners] (0,-1) -- (0,1);
\draw[ultra thick,->,rounded corners] (0.15,-1) -- (0.15,0) -- (1,0);
\node[above] at (0,1) {$m$};
\end{tikzpicture}
\\
1
\quad\quad\quad
&
a
\quad\quad\quad
&
(1-t^{m+1})
\quad\quad\quad
&
a
\end{array}\end{align}
Here $m$ denotes the number of incoming vertical arrows,  $a$ is the horizontal parameter associated to the vertex, and $t$ is a global parameter of the model. The vertex model that comes from such a construction is called the {\it $t$-boson model}, and  we will call vertices with such weights {\it bosonic vertices}. It is ideal for our purposes since, on the one hand, the wavefunctions of the $t$-boson model are known to be Hall--Littlewood polynomials \cite{tsilevich2006quantum}, while on the other its integrability is intrinsically related to the stochastic six-vertex model. It is therefore a useful tool for bridging the two sides of \eqref{equiv_distrib}, that we wish to prove.

\tikzset{arr/.style={postaction={decorate,thick,decoration={markings,mark = at position #1 with {\arrow{>}}}}}}

\begin{theorem}
\label{thm:RLL}
For any fixed $0 \leq i_1,i_2,j_1,j_2 \leq 1$ and $m,n \in \mathbb{Z}_{\geq 0}$, the Yang--Baxter equation holds:
\begin{align}
\label{eq:RLL}
\sum_{k_1,k_2 \in \{0,1\}}\
\sum_{p=0}^{\infty}\ \ \
\begin{tikzpicture}[baseline=(current bounding box.center),scale=0.8]
\draw[dotted,thick] (-2,1) node[left] {$i_1$} -- (-1,0) node[below] {$k_1$};
\draw[dotted,thick] (-2,0) node[left] {$i_2$} -- (-1,1) node[above] {$k_2$};
\draw[lgray,ultra thick] (-1,1) -- (1,1) node[right,black] {$j_2$};
\draw[lgray,ultra thick] (-1,0) -- (1,0) node[right,black] {$j_1$};
\draw[lgray,line width=10pt] (0,-1) -- (0,2);
\node[below] at (0,-1) {$m$};
\node at (0,0.5) {$p$};
\node[above] at (0,2) {$n$};
\draw[thin,dashed,->] (0,0) -- (1,-1) node[right] {$b^{-1}$};
\draw[thin,dashed,->] (0,1) -- (1,2) node[right] {$a$};
\end{tikzpicture}
\quad
=
\quad
\sum_{k_1,k_2 \in \{0,1\}}\
\sum_{p=0}^{\infty}\ \ \
\begin{tikzpicture}[baseline=(current bounding box.center),scale=0.8]
\draw[dotted,thick] (1,1) node[above] {$k_1$} -- (2,0) node[right] {$j_1$};
\draw[dotted,thick] (1,0) node[below] {$k_2$} -- (2,1) node[right] {$j_2$};
\draw[lgray,ultra thick] (-1,1) node[left,black] {$i_1$} -- (1,1);
\draw[lgray,ultra thick] (-1,0) node[left,black] {$i_2$} -- (1,0);
\draw[lgray,line width=10pt] (0,-1) -- (0,2);
\node[below] at (0,-1) {$m$};
\node at (0,0.5) {$p$};
\node[above] at (0,2) {$n$};
\draw[thin,dashed,->] (0,0) -- (-1,-1) node[left] {$a$};
\draw[thin,dashed,->] (0,1) -- (-1,2) node[left] {$b^{-1}$};
\end{tikzpicture}
\end{align}
where the spectral parameters of the bosonic vertices are $a$ and $b^{-1}$ as indicated on the picture, and the diagonally attached vertices are vertices in the stochastic six-vertex model of \eqref{eq:6v} with parameters $a_x=a$ and $a_y=b$, rotated clockwise by 45 degrees.
\end{theorem}

\begin{proof}
This is by direct computation, since there are only sixteen relations to verify (all possible choices of $i_1,i_2,j_1,j_2$), treating $m$ and $n$ as arbitrary non-negative integers. The sums over $p$ are finite and easily taken, since the Boltzmann weight of the configuration (either on the left or right hand side) vanishes unless $|m-p|, |n-p| \leq 1$.
\end{proof}

It is important to introduce an alternative normalization of the vertex weights \eqref{eq:black-vertices}, obtained by sending $a \rightarrow b^{-1}$ and then simply multiplying weights of all vertices by $b$:
\begin{align}
\label{eq:red-vertices}
\begin{array}{cccc}
\begin{tikzpicture}[scale=0.8,>=stealth]
\draw[lred, ultra thick] (-1,0) -- (1,0);
\draw[lred,line width=10pt] (0,-1) -- (0,1);
\node[below] at (0,-1) {$m$};
\draw[ultra thick,->,rounded corners] (-0.075,-1) -- (-0.075,1);
\draw[ultra thick,->,rounded corners] (0.075,-1) -- (0.075,1);
\node[above] at (0,1) {$m$};
\end{tikzpicture}
\quad\quad\quad
&
\begin{tikzpicture}[scale=0.8,>=stealth]
\draw[lred,ultra thick] (-1,0) -- (1,0);
\draw[lred,line width=10pt] (0,-1) -- (0,1);
\node[below] at (0,-1) {$m$};
\draw[ultra thick,->,rounded corners] (-0.075,-1) -- (-0.075,1);
\draw[ultra thick,->,rounded corners] (0.075,-1) -- (0.075,0) -- (1,0);
\node[above] at (0,1) {$m-1$};
\end{tikzpicture}
\quad\quad\quad
&
\begin{tikzpicture}[scale=0.8,>=stealth]
\draw[lred,ultra thick] (-1,0) -- (1,0);
\draw[lred,line width=10pt] (0,-1) -- (0,1);
\node[below] at (0,-1) {$m$};
\draw[ultra thick,->,rounded corners] (-1,0) -- (-0.15,0) -- (-0.15,1);
\draw[ultra thick,->,rounded corners] (0,-1) -- (0,1);
\draw[ultra thick,->,rounded corners] (0.15,-1) -- (0.15,1);
\node[above] at (0,1) {$m+1$};
\end{tikzpicture}
\quad\quad\quad
&
\begin{tikzpicture}[scale=0.8,>=stealth]
\draw[lred,ultra thick] (-1,0) -- (1,0);
\draw[lred,line width=10pt] (0,-1) -- (0,1);
\node[below] at (0,-1) {$m$};
\draw[ultra thick,->,rounded corners] (-1,0) -- (-0.15,0) -- (-0.15,1);
\draw[ultra thick,->,rounded corners] (0,-1) -- (0,1);
\draw[ultra thick,->,rounded corners] (0.15,-1) -- (0.15,0) -- (1,0);
\node[above] at (0,1) {$m$};
\end{tikzpicture}
\\
b
\quad\quad\quad
&
1
\quad\quad\quad
&
b (1-t^{m+1})
\quad\quad\quad
&
1
\end{array}
\end{align}
We use a red background to indicate that this normalization is employed, rather than that of \eqref{eq:black-vertices}.

\subsection{Row-operators and their exchange relations}
\label{sec:row-ops}

For all integers $i \geq 1$, let $V_i$ be an infinite dimensional vector space with basis vectors $\{\ket{m}_i\}_{m \in \mathbb{Z}_{\geq0}}$. Its dual space $V_i^{*}$ is spanned by $\{\bra{m}_i\}_{m \in \mathbb{Z}_{\geq0}}$, where $\bra{m}_i \ket{n}_i = \delta_{m,n}$ for all $m,n \in \mathbb{Z}_{\geq 0}$. Further, we let $V_{1\dots L}$ denote the tensor product $\bigotimes_{i =1}^{L} V_i$. Joining $L$ of the vertices \eqref{eq:black-vertices} with common spectral parameter $a$ horizontally (likewise for the vertices \eqref{eq:red-vertices}), and summing over all possible states on internal horizontal edges, we obtain a {\it row vertex.} We denote the Boltzmann weight of a row vertex as shown below:
\begin{align}
\label{eq:row*-vert}
w_a
\left(
\begin{tikzpicture}[scale=0.6,baseline=(current bounding box.center)]
\draw[lgray,thick] (-1,0) -- (6,0);
\node[left] at (-0.8,0) {\tiny $i$};\node[right] at (5.8,0) {\tiny $j$};
\foreach\x in {0,...,5}{
\draw[lgray,line width=7pt] (\x,-1) -- (\x,1);
}
\node[below,text centered] at (0,-0.8) {\tiny $m_{L}$};\node[above] at (0,0.8) {\tiny $n_{L}$};
\node[below,text centered] at (3,-0.8) {$\cdots$};\node[above] at (3,0.8) {$\cdots$};
\node[below,text centered] at (5,-0.8) {\tiny $m_1$};\node[above] at (5,0.8) {\tiny $n_1$};
\end{tikzpicture}
\right)
&=:
w_a\Big(i,\{m_1,\dots,m_{L}\} \Big| j,\{n_1,\dots,n_{L}\}\Big),
\\
\label{eq:row-vert}
w_b
\left(
\begin{tikzpicture}[scale=0.6,baseline=(current bounding box.center)]
\draw[lred,thick] (-1,0) -- (6,0);
\node[left] at (-0.8,0) {\tiny $i$};\node[right] at (5.8,0) {\tiny $j$};
\foreach\x in {0,...,5}{
\draw[lred,line width=7pt] (\x,-1) -- (\x,1);
}
\node[below,text centered] at (0,-0.8) {\tiny $m_{L}$};\node[above] at (0,0.8) {\tiny $n_{L}$};
\node[below,text centered] at (3,-0.8) {$\cdots$};\node[above] at (3,0.8) {$\cdots$};
\node[below,text centered] at (5,-0.8) {\tiny $m_1$};\node[above] at (5,0.8) {\tiny $n_1$};
\end{tikzpicture}
\right)
&=:
\bar{w}_b\Big(i,\{m_1,\dots,m_{L}\} \Big| j,\{n_1,\dots,n_{L}\}\Big).
\end{align}
We then construct row-operators that act linearly on $V_{1\dots L}$ as follows:
\begin{align*}
T_a(i|j) : \ket{n_1}_1 \otimes \cdots \otimes \ket{n_{L}}_{L}
\mapsto
\sum_{m_1,\dots,m_{L} \geq 0}
w_a\Big(i,\{m_1,\dots,m_{L}\} \Big| j,\{n_1,\dots,n_{L}\}\Big)
\ket{m_1}_1 \otimes \cdots \otimes \ket{m_{L}}_{L},
\\
\bar{T}_b(i|j) : \ket{n_1}_1 \otimes \cdots \otimes \ket{n_{L}}_{L}
\mapsto
\sum_{m_1,\dots,m_{L} \geq 0}
\bar{w}_b\Big(i,\{m_1,\dots,m_{L}\} \Big| j,\{n_1,\dots,n_{L}\}\Big)
\ket{m_1}_1 \otimes \cdots \otimes \ket{m_{L}}_{L}.
\end{align*}
There are in total four such operators, corresponding to all possible values of $0 \leq i,j \leq 1$. It is more conventional to label them alphabetically, by writing
\begin{align}
\label{eq:row-ops}
\begin{pmatrix}
T_a(0|0) & T_a(0|1)
\\ \\
T_a(1|0) & T_a(1|1)
\end{pmatrix}
=:
\begin{pmatrix}
A_{L}(a) & B_{L}(a)
\\ \\
C_{L}(a) & D_{L}(a)
\end{pmatrix},
\quad
\begin{pmatrix}
\bar{T}_b(0|0) & \bar{T}_b(0|1)
\\ \\
\bar{T}_b(1|0) & \bar{T}_b(1|1)
\end{pmatrix}
=:
\begin{pmatrix}
\bar{A}_{L}(b) & \bar{B}_{L}(b)
\\ \\
\bar{C}_{L}(b) & \bar{D}_{L}(b)
\end{pmatrix}.
\end{align}
Now consider the limit $L \rightarrow \infty$, assuming that $|a|, \vert b\vert<1$. In this limit, $T_a(i|j)$ only remains finite when $i=0$. Indeed when $L$ tends to infinity, since only finitely many $m_i$'s and $m_j$'s will remain nonzero, one finds that $T_a(1|j)$ produces infinitely many of the last type of vertex appearing in \eqref{eq:black-vertices} with $m=0$, whose weight is  $a$, giving the row-operator a vanishingly small Boltzmann weight. Similarly, $\bar{T}_b(i|j)$ only remains finite when $i=1$. As for the cases which have a non-vanishing limit, we find it convenient to define
\begin{align*}
A(a) := \lim_{L \rightarrow \infty} A_{L}(a),\quad
B(a) := \lim_{L \rightarrow \infty} B_{L}(a),\quad
\bar{C}(b) := \lim_{L \rightarrow \infty} \bar{C}_{L}(b),\quad
\bar{D}(b) := \lim_{L \rightarrow \infty} \bar{D}_{L}(b).
\end{align*}
\begin{proposition}\label{prop:graphicalCauchy}
Let $a,b$ be two complex parameters satisfying $|ab| < 1$. The following exchange relations hold:
\begin{equation}
\label{eq:exchange2}
\begin{aligned}
(1-ab)
\bar{C}(b) A(a)
&=
(1-tab)
A(a) \bar{C}(b),
\\
(1-ab)
\bar{C}(b) B(a)
&=
t(1-ab)
B(a) \bar{C}(b)
+
ab (1-t)
A(a) \bar{D}(b),
\\
(1-ab)
\bar{D}(b) A(a)
&=
(1-ab)
A(a) \bar{D}(b)
+
(1-t)
B(a) \bar{C}(b),
\\
(1-ab)
\bar{D}(b) B(a)
&=
(1-tab)
B(a) \bar{D}(b).
\end{aligned}
\end{equation}
Graphically, we can write all of these equations in the form
\begin{multline}
\label{graph-exchange}
\left(
\frac{1-a b}{1-t a b}
\right)
\sum_{p_1,p_2,\dots \geq 0}\ \ \
\begin{tikzpicture}[baseline=(current bounding box.center),>=stealth,scale=0.7]
\draw[lgray,ultra thick] (-1,1) node[left,black] {$a$}
-- (4,1) node[right,black] {$j_2$};
\draw[lred,ultra thick] (-1,0) node[left,black] {$b$}
-- (4,0) node[right,black] {$j_1$};
\foreach\x in {0,...,3}{
\draw[lgray,line width=10pt] (3-\x,0.5) -- (3-\x,2);
\draw[lred,line width=10pt] (3-\x,-1) -- (3-\x,0.5);
}
\node[below] at (3,-1) {$m_1$};
\node at (3,0.5) {$p_1$};
\node[above] at (3,2) {$n_1$};
\node[below] at (2,-1) {$m_2$};
\node at (2,0.5) {$p_2$};
\node[above] at (2,2) {$n_2$};
\node[text centered] at (0,0.5) {$\cdots$};
\node[text centered] at (1,0.5) {$\cdots$};
\draw[ultra thick,->] (-1,0) -- (0,0);
\end{tikzpicture}
\quad
=
\\
\quad
\sum_{ k_1,k_2 \in\lbrace 0,1\rbrace}\
\sum_{p_1,p_2,\dots \geq 0}\ \ \
\begin{tikzpicture}[baseline=(current bounding box.center),>=stealth,scale=0.7]
\foreach\x in {0,...,3}{
\draw[lgray,line width=10pt] (3-\x,-1) -- (3-\x,0.5);
\draw[lred,line width=10pt] (3-\x,0.5) -- (3-\x,2);
}
\draw[lred,ultra thick] (-1,1) node[left,black] {$b$}
-- (4,1);
\draw[lgray,ultra thick] (-1,0) node[left,black] {$a$}
-- (4,0);
\draw[dotted,thick] (4,1) node[above] {$k_1$} -- (5,0) node[right] {$j_1$};
\draw[dotted,thick] (4,0) node[below] {$k_2$} -- (5,1) node[right] {$j_2$};
\node[below] at (3,-1) {$m_1$};
\node at (3,0.5) {$p_1$};
\node[above] at (3,2) {$n_1$};
\node[below] at (2,-1) {$m_2$};
\node at (2,0.5) {$p_2$};
\node[above] at (2,2) {$n_2$};
\node[text centered] at (0,0.5) {$\cdots$};
\node[text centered] at (1,0.5) {$\cdots$};
\draw[ultra thick,->] (-1,1) -- (0,1);
\end{tikzpicture}
\end{multline}
where bosonic vertices on the same row have the same spectral parameter (as indicated on the left of each row). The four possibilities in \eqref{eq:exchange2} given by the four possible choices of $j_1,j_2 \in \{0,1\}$ in \eqref{graph-exchange}.
\end{proposition}

\begin{proof}
These identities are all relations in the Yang--Baxter algebra satisfied by the matrix entries $T_a(i|j)$ and $\bar{T}_a(i|j)$, in the limit $L \rightarrow \infty$, assuming the parameters $a,b$ satisfy $|a b| <1$. For more details of their derivation, using the same notations as in the present paper, see \cite{borodin2016between}.
\end{proof}

\subsection{A boundary relation}

In this section we note another property of the $t$-boson model, namely a reflection equation that it satisfies with respect to a particular choice of boundary \cite{wheeler2016refined}.

\begin{proposition}
\label{prop:local_gz}
Let $n \geq 0$ be any non-negative integer, and fix $i,j \in \{0,1\}$. The following identity holds:
\begin{align}
\label{eq:local_gz}
\sum_{m=0}^{\infty}
\prod_{k=1}^{m}
(1-t^{2k-1})
\times
w_a\left(
\begin{tikzpicture}[scale=0.7,>=stealth,baseline=(current bounding box.center)]
\draw[lgray,ultra thick] (-1.5,0) -- (1,0);
\draw[lgray,line width=10pt] (0,-1) -- (0,1);
\node[below] at (0,-1) {$2m$};
\node[above] at (0,1) {$n$};
\node[left] at (-1.5,0) {$i$};
\node[right] at (1,0) {$j$};
\node at (-1,0) {$\bullet$};
\end{tikzpicture}
\right)
=
\sum_{m=0}^{\infty}
\prod_{k=1}^{m}
(1-t^{2k-1})
\times
w_a\left(
\begin{tikzpicture}[scale=0.7,>=stealth,baseline=(current bounding box.center)]
\draw[lred,ultra thick] (-1,0) -- (1.5,0);
\draw[lred,line width=10pt] (0,-1) -- (0,1);
\node[below] at (0,-1) {$2m$};
\node[above] at (0,1) {$n$};
\node[left] at (-1,0) {$i$};
\node[right] at (1.5,0) {$j$};
\node at (1,0) {$\bullet$};
\end{tikzpicture}
\right)
\end{align}
where the dot has a path annihilating/creating property: if a rightward pointing path approaches the dot from its left, no path will emerge to the right of the dot (similarly, if no path approaches the dot from its left, a rightward pointing path will emerge to the right of the dot).
\end{proposition}

\begin{proof}
We need to check the four possible values for the pair $(i,j)$. In each case, the value of $2m$ at the base of the vertex is completely determined by path conservation, so the infinite sums over $m$ trivialize. Below we list the four cases.

\underline{\it Case 1: $i=j=0$.} On the left hand side of \eqref{eq:local_gz}, the vertex vanishes unless $2m+1 = n$. On the right hand side, the vertex vanishes unless $2m = n+1$. In either case, we see that $n$ must be odd, otherwise both sides vanish identically. When $n$ is odd, we have
\begin{align*}
\prod_{k=1}^{(n-1)/2}
(1-t^{2k-1})
\times
w_a\left(
\begin{tikzpicture}[scale=0.7,>=stealth,baseline=(current bounding box.center)]
\draw[lgray,ultra thick] (-1.5,0) -- (1,0);
\draw[lgray,line width=10pt] (0,-1) -- (0,1);
\draw[ultra thick,->,rounded corners] (-1,0) -- (0,0) -- (0,1);
\node[below] at (0,-1) {$n-1$};
\node[above] at (0,1) {$n$};
\node[left] at (-1.5,0) {$0$};
\node[right] at (1,0) {$0$};
\node at (-1,0) {$\bullet$};
\end{tikzpicture}
\right)
=
\prod_{k=1}^{(n+1)/2}
(1-t^{2k-1})
\times
w_a\left(
\begin{tikzpicture}[scale=0.7,>=stealth,baseline=(current bounding box.center)]
\draw[lred,ultra thick] (-1,0) -- (1.5,0);
\draw[lred,line width=10pt] (0,-1) -- (0,1);
\draw[ultra thick,->,rounded corners] (-0.075,-1) -- (-0.075,1);
\draw[ultra thick,->,rounded corners] (0.075,-1) -- (0.075,0) -- (1,0);
\node[below] at (0,-1) {$n+1$};
\node[above] at (0,1) {$n$};
\node[left] at (-1,0) {$0$};
\node[right] at (1.5,0) {$0$};
\node at (1,0) {$\bullet$};
\end{tikzpicture}
\right),
\end{align*}
with the equality obviously holding thanks to the vertex weights \eqref{eq:black-vertices} and \eqref{eq:red-vertices}.

\underline{\it Case 2: $i=0, j=1$.} By conservation of paths, the left and right hand sides of \eqref{eq:local_gz} vanish unless $2m=n$. The equation then reads
\begin{align*}
\prod_{k=1}^{n/2}
(1-t^{2k-1})
\times
w_a\left(
\begin{tikzpicture}[scale=0.7,>=stealth,baseline=(current bounding box.center)]
\draw[lgray,ultra thick] (-1.5,0) -- (1,0);
\draw[lgray,line width=10pt] (0,-1) -- (0,1);
\draw[ultra thick,->,rounded corners] (-1,0) -- (-0.15,0) -- (-0.15,1);
\draw[ultra thick,->,rounded corners] (0,-1) -- (0,1);
\draw[ultra thick,->,rounded corners] (0.15,-1) -- (0.15,0) -- (1,0);
\node[below] at (0,-1) {$n$};
\node[above] at (0,1) {$n$};
\node[left] at (-1.5,0) {$0$};
\node[right] at (1,0) {$1$};
\node at (-1,0) {$\bullet$};
\end{tikzpicture}
\right)
=
\prod_{k=1}^{n/2}
(1-t^{2k-1})
\times
w_a\left(
\begin{tikzpicture}[scale=0.7,>=stealth,baseline=(current bounding box.center)]
\draw[lred,ultra thick] (-1,0) -- (1.5,0);
\draw[lred,line width=10pt] (0,-1) -- (0,1);
\draw[ultra thick,->,rounded corners] (-0.075,-1) -- (-0.075,1);
\draw[ultra thick,->,rounded corners] (0.075,-1) -- (0.075,1);
\draw[ultra thick,->,rounded corners] (1,0) -- (1.5,0);
\node[below] at (0,-1) {$n$};
\node[above] at (0,1) {$n$};
\node[left] at (-1,0) {$0$};
\node[right] at (1.5,0) {$1$};
\node at (1,0) {$\bullet$};
\end{tikzpicture}
\right),
\end{align*}
and the equality of the two sides is immediate since both vertices have weight equal to $a$.

\underline{\it Case 3: $i=1, j=0$.} Similarly to case 2, one finds that the left and right hand sides of \eqref{eq:local_gz} vanish unless $2m=n$. This yields the identity
\begin{align*}
\prod_{k=1}^{n/2}
(1-t^{2k-1})
\times
w_a\left(
\begin{tikzpicture}[scale=0.7,>=stealth,baseline=(current bounding box.center)]
\draw[lgray,ultra thick] (-1.5,0) -- (1,0);
\draw[lgray,line width=10pt] (0,-1) -- (0,1);
\draw[ultra thick,->,rounded corners] (-0.075,-1) -- (-0.075,1);
\draw[ultra thick,->,rounded corners] (0.075,-1) -- (0.075,1);
\draw[ultra thick,->,rounded corners] (-1.5,0) -- (-1,0);
\node[below] at (0,-1) {$n$};
\node[above] at (0,1) {$n$};
\node[left] at (-1.5,0) {$1$};
\node[right] at (1,0) {$0$};
\node at (-1,0) {$\bullet$};
\end{tikzpicture}
\right)
=
\prod_{k=1}^{n/2}
(1-t^{2k-1})
\times
w_a\left(
\begin{tikzpicture}[scale=0.7,>=stealth,baseline=(current bounding box.center)]
\draw[lred,ultra thick] (-1,0) -- (1.5,0);
\draw[lred,line width=10pt] (0,-1) -- (0,1);
\draw[ultra thick,->,rounded corners] (-1,0) -- (-0.15,0) -- (-0.15,1);
\draw[ultra thick,->,rounded corners] (0,-1) -- (0,1);
\draw[ultra thick,->,rounded corners] (0.15,-1) -- (0.15,0) -- (1,0);
\node[below] at (0,-1) {$n$};
\node[above] at (0,1) {$n$};
\node[left] at (-1,0) {$1$};
\node[right] at (1.5,0) {$0$};
\node at (1,0) {$\bullet$};
\end{tikzpicture}
\right),
\end{align*}
which holds because the vertices on both sides have Boltzmann weight equal to $1$.

\underline{\it Case 4: $i=j=1$.} This case is analogous to case 1. Again we find that both sides of \eqref{eq:local_gz} vanish unless $n$ is odd, and in the situation where $n$ is odd one has
\begin{align*}
\prod_{k=1}^{(n+1)/2}
(1-t^{2k-1})
\times
w_a\left(
\begin{tikzpicture}[scale=0.7,>=stealth,baseline=(current bounding box.center)]
\draw[lgray,ultra thick] (-1.5,0) -- (1,0);
\draw[lgray,line width=10pt] (0,-1) -- (0,1);
\draw[ultra thick,->,rounded corners] (-0.075,-1) -- (-0.075,1);
\draw[ultra thick,->,rounded corners] (0.075,-1) -- (0.075,0) -- (1,0);
\draw[ultra thick,->,rounded corners] (-1.5,0) -- (-1,0);
\node[below] at (0,-1) {$n+1$};
\node[above] at (0,1) {$n$};
\node[left] at (-1.5,0) {$1$};
\node[right] at (1,0) {$1$};
\node at (-1,0) {$\bullet$};
\end{tikzpicture}
\right)
=
\prod_{k=1}^{(n-1)/2}
(1-t^{2k-1})
\times
w_a\left(
\begin{tikzpicture}[scale=0.7,>=stealth,baseline=(current bounding box.center)]
\draw[lred,ultra thick] (-1,0) -- (1.5,0);
\draw[lred,line width=10pt] (0,-1) -- (0,1);
\draw[ultra thick,->,rounded corners] (-1,0) -- (0,0) -- (0,1);
\draw[ultra thick,->,rounded corners] (1,0) -- (1.5,0);
\node[below] at (0,-1) {$n-1$};
\node[above] at (0,1) {$n$};
\node[left] at (-1,0) {$1$};
\node[right] at (1.5,0) {$1$};
\node at (1,0) {$\bullet$};
\end{tikzpicture}
\right),
\end{align*}
where the equality can be easily checked using the Boltzmann weights \eqref{eq:black-vertices} and \eqref{eq:red-vertices}.
\end{proof}
For a partition $\lambda=1^{m_1}2^{m_2}\dots$, we introduce  the shorthand notations
$$\bra{\lambda} = \bigotimes_{i=1}^{\infty}\bra{m_i}_i, \ \ \ \ket{\la} = \bigotimes_{i=1}^{\infty}\ket{m_i}_i.$$
Proposition \ref{prop:local_gz} is a local relation on $t$-boson vertices, which can be extended to a global relation in the following way.

\begin{proposition}
Define a boundary covector
\begin{align*}
\bra{\rm ev}
:=
\sum_{\lambda' \ \text{even}}
\bel{\lambda}
\bra{\lambda}
=
\ssum_{m_1,m_2,\cdots}\
\prod_{i=1}^{\infty}
\prod_{k=1}^{m_i}
(1-t^{2k-1})
\times
\bigotimes_{i=1}^{\infty}
\bra{2m_i}_i,
\end{align*}
where the sum is over partitions $\la=1^{2m_1}2^{2m_2}\dots$ so that  only finitely many $m_j$'s are nonzero. We will use the notation $ \sum\limits^{*}$ to denote summations over infinite sequences of nonnegative integers, finitely many of which are nonzero.
The boundary covector satisfies the following reflection equations:
\begin{align}
\label{eq:global_gz}
\bra{\rm ev} A(a) = \bra{\rm ev} \bar{D}(a),
\quad
\bra{\rm ev} B(a) = \bra{\rm ev} \bar{C}(a).
\end{align}
\label{prop:gzrelation}
\end{proposition}

\begin{proof}
We need to check all possible components of the equations \eqref{eq:global_gz}, by projecting onto the arbitrary state
$\ket{n} := \otimes_{i=1}^{\infty} \ket{n_i}_i$. Let us consider firstly the proposed equation
$\bra{\rm ev} A(a) \ket{n} = \bra{\rm ev} \bar{D}(a) \ket{n}$, which when expressed pictorially reads
\begin{multline}
\label{dot-transfer}
\ssum_{m_1,m_2,\cdots}\
\prod_{i=1}^{\infty}
\prod_{k=1}^{m_i}
(1-t^{2k-1})
\times
w_a
\left(
\begin{tikzpicture}[baseline=(current bounding box.center),>=stealth,scale=0.8]
\node[left] (0,0) {$1$};
\draw[lgray,ultra thick] (0,0) -- (8,0);
\draw[ultra thick,->] (0,0) -- (1,0);
\foreach\x in {0,...,5}{
\draw[lgray,line width=10pt] (7-\x,-1) -- (7-\x,1);
}
\node at (1,0) {$\bullet$};
\node[right] at (8,0) {$0$};
\foreach\x in {1,2,3}{
\node[text centered,below] at (8-\x,-1) {\tiny $2m_{\x}$};
\node[text centered,above] at (8-\x,1) {\tiny $n_{\x}$};
}
\foreach\x in {3,4}{
\node[text centered,below] at (7-\x,-1) {\tiny $\cdots$};
\node[text centered,above] at (7-\x,1) {\tiny $\cdots$};
}
\end{tikzpicture}
\right)
\\
=
\ssum_{m_1,m_2,\cdots }\
\prod_{i=1}^{\infty}
\prod_{k=1}^{m_i}
(1-t^{2k-1})
\times
w_a
\left(
\begin{tikzpicture}[baseline=(current bounding box.center),>=stealth,scale=0.8]
\draw[lred,ultra thick] (0,0) -- (8,0);
\foreach\x in {1,...,6}{
\draw[lred,line width=10pt] (7-\x,-1) -- (7-\x,1);
}
\node[left] at (0,0) {$1$};
\node at (7,0) {$\bullet$};
\draw[ultra thick,->] (0,0) -- (1,0);
\draw[ultra thick,->] (6,0) -- (7,0);
\node[right] at (8,0) {$0$};
\foreach\x in {1,2,3}{
\node[text centered,below] at (7-\x,-1) {\tiny $2m_{\x}$};
\node[text centered,above] at (7-\x,1) {\tiny $n_{\x}$};
}
\foreach\x in {4,5}{
\node[text centered,below] at (7-\x,-1) {\tiny $\cdots$};
\node[text centered,above] at (7-\x,1) {\tiny $\cdots$};
}
\end{tikzpicture}
\right).
\end{multline}
This relation is now seen to be true by infinitely many applications of \eqref{eq:local_gz} (there will be finitely many nontrivial ones as only finitely many $m_i$'s and $n_j$'s are nonzero), transferring the dot at the left edge of the lattice all the way to the right edge. The proof of the relation $\bra{\rm ev} B(a) \ket{n} = \bra{\rm ev} \bar{C}(a) \ket{n}$ follows by the same argument; one needs only to replace the state $0$ on the right edge of the partition functions in \eqref{dot-transfer} by the state $1$.

\end{proof}

\subsection{One variable skew Hall--Littlewood polynomials}
Comparing the Boltzmann weights \eqref{eq:black-vertices} and \eqref{eq:red-vertices} used in the row-to-row operators with the explicit form of the one-variable skew Hall--Littlewood polynomials, one obtains the following lemma.
\begin{lemma}[Lemma 5.3 \cite{borodin2016between}]
The matrix elements of the operators $A(a),B(a),\bar{C}(b),\bar{D}(b)$ are one-variable skew Hall--Littlewood polynomials:
\begin{align}
\label{P-one-row}
&
\bra{\lambda} A(a) \ket{\mu}
=
\left( \mathds{1}_{\ell(\lambda) = \ell(\mu)} \right)
P_{\lambda/\mu}(a),
\quad
&&
\bra{\lambda} B(a) \ket{\mu}
=
\left( \mathds{1}_{\ell(\lambda) = \ell(\mu)+1} \right)
P_{\lambda/\mu}(a),
\\
\label{Q-one-row}
&
\bra{\mu} \bar{C}(b) \ket{\lambda}
=
\left( \mathds{1}_{\ell(\lambda) = \ell(\mu)+1} \right)
Q_{\lambda/\mu}(b),
\quad
&&
\bra{\mu} \bar{D}(b) \ket{\lambda}
=
\left( \mathds{1}_{\ell(\lambda) = \ell(\mu)} \right)
Q_{\lambda/\mu}(b),
\end{align}
where $\lambda$ and $\mu$ are any two partitions.
\label{lem:partitionrowLittlewood}
\end{lemma}

%\begin{proof}
%These relations follow immediately by comparing the Boltzmann weights \eqref{eq:black-vertices} and \eqref{eq:red-vertices} used in the row-to-row operators with the explicit form of the one-variable skew Hall--Littlewood polynomials \cite[Chapter III, Section 5]{macdonald1995symmetric}:
%%
%\begin{align*}
%P_{\lambda/\mu}(a)
%=
%\mathds{1}_{\lambda \succ \mu}
%\cdot
%a^{|\lambda|-|\mu|}
%\prod_{i: m_i(\lambda)+1 = m_i(\mu)}
%(1-t^{m_i(\mu)}),
%\end{align*}
%
%\begin{align*}
%Q_{\lambda/\mu}(b)
%=
%\mathds{1}_{\lambda \succ \mu}
%\cdot
%b^{|\lambda|-|\mu|}
%\prod_{i: m_i(\lambda) = m_i(\mu)+1}
%(1-t^{m_i(\lambda)}).
%\end{align*}
%%
%The additional constraints on the lengths of the partitions in \eqref{P-one-row}--\eqref{Q-one-row} are a consequence of up-right path-conservation.
%\end{proof}
\begin{remark}
In light of Lemma \ref{lem:partitionrowLittlewood}, the exchange relation in Proposition \ref{prop:graphicalCauchy} corresponds to the skew Cauchy identity \eqref{eq:SkewCauchy} in the case $q=0$ when Hall-Littlewood polynomials are evaluated in a single variable,  while the exchange relation in Proposition \ref{prop:gzrelation} corresponds to the skew Littlewood identity \eqref{eq:skewLittlewood}.
\end{remark}
\begin{proposition}
Let $S=(s_1,\dots,s_n)$ denote a binary string, with each $s_k \in \{0,1\}$. The support distribution (defined in Theorem \ref{th:matchingHL6V}) in the ascending half-space Hall--Littlewood process can be written in the form
\begin{align}
\label{eq:distrib_exp}
\PHL\left(
[\lambda^{(1)} \subset \cdots \subset \lambda^{(n)}] = (s_1,\dots,s_n)
\right)
=
\frac{1}{\Phi(a_1,\dots,a_n)}
\times
\Big\langle {\rm ev} \Big|
\prod_{k=1}^{n}
O_k^S(a_k)
\Big| \varnothing \Big\rangle,
\end{align}
where we have defined
\begin{align*}
O_k^{S}(a)
=
\left\{
\begin{array}{ll}
A(a), \quad & s_k = 0,
\\
B(a), \quad & s_k = 1,
\end{array}
\right.
\end{align*}
and where the product in \eqref{eq:distrib_exp} is ordered from right to left as the index $k$ increases.
\end{proposition}

\begin{proof}
We insert the identity $\sum_{\lambda^{(k)}} \ket{\lambda^{(k)}}\bra{\lambda^{(k)}}$ at the left of each operator $O_k^S(a_k)$, for all $1 \leq k \leq n$. Using the formulae \eqref{P-one-row} for the one-variable skew Hall--Littlewood polynomials, this produces the sum
\begin{align*}
\Big\langle {\rm ev} \Big|
\prod_{k=1}^{n}
O_k^S(a_k)
\Big| \varnothing \Big\rangle
=
\sum_{\lambda^{(1)} \subset \cdots \subset \lambda^{(n)} = \lambda}
\mathds{1}_{\lambda'\ \text{even}}
\cdot
b^{\rm el}_{\lambda}
\prod_{k=1}^{n}
\left(
P_{\lambda^{(k)} / \lambda^{(k-1)}}(a_k)
\mathds{1}_{\ell(\lambda^{(k)}) - \ell(\lambda^{(k-1)}) = s_k}
\right),
\end{align*}
which when divided by $\Phi(a_1,\dots,a_n)$ recovers precisely the claimed distribution.
\end{proof}

\subsection{Equivalence between support and path string distributions}

We are now ready to prove Theorem \ref{th:matchingHL6V}, the direct equivalence between the support distribution in the ascending half-space Hall--Littlewood process and the path string distribution in the half-space stochastic six-vertex model.

\begin{proof}[Proof of Theorem \ref{th:matchingHL6V}]
We start by expressing the support distribution as an expectation value in the $t$-boson model, as in \eqref{eq:distrib_exp}. Graphically, this relation takes the form
\begin{multline}
\PHL
\left(
[\lambda^{(1)} \subset \cdots \subset \lambda^{(n)}] = (s_1,\dots,s_n)
\right)
=
\\
\prod_{1 \leq i<j \leq n}
\left(
\frac{1-a_i a_j}{1-t a_i a_j}
\right)
\times
\ssum_{m_1,m_2,\cdots}\
\prod_{i=1}^{\infty}
\prod_{k=1}^{m_i}
(1-t^{2k-1})
\begin{tikzpicture}[scale=0.7,baseline=(current bounding box.center),>=stealth]
%lattice
\foreach\x in {0,...,6}{
\draw[lgray,line width=10pt] (\x,0) -- (\x,7);
}
\foreach\y in {1,...,6}{
\draw[lgray,thick] (-1,\y) -- (7,\y);
}
%black paths
\draw[ultra thick,->] (6,1) -- (7,1); \draw[ultra thick,->] (6,2) -- (7,2);
\draw[ultra thick,->] (6,4) -- (7,4); \draw[ultra thick,->] (6,6) -- (7,6);
\node[right] at (7,1) {$s_n$}; \node[right] at (7,6) {$s_1$};
%labels
\node[left] at (-1,6) {$a_1$};
\node[left] at (-1,3.5) {$\vdots$};
\node[left] at (8,3.5) {$\vdots$};
\node[left] at (-1,1) {$a_n$};
\node[below] at (6,0) {\tiny $2m_1$};
\node[below] at (5,0) {\tiny $2m_2$};
\node[below] at (4,0) {\tiny $2m_3$};
\node[below] at (2,0) {$\cdots$};
\node[above] at (6,7) {\tiny $0$};
\node[above] at (5,7) {\tiny $0$};
\node[above] at (4,7) {\tiny $0$};
\node[above] at (2,7) {$\cdots$};
\end{tikzpicture}
\end{multline}
where  the boundary conditions of the partition function are as follows: {\bf 1.} No paths enter from the left edge of the lattice, which is considered to be infinitely far to the left. {\bf 2.} $2m_i$ paths enter from the bottom of the $i$-th column, where columns are counted from right to left. {\bf 3.} $s_i\in \lbrace 0,1\rbrace$ paths leave via the right boundary of the $i$-th row, where rows are counted from top to bottom. {\bf 4.} No paths leave from the top of the lattice.

Now consider the $n$-th row of the lattice. Depending on the value of $s_n$, it either encodes the operator $A(a_n)$ (in the case $s_n=0$) or the operator $B(a_n)$ (in the case $s_n=1$). Irrespective of the value of $s_n$, we will be able to apply one of the relations \eqref{eq:global_gz} to convert this operator into $\bar{D}(a_n)$ or $\bar{C}(a_n)$. This takes us to the partition function
\begin{multline}
\label{eq:proof1}
\PHL
\left(
[\lambda^{(1)} \subset \cdots \subset \lambda^{(n)}] = (s_1,\dots,s_n)
\right)
=
\\
\prod_{1 \leq i<j \leq n}
\left(
\frac{1-a_i a_j}{1-t a_i a_j}
\right)
\times
\ssum_{m_1,m_2,\cdots}\
\prod_{i=1}^{\infty}
\prod_{k=1}^{m_i}
(1-t^{2k-1})
\begin{tikzpicture}[scale=0.7,baseline=(current bounding box.center),>=stealth]
%lattice
\foreach\x in {0,...,6}{
\draw[lgray,line width=10pt] (\x,1) -- (\x,7);
}
\foreach\y in {2,...,6}{
\draw[lgray,thick] (-1,\y) -- (7,\y);
}
\foreach\x in {0,...,6}{
\draw[lred,line width=10pt] (\x,-1) -- (\x,1);
}
\foreach\y in {0}{
\draw[lred,thick] (-1,\y) -- (8,\y);
}
\node at (7,0) {$\bullet$};
%black paths
\draw[ultra thick,->] (-1,0) -- (0,0);
\draw[ultra thick,->] (7,0) -- (8,0); \draw[ultra thick,->] (6,2) -- (7,2);
\draw[ultra thick,->] (6,4) -- (7,4); \draw[ultra thick,->] (6,6) -- (7,6);
\node[right] at (7,2) {$s_{n-1}$};
\node[right] at (8,0) {$s_n$}; \node[right] at (7,6) {$s_1$};
%labels
\node[left] at (-1,6) {$a_1$};
\node[left] at (-1,3.5) {$\vdots$};
\node[left] at (8,3.5) {$\vdots$};
\node[left] at (-1,2) {$a_{n-1}$};
\node[left] at (-1,0) {$a_n$};
\node[below] at (6,-1) {\tiny $2m_1$};
\node[below] at (5,-1) {\tiny $2m_2$};
\node[below] at (4,-1) {\tiny $2m_3$};
\node[below] at (2,-1) {$\cdots$};
\node[above] at (6,7) {\tiny $0$};
\node[above] at (5,7) {\tiny $0$};
\node[above] at (4,7) {\tiny $0$};
\node[above] at (2,7) {$\cdots$};
\end{tikzpicture}
\end{multline}
Noting that the factor $\prod_{i=1}^{n-1} (1-a_i a_n)/(1-t a_i a_n)$ is present in  \eqref{eq:proof1}, we can now use the relation \eqref{graph-exchange} $n$ times to transfer the red row to the top of the lattice. The result of this procedure is
\begin{multline}
\label{eq:proof2}
\PHL
\left(
[\lambda^{(1)} \subset \cdots \subset \lambda^{(n)}] = (s_1,\dots,s_n)
\right)
=
\prod_{1 \leq i<j \leq n-1}
\left(
\frac{1-a_i a_j}{1-t a_i a_j}
\right)
\times
\\
\ssum_{m_1,m_2,\cdots}\
\prod_{i=1}^{\infty}
\prod_{k=1}^{m_i}
(1-t^{2k-1})
\begin{tikzpicture}[scale=0.7,baseline=(current bounding box.center),>=stealth]
%lattice
\foreach\x in {0,...,6}{
\draw[lgray,line width=10pt] (\x,-1) -- (\x,5);
}
\foreach\y in {0,...,4}{
\draw[lgray,thick] (-1,\y) -- (7,\y);
}
\foreach\x in {0,...,6}{
\draw[lred,line width=10pt] (\x,5) -- (\x,7);
}
\foreach\y in {6}{
\draw[lred,thick] (-1,\y) -- (7,\y);
}
%side lattice
\draw[thick, dotted] (7,6) -- (10.5,2.5);
\node at (10.5,2.5) {$\bullet$}; \draw[ultra thick,->] (10.5,2.5) -- (11,3);
\draw[thick, dotted] (7,4) -- (8.5,5.5) node[above right] {$s_1$};
\draw[thick, dotted] (7,3) -- (9,5);
\draw[thick, dotted] (7,2) -- (9.5,4.5);
\draw[thick, dotted] (7,1) -- (10,4);
\draw[thick, dotted] (7,0) -- (10.5,3.5) node[above right] {$s_{n-1}$};
\draw[thick, dotted] (10.5,2.5) -- (11,3) node[above right] {$s_n$};
\draw[ultra thick,->] (8,5) -- (8.5,5.5); \draw[ultra thick,->] (9,4) -- (9.5,4.5);
\draw[ultra thick,->] (10,3) -- (10.5,3.5);
%black paths
\draw[ultra thick,->] (-1,6) -- (0,6);
%labels
\node[left] at (-1,4) {$a_1$};
\node[left] at (-1,1.5) {$\vdots$};
\node[left] at (-1,0) {$a_{n-1}$};
\node[left] at (-1,6) {$a_n$};
\node[below] at (6,-1) {\tiny $2m_1$};
\node[below] at (5,-1) {\tiny $2m_2$};
\node[below] at (4,-1) {\tiny $2m_3$};
\node[below] at (2,-1) {$\cdots$};
\node[above] at (6,7) {\tiny $0$};
\node[above] at (5,7) {\tiny $0$};
\node[above] at (4,7) {\tiny $0$};
\node[above] at (2,7) {$\cdots$};
\end{tikzpicture}
\end{multline}
We can then iterate the steps in \eqref{eq:proof1} and \eqref{eq:proof2} a further $n-1$ times. This converts all grey rows into red ones:
\begin{multline}
\PHL
\left(
[\lambda^{(1)} \subset \cdots \subset \lambda^{(n)}] = (s_1,\dots,s_n)
\right)
=
\\
\ssum_{m_1,m_2,\cdots}\
\prod_{i=1}^{\infty}
\prod_{k=1}^{m_i}
(1-t^{2k-1})
\begin{tikzpicture}[scale=0.7,baseline=(current bounding box.center),>=stealth]
%lattice
\foreach\x in {0,...,6}{
\draw[lred,line width=10pt] (\x,0) -- (\x,7);
}
\foreach\y in {1,...,6}{
\draw[lred,thick] (-1,\y) -- (7,\y);
\draw[ultra thick,->] (-1,\y) -- (0,\y);
}
%side lattice
\draw[thick, dotted] (7,6) -- (12.5,0.5); \node at (12.5,0.5) {$\bullet$};
\draw[thick, dotted] (12.5,0.5) -- (13,1) node[above right] {$s_n$}; \draw[ultra thick,->] (12.5,0.5) -- (13,1);
\draw[thick, dotted] (7,5) -- (11.5,0.5); \node at (11.5,0.5) {$\bullet$};
\draw[thick, dotted] (11.5,0.5) -- (12.5,1.5); \draw[ultra thick,->] (12,1) -- (12.5,1.5);
\draw[thick, dotted] (7,4) -- (10.5,0.5); \node at (10.5,0.5) {$\bullet$};
\draw[thick, dotted] (10.5,0.5) -- (12,2);
\draw[thick, dotted] (7,3) -- (9.5,0.5); \node at (9.5,0.5) {$\bullet$};
\draw[thick, dotted] (9.5,0.5) -- (11.5,2.5); \draw[ultra thick,->] (11,2) -- (11.5,2.5);
\draw[thick, dotted] (7,2) -- (8.5,0.5); \node at (8.5,0.5) {$\bullet$};
\draw[thick, dotted] (8.5,0.5) -- (11,3);
\draw[thick, dotted] (7,1) -- (7.5,0.5); \node at (7.5,0.5) {$\bullet$};
\draw[thick, dotted] (7.5,0.5) -- (10.5,3.5) node[above right] {$s_1$}; \draw[ultra thick,->] (10,3) -- (10.5,3.5);
%labels
\node[left] at (-1,6) {$a_n$};
\node[left] at (-1,3.5) {$\vdots$};
\node[left] at (-1,1) {$a_1$};
\node[below] at (6,0) {\tiny $2m_1$};
\node[below] at (5,0) {\tiny $2m_2$};
\node[below] at (4,0) {\tiny $2m_3$};
\node[below] at (2,0) {$\cdots$};
\node[above] at (6,7) {\tiny $0$};
\node[above] at (5,7) {\tiny $0$};
\node[above] at (4,7) {\tiny $0$};
\node[above] at (2,7) {$\cdots$};
\end{tikzpicture}
\end{multline}
Now we observe that the bosonic lattice is completely trivialized, since the $n$ incoming paths at the left edge cannot leave the lattice via the top external edges. This forces them to propagate horizontally, so that the red horizontal lines are completely saturated by paths. This in turn means that each $m_i$ must be equal to zero, trivializing the summation. The resulting frozen lattice configuration has weight $1$, so we remove it and retain only the six-vertex partition function emerging from the right:
\begin{align}
\PHL
\left(
[\lambda^{(1)} \subset \cdots \subset \lambda^{(n)}] = \{s_1,\dots,s_n\}
\right)
=
\begin{tikzpicture}[scale=0.7,baseline=(current bounding box.center),>=stealth]
\foreach \x in {1, ..., 6} {
\draw[dotted] (0,\x) -- (\x,\x) -- (\x, 7);
\draw[ultra thick,->] (0,\x) --(1,\x);
\node[above] at (\x,7) {$s_{\x}$};
\node at (\x,\x) {$\bullet$};
\draw[gray] (-1,\x) node{$a_\x$};
}
\foreach \x in {1,3,5,6}{
\draw[ultra thick,->] (\x,6) -- (\x,7);
}
\end{tikzpicture}
\end{align}
The partition function appearing on the right hand side evaluates the probability that the $n$-th path string in the stochastic six-vertex model is equal to $\{s_1,\dots,s_n\}$, i.e. $\PP( \sigma_n(\mathcal{C}) = \{s_1,\dots,s_n\})$, completing the proof of \eqref{equiv_distrib}.

\end{proof}

\begin{corollary}
For any $x\in \R$, $n\in 2\Z_{>0}$, and any parameters $(a_1, \dots, a_n)\in(0,1)^n$,
\begin{equation}
\EE\left[\frac{1}{(-t^{x+n-\mathfrak{h}(n,n)}, t^2)_{\infty}}\right] =  \Pf\left[  \mathsf{J}+\fkernel_x\cdot \kernel^{\complement}\right]_{\ell^2(\Z_{\geqslant 0})},
\end{equation}
where $\mathfrak{h}(n,n)$ is the height function at $(n,n)$ in the half-space  stochastic six-vertex model, and $\fkernel_x$ is defined in \eqref{eq:deffkernel}.
\label{cor:fredholmhnn}
\end{corollary}
\begin{proof}
This follows immediately from Proposition \ref{prop:FredholmHL} and Corollary \ref{cor:equaldistHL6V}.
\end{proof}
\begin{remark}
Corollary \ref{cor:fredholmhnn} could be used to show an analogue of Theorem \ref{th:GOElimitintro} for the height function of the half-space stochastic six-vertex model, i.e. proving that $\mathfrak{h}(n,n)$ has Tracy-Widom GOE fluctuations as $n$ goes to infinity. A possible approach would be to adapt the arguments of \cite[Appendix B]{aggarwal2016phase} to the half-space case and we do not pursue that here.
\label{rem:asymptotics6v}
\end{remark}

\section{Half-line ASEP with open boundary}
\label{sec:ASEP}
Consider the six-vertex model in a half-quadrant from Section \ref{sec:def6v} where $a_x\equiv a$, and  scale $a$ as
$$ a=1-\frac{(1-t)\e}{2},\ \ \  \e\xrightarrow[\e>0]{}0,$$
so that to first order in $\e$,
\begin{equation}
\PP\left(\pathrr\right) \approx \e, \ \ \PP\left(\pathru\right) \approx 1-\e,\ \ \PP\left(\pathuu\right) \approx t\e, \ \ \PP\left(\pathur\right) \approx 1-t\e.
\label{eq:limitweights}
\end{equation}
Moreover, we rescale $n$ as $n=\tau \e^{-1}$ with finite $\tau>0$.

\begin{definition}
The \emph{half-line ASEP} is a continuous time Markov process on the state-space parametrized by occupation variables $\big\lbrace  \big(\eta_x\big)_{x\in\Z_{>0}}\in \lbrace 0,1\rbrace^{\Z_{>0}} \big\rbrace$. The state $\eta(\tau)$ at time $\tau$ evolves according to the following dynamics: at any given time $\tau\in [0,\infty)$ and $x\in \Z_{>0}$, a particle jumps from site $x$ to $x+1$ at exponential rate
$$\eta_x(\tau)(1-\eta_{x+1}(\tau)) \in \lbrace 0,1\rbrace, $$ and from site $x+1$ to $x$ at exponential rate $$t\ \eta_{x+1}(\tau)(1-\eta_x(\tau))\in\lbrace 0,t\rbrace.$$ Further,  a particle is created or annihilated at the site $1$ at exponential rates
$$ \ratealpha\ (1-\eta_1(\tau)) \ \  \text{and}\ \  \rategamma\ \eta_1(\tau).$$
All these events are independent. We will restrict our attention to the case $\ratealpha=1/2$ and $\rategamma=t/2$ (see Figure \ref{fig:halfASEP}), and to the empty initial condition where there are no particles in the system at time $0$.  We define the current at site $x$ by
$$ N_x(\tau) = \sum_{i=x}^{\infty} \eta_i(\tau),$$
and we simply denote by $\NN(\tau) = N_1(\tau)$ the  number of particles in the system at time $\tau$.
\label{def:halflineASEP}
\end{definition}
\begin{figure}
\begin{tikzpicture}[scale=0.7]
\draw[thick] (-1.2, 0) circle(1.2);
\draw (-1.2,0) node{reservoir};
\draw[thick] (0, 0) -- (12.5, 0);
\foreach \x in {1, ..., 12} {
	\draw[gray] (\x, 0.15) -- (\x, -0.15) node[anchor=north]{\footnotesize $\x$};
}

\fill[thick] (3, 0) circle(0.2);
\fill[thick] (6, 0) circle(0.2);
\fill[thick] (7, 0) circle(0.2);
\fill[thick] (10, 0) circle(0.2);
\draw[thick, ->] (3, 0.3)  to[bend left] node[midway, above]{$1$} (4, 0.3);
\draw[thick, ->] (6, 0.3)  to[bend right] node[midway, above]{$t$} (5, 0.3);
\draw[thick, ->] (7, 0.3) to[bend left] node[midway, above]{$1$} (8, 0.3);
\draw[thick, ->] (10, 0.3) to[bend left] node[midway, above]{$1$} (11, 0.3);
\draw[thick, ->] (10, 0.3) to[bend right] node[midway, above]{$t$} (9, 0.3);
\draw[thick, ->] (-0.1, 0.5) to[bend left] node[midway, above]{$1/2$} (0.9, 0.4);
\draw[thick, <-] (0, -0.5) to[bend right] node[midway, below]{$t/2$} (0.9, -0.4);
\end{tikzpicture}
\caption{Jump rates in the half-line ASEP. }
\label{fig:halfASEP}
\end{figure}
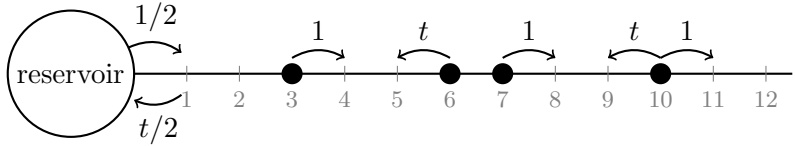

\begin{proposition}
Under the scalings and boundary and initial conditions as in Definition \ref{def:halflineASEP} (with $\ratealpha=1/2$ and $\rategamma=t/2$), for any $x\in\lbrace 1, 2, \dots\rbrace$,
$$ n-x -\mathfrak{h}(n-x,n) \xRightarrow[\e \to 0]{} N_x(\tau),$$
where $\mathfrak{h}$ is defined in \eqref{eq:defheight6v} and $N_x(\tau)$ is defined in Definition \ref{def:halflineASEP}.

Moreover, along a sequence of $\e$ such that $n
$ is even,
$$ n- \mathfrak{h}(n,n) \xRightarrow[\e \to 0]{} \lceil \NN(\tau)\rceil_2,$$
where for an integer $k$, we define
$$\lceil k \rceil_2 := \min \lbrace i\in 2\Z: i \geqslant k \rbrace.$$
\label{prop:cv6vtoASEP}
\end{proposition}
\begin{remark}
For the stochastic six-vertex model in a quadrant and ASEP on $\Z$, a heuristic approach to the convergence was provided in \cite{borodin2016stochasticsix}: Consider the ensemble of paths in the six-vertex model   and interpret each path as the trajectory of a particle where the vertical axis is the time. Under the scalings that we consider, the dynamics of this particle system converge to those of ASEP (modulo a shift of particle positions by time). The convergence was proven rigorously in \cite{aggarwal2017convergence} for general initial condition. The main difficulty is that for a system with infinitely many particles, the distribution of one particle may depend on the position of other particles far away. This is why  \cite{aggarwal2017convergence} considers versions of both models restricted to an interval $[-N,M]$ (which consequently have finitely many particles), proves that the convergence holds for the bounded models, and shows that the unbounded models are well approximated by the bounded ones as $M,N$ goes to infinity.  In our case, the only additional complication is the boundary, which we address in the below proof.
\end{remark}
\begin{proof}[Proof of Proposition \ref{prop:cv6vtoASEP}]
For the empty initial condition that we consider in this paper, we will show that the tail distribution of the total number of particles injected by the reservoir into the system decays exponentially fast,  uniformly in $\epsilon$, so that the convergence of the dynamics of finitely many particles is enough to prove Proposition \ref{prop:cv6vtoASEP}. Hence we will focus on carefully justifying that the boundary behaviour in the six-vertex model converges to that of half-line ASEP as in Definition \ref{def:halflineASEP}.

The first step is to apply a particle-hole inversion to the ensemble of paths defined by the stochastic six-vertex model in a half-quadrant (see Figure \ref{fig:particlehole}). This means that edges of the lattice occupied by a path become empty and empty edges become occupied by a path.
\begin{figure}
\begin{center}
\begin{tikzpicture}[scale=0.8]
\foreach \x in {1, ..., 7} {
\draw[dotted] (0,\x) -- (\x,\x) -- (\x, 8);
\draw (7.9, \x) -- (8.1,\x) node[right]{\scriptsize{$\x$}};
}
\draw[path] (1,1) -- (1,2) -- (2,2);
\draw[path] (3,3) -- (3,5) -- (4,5) --  (4,8);
\draw[path] (4,4) -- (4,5) -- (5,5);
\draw[path] (6,6) -- (6,8);
\draw[path] (7,7) -- (7,8);
\draw[axis] (8,0.5) -- (8,8) node[right]{time};

\begin{scope}[xshift=10cm]
\foreach \x in {1, ..., 5} {
\draw (\x, 0) node{\scriptsize{$\x$}};
\foreach \y in {1, ..., 7}{
\draw (\x-1,\y-0.1) -- (\x-1, \y+0.1 );
\draw[path] (0, \y) -- (5,\y);
}
}
\fill[gray] (0,1) circle(0.1);
\fill[gray] (0,3) circle(0.1);
\fill[gray] (0,4) circle(0.1);
\fill (1,4) circle(0.1);
\fill (1,5) circle(0.1);
\fill[gray] (0,6) circle(0.1);
\fill (2,6) circle(0.1);
\fill[gray] (0,7) circle(0.1);
\fill (1,7) circle(0.1);
\fill (3,7) circle(0.1);
\end{scope}
\end{tikzpicture}
\end{center}
\caption{Left: the same arrow configuration as in Figure \ref{fig:example6v} after particle-hole transformation. Right: corresponding particle configurations. The presence of the particle in gray at zero is completely determined by the rest of the configuration, ensuring that the total number of particles and the time have the same parity. A particle may be injected at site $1$ in the next step only when there is a gray particle at $0$. A particle may be removed from the system in the next step only when there is no gray particle at $0$.}
\label{fig:particlehole}
\end{figure}
Next we associate to the (particle-hole transformed) ensemble of paths the evolution of a particle configuration: Let $\zeta_i(s)$  (viewed as an occupation variable at position $i$ and time $s$)  be $1$ if there is an outgoing vertical arrow out of vertex $ (s-i,s)$ and $0$ else (see Figure \ref{fig:particlehole}). Denote by
$$ y_1>y_2> \dots >y_{M(s)}>0$$
the associated positions of the $M(s)$ particles occupying positive sites. We do not record the presence of the particle at $0$ because it can be deduced from the parity of $s+M(s)$ (indeed, $s+M(s)+\zeta_0(s)$ is always even). Note that $M(s)=s-1-\mathfrak{h}(s-1,s)$ and for even $s$, one can write $\lceil M(s)\rceil_2 = s-\mathfrak{h}(s,s)$.

The particle configuration $(y_i)_{i\in \Z_{>0}}$ can be constructed as a discrete-time Markov process where particles are updated from right to left -- i.e. $y_1$, then $y_2$, etc. -- according to the following rules:
\begin{enumerate}
\item  For $1\leqslant i\leqslant M(s)$, the $i$th particle jumps by $1$ at time $s+1$ with probability  $\e+o(\e)$, provided $y_{i-1}(s+1)>y_i(s)+1$  (with the convention that $y_0=\infty$).
\item  For $1\leqslant i\leqslant M(s)$, and for all  $1\leqslant j \leqslant y_i(s)-y_{i+1}(s)-1$ such that $y_i(s)-j>0$, the $i$th particle jumps by $-j$ at time $s+1$ with probability $(t\e)^j+o(\e^j)$.
\item If $M(s)+s$ is odd and $y_{M(s)}(s)>1$, a new particle is created at $y_{M(s+1)}(s+1)=1$ and $M(s+1)=M(s)+1$, happening with probability $\e+o(\e)$.
\item If $M(s)+s$ is odd  and $y_{M(s)}(s)=k>0$, this particle  ejects from the system at time $s+1$ and $M(s+1)=M(s)-1$ with probability $(t\e)^k+o(\e^k)$.
\item Each particle stays put with the complementary probability.
\end{enumerate}

Let $x^{\e}_i(\tau): = y_i(\tau\e^{-1})$ be the position of the particles at large times.
Let $I_{6v}(s)$ (resp. $I(\tau)$) be the total number of particles injected in the six-vertex  process $\zeta$ (resp. the ASEP process  $\eta$) between times $0$ and $s$ (resp. $0$ and $\tau$). Observe that the tail distribution of $I_{6v}(s)$ can be naively bounded by
\begin{equation}
\PP\big(I_{6v}(\tau \e^{-1})>x\big) <\PP\big(\mathfrak{Bin}(\epsilon, \tau\e^{-1})>x\big) < e^{-c x},
\label{eq:controlI}
\end{equation}
for some constant $c>0$ independent of $\epsilon$, where $\mathfrak{Bin}(p,n)$ denotes a Binomial random variable with parameters $p$ and $n$. If for any $x,k\in \Z_{\geqslant 0}$,
\begin{equation}
 \PP\Big(M\big(\tau\e^{-1}\big)=x\ \& \  I_{6v}\big(\tau\e^{-1}\big)=k\Big) \xrightarrow[\e \to 0]{} \PP\big(\NN(\tau)=x\ \& \ I(\tau)=k\big),
 \label{eq:CVtoshow}
\end{equation}
we may sum over $k$ and use \eqref{eq:controlI} to conclude that
$$ \PP\Big(M\big(\tau\e^{-1}\big)=x\Big) \xrightarrow[\e \to 0]{} \PP\big(\NN(\tau)=x\big).$$

Now we turn to the proof of \eqref{eq:CVtoshow}. Fix $k\in \Z_{\geqslant 0}$ and let us restrict our analysis to the event where  $I_{6v}(\tau\e^{-1})=k$, so that we need to consider the dynamics  of only finitely many particles $x^{\e}_{1}(\tau), \dots ,  x^{\e}_{k}(\tau)$.
It is clear (from the convergence of discrete time random walk on $\Z^d$ to a continuous time one, or more precisely from \cite{aggarwal2017convergence}) that dynamics of particles away from the boundary converge to  ASEP. According to the rule (3) of the dynamics, whenever the site $1$ is empty, a particle is added to the system at site $1$ after a random time that converges to an  exponential with rate $1/2$ (this $1/2$ comes from the fact that the particle creation is not possible at all times but only half of them, because of the parity condition). Similarly, according to the rule (4) of the dynamics, whenever the site $1$ is occupied, this particle is removed from the system after a random time that converges to an  exponential with rate $t/2$. Hence we have shown that on the event where  $I_{6v}(\tau\e^{-1})=k$, the joint distribution of $x^{\e}_{1}(\tau), \dots ,  x^{\e}_{k}(\tau)$ converge as $\e$ to zero to the distribution of ASEP particles $x_{1}(\tau), \dots ,  x_{k}(\tau)$, which implies \eqref{eq:CVtoshow}.

One can similarly deduce the convergence of $n-x-\mathfrak{h}(n-x, n)$ to $N_x(\tau)$ for any $x\in\lbrace 1, 2, \dots\rbrace$.
\end{proof}

\begin{proposition} For any time $\tau>0$ and  $x\in \R$,
\begin{equation}
 \EE\left[   \frac{1}{(-t^{x+\lceil  \NN(\tau)  \rceil_2}, t^2)_{\infty}} \right]  =  \Pf\left[ \mathsf{J}+ \fkernel_x\cdot \kernel^{\mathrm{ASEP}}\right]_{\ell^2(\Z_{\geqslant 0})},
 \label{eq:FredholmASEP}
\end{equation}
where $\fkernel_x$ is defined in \eqref{eq:deffkernel} and
\begin{subequations}
\begin{align}
\kernel^{\mathrm{ASEP}}_{11}(u,v)&= \frac{1}{(2\I\pi)^2} \iint \frac{(w-z)g(z)g(w)}{(z^2-1)(w^2-1)(zw-1)}
 \frac{\mathrm{d}z}{z^u}\frac{\mathrm{d}w}{w^v}  + \frac{1}{2\I\pi}\int  \frac{g(z)}{z^2-1} \frac{\mathrm{d}z}{z^v}\mathds{1}_{u \in 2\Z}\label{eq:K11ASEP}\\
&\hspace{4cm} -\frac{1}{2\I\pi}\int  \frac{g(z)}{z^2-1} \frac{\mathrm{d}z}{z^u}\mathds{1}_{v \in 2\Z}  + r(u,v) \nonumber \\
 \kernel^{\mathrm{ASEP}}_{12}(u,v)&= \frac{1}{(2\I\pi)^2} \iint \frac{(w-z)g(z)g(w)}{(z^2-1)w(zw-1)}   \frac{\mathrm{d}z}{z^u}\frac{\mathrm{d}w}{w^v} +  \frac{1}{2\I\pi} \int \frac{g(z)}{z^{v+1}}\mathrm{d}z\mathds{1}_{u\in 2\Z}  \label{eq:K12ASEP}\\
\kernel^{\mathrm{ASEP}}_{22}(u,v)&= \frac{1}{(2\I\pi)^2} \iint \frac{(w-z)g(z)g(w)}{zw(zw-1)} \frac{\mathrm{d}z}{z^u}\frac{\mathrm{d}w}{w^v},
 \label{eq:K22ASEP}
\end{align}
 \label{eq:KASEP}
\end{subequations}
where
$g(z) = \exp\left( \frac{(1-t)\tau}{2}\frac{z+1}{z-1}\right)$, $r(u,v)$ is defined in \eqref{eq:defr}   and the contours are chosen as positively oriented circles with radius less that $1$.
\label{prop:Fredholmcurrent}
\end{proposition}
\begin{proof}
Recall that from Corollary \ref{cor:fredholmhnn},
\begin{equation}
\EPHL\left[\frac{1}{(-t^{x+n-\mathfrak{h}(n,n)}, t^2)_{\infty}}\right] =  \Pf\left[  \mathsf{J}+\fkernel_x\cdot \kernel^{\complement}\right]_{\ell^2(\Z_{\geqslant 0})}.
\label{eq:relationheightSchur}
\end{equation}

Fix $\tau>0$ and consider an even integer $n$  and the corresponding $\e$ such that $ n=\tau \e^{-1}$. Recall that $a$ is a parameter also depending on $n$ through $\e$. We will let $n$ go to infinity along the even integers.
Under the scalings considered, with $f(z) = \left(\frac{z-a}{1-az}\right)^n$ as in Lemma \ref{lem:residuecomputations},
$$ (-1)^n f(z)\xrightarrow[n \to \infty]{} g(z),$$
uniformly on compact subsets of $\C\setminus\lbrace 1\rbrace$,
and we can discard this factor $(-1)^n$ since $n$ is assumed even.

By Proposition \ref{prop:cv6vtoASEP}, the left hand side of   \eqref{eq:relationheightSchur} converges to the left hand side of \eqref{eq:FredholmASEP} as $n$ goes to infinity (note that the observables under expectations have values in $(0,1)$ thus pointwise convergence is sufficient).

On the other hand, $\kernel^{\complement}$ converges pointwise to $\kernel^{\mathrm{ASEP}}$ as $\e\to 0$. To conclude that the Fredholm Pfaffians also converge, it is enough to check that the Fredholm determinant expansion is absolutely convergent. The integrand of the kernel is uniformly (in $\e$) bounded on its contour  and the integrals are on finite contours, so that the kernel is bounded. One concludes using Hadamard's bound (see Lemma \ref{lem:hadamard}) to control the Fredholm Pfaffian expansion.
\end{proof}
%\begin{remark}
%We know from Theorem \ref{th:matchingHL6V} that $\mathfrak{h}(n,n) $ and $\ell(\mu)$ have the same distribution, where $\mu$ is distributed according to $\PHL$ with specialization $a, \dots, a$ and the weights of the six-vertex model are such that $a_x\equiv a$.
%One may then try to use the asymptotic equivalence in  Proposition \ref{prop:asymptoticeq} to compute the asymptotics of  $\ell(\mu)$ using Schur process asymptotics, and relate this to the current in ASEP using Proposition \ref{prop:cv6vtoASEP}. Unfortunately, under the scalings considered, the random variables at play do not spread, thus making  the asymptotic equivalence a trivial statement.
%\end{remark}
\begin{remark}
We could define a Pfaffian point process with correlation kernel $\kernel^{\rm ASEP}$. The Pfaffian point process describes the edge of the Pfaffian Schur measure under the specific limit that we considered. For the (determinantal) Schur  measure, a similar limit was computed in \cite{borodin2016asep} and called discrete Laguerre ensemble due to its close connection to Laguerre orthogonal polynomials.
\end{remark}

\section{Fluctuations of the current in half-line-ASEP}
\label{sec:GOElimit}

%\begin{definition}
%Let $\mathcal{C}_{a}^{\fkernel}$ be the union of two semi-infinite rays departing $a\in\C$ with angles $\fkernel$ and $-\fkernel$ (see Figure \ref{fig:basicrays}). We  assume that the contour is oriented from $a+\infty e^{-i\fkernel}$ to $a+\infty e^{+i\fkernel}$.
%\label{def:basicrays}
%\end{definition}
%\begin{figure}
%\begin{tikzpicture}[scale=0.9]
%\draw[->] (-2,0) -- (4,0);
%\draw[->] (0,-3.2) -- (0,3.2);
%\draw[thick] (1, 0) -- (2,3);
%\draw[thick] (1, 0) -- (2,-3);
%\draw[thick, ->] (1, 0) -- (1.5,1.5);
%\fill (1,0) circle(0.05);
%\draw node at (0.9,-0.2) {$a$};
%\draw[->] (2,0) arc(0:70:1);
%\draw node at (2,1) {$\fkernel$};
%\draw node at (2,-2) {$\mathcal{C}_a^{\fkernel}$};
%\end{tikzpicture}
%\caption{The contour $\mathcal{C}_a^{\fkernel}$ from Definition \ref{def:basicrays}.}
%\label{fig:basicrays}
%\end{figure}
\begin{definition}
The \emph{GOE Tracy-Widom distribution} \cite{tracy1996orthogonal} is a continuous probability distribution on $\R$ with cumulative distribution function given by (see e.g. Section 2.3 in \cite{baik2017pfaffian})
$$ F_{\rm GOE}(x)
 = \Pf\big[ \mathsf{J}- \kernel^{\rm GOE}\big]_{\mathbb{L}^2(x, \infty)},$$
where $\kernel^{\rm GOE}$ is the $2\times 2$ matrix valued kernel defined by
\begin{subequations}
\begin{align}
\kernel_{11}^{\rm GOE}(x,y) &= \frac{1}{(2\I\pi)^2} \int \mathrm{d}z\int\mathrm{d}w  \frac{z-w}{z+w} e^{z^3/3 + w^3/3 - xz -yw},\\
\kernel_{12}^{\rm GOE}(x,y) &= -\kernel_{21}^{\rm GOE}(x,y) = \frac{1}{(2\I\pi)^2}\int\mathrm{d}z\int\mathrm{d}w  \frac{w-z}{2w(z+w)} e^{z^3/3 + w^3/3 - xz -yw} \\
& +\frac{1}{2}\frac{1}{2\I\pi} \int \mathrm{d}z e^{z^3/3  - xz },\nonumber \\
\kernel_{22}^{\rm GOE}(x,y) &= \frac{1}{(2\I\pi)^2} \int\mathrm{d}z\int\mathrm{d}w  \frac{z-w}{4zw(z+w)} e^{z^3/3 + w^3/3 - xz -yw} \\
& +\frac{1}{2\I\pi} \int e^{z^3/3-zx}\frac{\mathrm{d}z}{4z} - \frac{1}{2\I\pi} \int e^{z^3/3-zy}\frac{\mathrm{d}z}{4z} -\frac{\sgn{(x-y)}}{4}\nonumber,
\end{align}
\label{eq:KGOE}
\end{subequations}
where all integration contours are $\mathcal{C}_1^{\pi/3}$, the contour formed by the union of two semi-infinite rays departing $1$ with angles $\pi/3$ and $-\pi/3$,  oriented from $1+\infty e^{-\I\pi/3}$ to $1+\infty e^{+\I\pi/3}$. Moreover, there exists a Pfaffian point process on $\R$ having this kernel and we call this point process the \emph{GOE point process}. We will denote its points by
$ \mathfrak{a}_1>\mathfrak{a}_2>\dots $
\label{def:GOEdistribution}
\end{definition}

\begin{theorem} For any $t\in[0,1)$,
\begin{equation}
\lim_{T\to \infty} \PP\left(\frac{\NN\left(\frac{T}{1-t}\right) - \frac{T}{4}}{2^{-4/3}T^{1/3}}> -x\right) = F^{\rm GOE}(x).
\label{eq:GOElimit}
\end{equation}
\label{th:GOElimit}
\end{theorem}
\begin{proof}
First remark that $\lceil \NN(\tau)\rceil_2$ and $\NN(\tau)$ differ by at most one and consequently  have the same limit under the scalings considered in the statement of the Theorem.

Recall that from Proposition \ref{prop:Fredholmcurrent},
\begin{equation}
\EE\left[   \frac{1}{(-t^{y+\lceil  \NN(\tau)  \rceil_2}, t^2)_{\infty}} \right]  = \Pf\left[ \mathsf{J}+ \fkernel_y\cdot  \kernel^{\mathrm{ASEP}}\right]_{\ell^2(\Z_{\geqslant 0})}.
\label{eq:Fredholasymptotics}
\end{equation}
Define the scaling function
$$ \scaling{x}:=  \frac{T}{4}-2^{-4/3}x T^{1/3}.$$
We can let $\tau=\frac{T}{1-t}$ and scale $y$ as $y=- \scaling{\kappa}$ for a fixed $\kappa\in \R$.
Assume that under this scaling  the right-hand-side of \eqref{eq:Fredholasymptotics} converges for every $\kappa\in \R$ to the distribution function $F^{\rm GOE}(\kappa)$.  Then, using Lemma \cite[Lemma 4.1.39]{borodin2014macdonald}, the random variable
\begin{equation}
\frac{\NN\left(\frac{T}{1-t}\right)-\frac{T}{4}}{2^{-4/3}\kappa T^{1/3}}
\label{eq:limitwewant}
\end{equation}
weakly converges as $T$ goes to infinity to a random variable with distribution function $F^{\rm GOE}(\kappa)$.
We now consider the right-hand-side of \eqref{eq:Fredholasymptotics}, the function $\fkernel_y$ multiplying the kernel is asymptotically an indicator function.
\begin{lemma} For a fixed $\kappa \in \R$, the function $\R\to (-1,0)$ defined by
$$\fkernel^{(T)}_{\kappa} : x\mapsto \fkernel_{\scaling{\kappa}}(\scaling{x})$$
converges pointwise to $x\mapsto -\mathds{1}_{x>\kappa}$.
\end{lemma}
\begin{proof}
We have
$$\fkernel^{(T)}_{\kappa}(x) = \fkernel_{\scaling{\kappa}}(\scaling{x})  = \frac{(-t^{1+2^{-4/3}T^{1/3}(\kappa-x)}; t^2)_{\infty}}{(-t^{2^{-4/3}T^{1/3}(\kappa-x)}; t^2)_{\infty}}-1.$$
As $T$ goes to infinity, if $\kappa>x$ the first argument in the $q$-Pochhammer symbols goes to $0$, so that the $q$-Pochhammer symbols themselves converge to $1$ and $\fkernel$ goes to zero. If $\kappa<x$,  we need to show that the ratio of $q$-Pochhammer symbols converges to zero.
Let
$$p(x) :=    \frac{(-t^{1+x} ; t^2)_{\infty}}{(-t^x; t^2)_{\infty}}.$$
 We have
$$ p(x)p(x+1) = \frac{1}{1+t^x},$$
which, since $p(x)$ is increasing,  shows that $p(x)$ goes to zero as $x$ goes to $-\infty$. Hence $\fkernel_{\scaling{\kappa}}(\scaling{x})$ goes to $-1$ as $T$ goes to infinity when $\kappa<x$.
\end{proof}

Thus our main task is to compute the limit of $\kernel^{\rm ASEP}$. We will use Laplace's method to find the asymptotics of the kernel in \eqref{eq:K11ASEP}, \eqref{eq:K12ASEP} and \eqref{eq:K22ASEP}.
 The limit of \eqref{eq:Fredholasymptotics} as $T$ goes to infinity can be computed as  the limit of
 $$\Pf\Big[\mathsf{J}+\fkernel^{(T)}_{\kappa} \cdot \kernel^{(T)}\Big]_{\ell^2(\mathbb{D}_T)}$$ where $\kernel^{(T)}$ is the rescaled kernel
\begin{equation} \kernel^{(T)}(x,y) := \frac{(-1)^{\scaling{x}+\scaling{y}}}{2^{-4/3}T^{1/3}} \begin{pmatrix}  \kernel_{11}^{\rm ASEP}\left(\scaling{x},\scaling{y}\right) & 2^{-4/3}T^{1/3}\kernel_{12}^{\rm ASEP}\left(\scaling{x},\scaling{y}\right)\\
2^{-4/3}T^{1/3} \kernel_{21}^{\rm ASEP}\left(\scaling{x},\scaling{y}\right) & \left(2^{-4/3}T^{1/3}\right)^2 \kernel_{22}^{\rm ASEP}\left(\scaling{x},\scaling{y}\right)
\end{pmatrix}  ,\label{eq:defKT}\end{equation}
and the domain $\mathbb{D}_T$ of the Fredholm Pfaffian is defined so that $\scaling{\mathbb{D}_T}=\Z_{\geqslant 0}$.
%$$ \mathbb{D}_T=\Big\lbrace -\frac{T}{4}\frac{1}{2^{-4/3}T^{1/3}}, -\frac{T}{4}\frac{1}{2^{-4/3}T^{1/3}}+\frac{1}{2^{-4/3}T^{1/3}}, -\frac{T}{4}\frac{1}{2^{-4/3}T^{1/3}}+\frac{2}{2^{-4/3}T^{1/3}}, \dots  \Big\rbrace.$$
The presence of the factor $(-1)^{\scaling{x}+\scaling{y}}$ is a technical convenience whose purpose  shall be explained later. It does not change the value of the Fredholm Pfaffian because this extra factor has the same effect as a conjugation of the kernel (since $\scaling{x}, \scaling{y}\in\Z$, $(-1)^{\scaling{x}+\scaling{y}}=(-1)^{\scaling{x}-\scaling{y}}=(-1)^{-\scaling{x}+\scaling{y}}=(-1)^{-\scaling{x}-\scaling{y}}$). The power of $2^{-4/3}T^{1/3}$ that multiplies each entry of $\kernel^{\rm ASEP}$ is determined so that  each entry remains bounded as $T$ goes to infinity, and the global factor  $(2^{-4/3}T^{1/3})^{-1}$ in front of the kernel will disappear when we will approximate discrete sums by integrals.

Let us first examine formally the limit of $\kernel^{(T)}$. In the formulas \eqref{eq:KASEP} for $\kernel^{\rm ASEP}$, the parameter $T$ appears in the function $g(z)$ and in the variables $u,v$ through the scalings that we consider. If $\scaling{x}=\frac{T}{4}-2^{-4/3}T^{1/3}x$ as above, we can rewrite the factors depending on $T$ using
\begin{equation}
 \frac{g(z)}{z^{\scaling{x}}} = \exp\left(T \G(z) + 2^{-4/3}T^{1/3}x \log(z) \right),
 \label{eq:exponentialfactor}
\end{equation}
where $$ \G(z) = \frac{1}{2}\frac{z+1}{z-1} - \frac{1}{4}\log(z),$$
and the branch cut of the logarithm is taken on the positive reals. The function $z\mapsto \Real[\G(z)]$ will control the asymptotics of the kernel.
One can check that $\G'(-1)=\G''(-1)=0$, and $\G'''(-1)=1/8$, so that the Taylor expansion of $\G$ at $-1$ is
$$ T\G(-1 + \tilde z T^{-1/3}) \approx -T \frac{\I \pi}{4} +\frac{2^{-4}(\tilde z)^3 }{3}.$$
Moreover,
$$ 2^{-4/3}T^{1/3}x \ln(-1 + \tilde z T^{-1/3}) = 2^{-4/3}T^{1/3}x \I \pi - 2^{-4/3}x \tilde z + \mathcal{O}(T^{-1/3}), $$
so that for $z=1+ \tilde z T^{-1/3}$,
$$ \frac{g(z)}{z^{\scaling{x}}} = \exp\left(\scaling{x}\I\pi + \frac{2^{-4}(\tilde z)^3 }{3} - 2^{-4/3}x \tilde z  \right) + \mathcal{O}(T^{-1/3}). $$
The factor  $\exp(\scaling{x}\I\pi)$ will be cancelled by the $(-1)^{\scaling{x}}$ in the definition of $\kernel^{(T)}$ in \eqref{eq:defKT}, and this is why we have multiplied the kernel by a factor $(-1)^{\scaling{x}+\scaling{y}}$. Thus, under the scalings considered, in the neighborhood of the critical point, the first integrand in the formula for $\kernel_{11}^{\rm ASEP}$ will converge pointwise to
$$ \frac{\tilde z - \tilde w}{4\tilde z \tilde w (\tilde z + \tilde w)}\exp\left( \frac{\sigma^3 \tilde{z}^3}{3} + \frac{\sigma^3 \tilde{w}^3}{3} - \sigma \tilde z x -\sigma \tilde w y \right), $$
where $\sigma=2^{-4/3}$ and we have rescaled the integration variables as $z=-1 + \tilde z T^{-1/3}$ and likewise for $w$. It is easy to see that all other integrands converge pointwise to similar formulas as well, and we will produce the formulas later. However, the factors $\mathds{1}_{u \in 2\Z}$, $\mathds{1}_{v \in 2\Z}$, and $\mathds{1}_{v-u \in 2\Z+1}$ that appear in \eqref{eq:K11ASEP} and \eqref{eq:K12ASEP} do not have a pointwise limit as $T$ goes to infinity since they depend on the parity of $\scaling{x}$ and $\scaling{y}$.

Since we are interested in the convergence of $\Pf\big[\mathsf{J}-\fkernel^{(T)}_{\kappa} \cdot  \kernel^{(T)}\big]$, we do not need to prove  pointwise convergence of $\kernel^{(T)}$. Instead, we will first search for a kernel $\kernel^{\infty}$ defined on $\mathbb{L}^2(\kappa, \infty)$ for any $\kappa\in \R $ such that for any compact Lebesgue measurable set $\mathbb{K}\subset \R^k$,
\begin{equation}
  \sum_{x_1, \dots, x_k\in \mathbb{K}} \Pf\left(\kernel^{(T)}(x_i,x_j)\right)_{i,j=1}^k \xrightarrow[T\to\infty]{}  \int_{\mathbb{K}}\mathrm{d}x_1 \dots\mathrm{d}x_k \Pf\left(\kernel^{\infty}(x_i, x_j)\right)_{i,j=1}^k,
  \label{eq:cvoncompacts}
\end{equation}
where the sum on the left-hand-side is over all $k$-tuples $(x_1, \dots, x_k) \in \mathbb{K}$ such that for all $1\leqslant i\leqslant k$, $\scaling{x_i}= \frac{T}{4}-2^{-4/3}T^{1/3}x_i\in \Z$. If $k=1$, one may replace the indicator functions in $\kernel^{(T)}$ by their average value, that is
$$ \mathds{1}_{u \in 2\Z}\approx \frac{1}{2}, \ \  \mathds{1}_{v \in 2\Z}\approx \frac{1}{2}.$$
One has to be careful that in general, the average of a product is not the product of averages. However, for any $k\geqslant 1$, if we expand the Pfaffian in the left hand-side of \eqref{eq:cvoncompacts} as in \eqref{eq:defPfaffian}, we notice that indicator functions involving $x_i$ can be multiplied by an indicator function involving $x_j$ only when $j\neq i$, so that we can approximate all indicator functions by their average value. More precisely, we will use
\begin{align*}
&(-1)^{\scaling{x}}\mathds{1}_{\scaling{x}\in 2\Z} \Rightarrow \frac{1}{2},\\
&(-1)^{\scaling{x}+\scaling{y}}r(\scaling{x},\scaling{y}) \Rightarrow \frac{1}{4}\sgn(y-x),
\end{align*}
where $\Rightarrow$ means that the convergence holds in the integrated sense of \eqref{eq:cvoncompacts}. Indeed, we have
$$r(u,v) = \frac 1 4 \big((-1)^u-(-1)^v\big) + \frac 1 2 \sgn(v-u)\mathds{1}_{v-u\in 2\Z+1}.$$
The quantities $(-1)^u$ and $(-1)^v$ have average $0$ and the quantity $(-1)^{u+v}\mathds{1}_{v-u\in 2\Z+1}$ has average $-1/2$, so that $ (-1)^{u+v} r(u,v) $ has average $\frac{1}{4}\sgn(u-v)$.  Finally $\sgn(\scaling{x}- \scaling{y}) = \sgn(y-x)$.

We claim that Laplace's method shows that \eqref{eq:cvoncompacts} holds with $\kernel^{\infty}$ being the kernel obtained by approximating all integrals in $\kernel^{(T)}$ for $z$ and $w$ in a neighbourhoud of size $T^{-1/3}$ around the critical point,  taking the pointwise limit of integrands, and replacing indicator functions by their average values as prescribed above. In order to justify this rigorously, we need to check two facts:
\begin{enumerate}
\item The integration in \eqref{eq:KASEP} can be restricted to a neighbourhood of $-1$ of arbitrarily small size, making an error going to zero  as $T$ goes to infinity uniformly in $x,y$.
\item All integrands in  $\kernel^{(T)}$ can be indeed approximated  by  their pointwise limit, making an error going to zero.
\end{enumerate}
To prove (1), it is enough to show that the contour employed in $\kernel^{\rm ASEP}$ can be freely deformed to a finite contour $\mathcal{C}$ satisfying the following property:  for any $\eta>0$, there exist $\xi>0$ such that for all $z\in \mathcal{C}$ with $\vert z+1\vert>\eta$,  we have $\Real[\G(z)]<-\xi$. This will result in the integrand being exponentially small as $T$ goes to infinity outside of a $\eta$-neighborhood around the critical point. One readily checks that the circle of radius $1$ centered at $0$ is a contour line for $\Real[\G(z)]$. Moreover, the Taylor expansion at $-1$ implies that there is also a contour line departing $-1$ with angles $\pm\pi/6$, and because $\Real[\G(z)]$ is harmonic, it must enclose the singularity at $0$ and stay inside of the circle of radius $1$. Hence, one may find a contour with the desired properties between two closed contour lines (see Figure \ref{fig:contour} where a possible choice is depicted), and Taylor expansion of $\G$ around $-1$ shows that this contour $\mathcal{C}$ may depart $-1$ with any angles between $\pm\pi/6$ and $\pm\pi/2$. Since integrations must avoid the pole at $-1$, the contour needs to be modified in a region of size $\mathcal{O}(T^{-1/3})$ around $-1$, and one can choose a contour departing the point $-1+T^{-1/3}$ with angles $\pm\pi/3$, thus leading in the limit to the contour $\mathcal{C}_1^{\pi/3}$ from Definition \ref{def:GOEdistribution}.
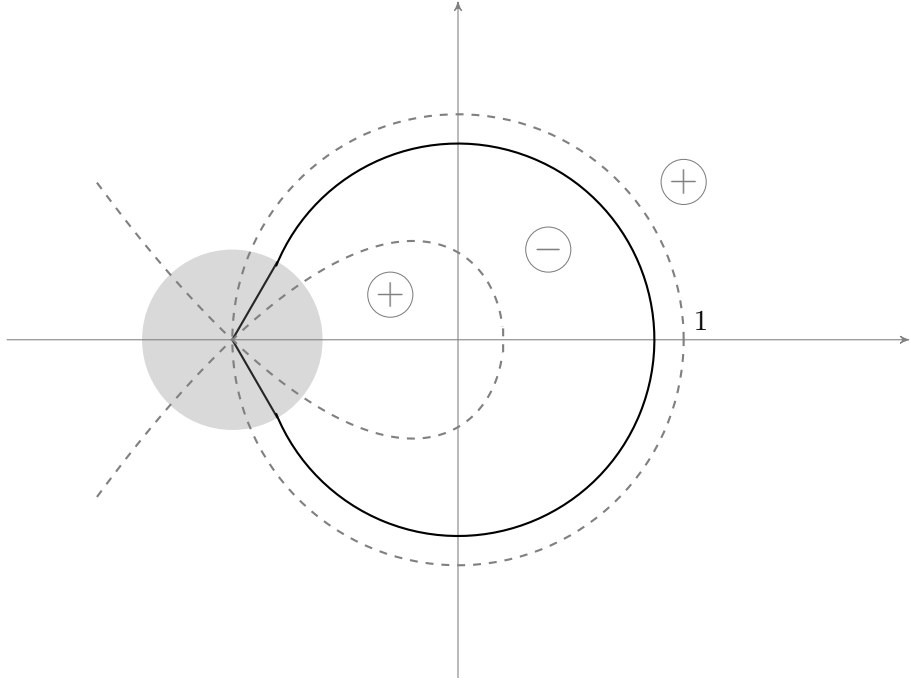
\begin{figure}
	\begin{tikzpicture}[scale=3]
	\draw[gray, ->, >=stealth'] (-2,0) --(2,0);
	\draw[gray, ->, >=stealth'] (0,-1.5) -- (0,1.5);
	\draw  (1, 0) node[above right] {$1$};
	\draw[thick, gray, dashed] (0,0) circle(1);
	\draw[thick] (-1,0) -- ++(60:0.4);
	\draw[thick] (-1,0) -- ++(-60:0.4);
	\fill[gray, opacity=0.3] (-1,0) circle(0.4);
	\draw[thick, ] (158:0.87) arc(158:-158:0.87);
	\draw[thick, gray, dashed] plot [domain = -1.6:0.2,samples=500]
	(\x,{ sqrt(3/4*(\x+1)*(\x+1)*(-\x+0.2))});
	\draw[thick, gray, dashed] plot [domain = -1.6:0.2,samples=500]
	(\x, {- sqrt(3/4*(\x+1)*(\x+1)*(-\x+0.2))});
	\draw[ gray] (-0.3, 0.2) node{{\Large $+$}};
	\draw[gray] (-0.3, 0.2) circle(0.1);
	\draw[ gray] (0.4, 0.4) node{{\Large $-$}};
	\draw[gray] (0.4, 0.4) circle(0.1);
	\draw[ gray] (1, 0.7) node{{\Large $+$}};
	\draw[gray] (1, 0.7) circle(0.1);
	\draw[gray, thick, dashed] (0.2, -0.055) -- (0.2, 0.055);
	\end{tikzpicture}
	\caption{Contour $\mathcal{C}$ used in the proof of Theorem \ref{th:GOElimit} and  Lemma \ref{lem:expoboundKPZ} (in black). The dashed lines are contour lines of $\Real[\G(z)]$ and we have indicated by the symbol \raisebox{4pt}{\protect\tikz{1}{\protect\draw[ gray] (0, 0) node{ $+$};
	\protect\draw[gray] (0,0) circle(0.2);}} or  \raisebox{4pt}{\protect\tikz{1}{\protect\draw[ gray] (0, 0) node{ $-$};
	\protect\draw[gray] (0,0) circle(0.2);}} regions where $\Real[\G(z)]$ is positive or negative. The gray area is the $\eta$-neighborhood around $-1$ . The part of the contour formed by two segments departing $-1$ is the contour used in $I_1(x)$ and the circular part is the contour used in $I_2(x)$ ($I_1(x)$ and $I_2(x)$ are defined in the proof of Lemma \ref{lem:expoboundKPZ}).}
	\label{fig:contour}
\end{figure}

Proving (2) amounts to control the error made replacing all quantities by their Taylor  approximations, and it can be done using very standard bounds (see e.g. Equations (71) and (72) in \cite{baik2017pfaffian}). Since the functions $\G$ and $\log$ admit a Taylor expansion at any order and the exponential factors of the form \eqref{eq:exponentialfactor} are simply multiplied by rational functions in $z$ and $w$, the necessary bounds are exactly the same as in previous papers ( e.g. \cite{baik2017pfaffian, aggarwal2016phase, borodin2012free}) and we do not repeat the argument here.

At this point, we have deduced that $\kernel^{(T)}$ converges to $\kernel^{\infty}$ in the sense of \eqref{eq:cvoncompacts} where
\begin{align*}
\kernel^{\infty}_{11}(x,y) &= \frac{1}{(2\I\pi)^2} \iint \mathrm{d}z\mathrm{d}w \frac{z-w}{4zw(z+w)}\exp\left( \frac{\sigma^3 \tilde{z}^3}{3} + \frac{\sigma^3 w^3}{3} - \sigma z x -\sigma  w y \right) \\
&- \frac{1}{2\I\pi} \int \frac{\mathrm{d}z}{4z}\exp\left( \frac{\sigma^3 z^3}{3}  - \sigma z y  \right)
 + \frac{1}{2\I\pi} \int \frac{\mathrm{d}z}{4z}\exp\left( \frac{\sigma^3 z^3}{3}  - \sigma z x  \right) + \frac{\sgn(y-x)}{4},\\
\kernel^{\infty}_{12}(x,y) &=\frac{1}{(2\I\pi)^2} \iint \mathrm{d}z\mathrm{d}w \frac{z-w}{2z(z+w)}\exp\left( \frac{\sigma^3 z^3}{3} + \frac{\sigma^3 w^3}{3} - \sigma z x -\sigma w y \right) \\
& +\frac{1}{2\I\pi}\int \frac{1}{2}\exp\left( \frac{\sigma^3 z^3}{3}  - \sigma  z y  \right),\\
\kernel^{\infty}_{22}(x,y)&= \frac{1}{(2\I\pi)^2} \iint \mathrm{d}z\mathrm{d}w \frac{z-w}{z+w}\exp\left( \frac{\sigma^3 z^3}{3} + \frac{\sigma^3 w^3}{3} - \sigma z x -\sigma w y \right),
\end{align*}
where the contour for $z$ and $w$ is $\mathcal{C}_1^{\pi/3}$ (see Definition \ref{def:GOEdistribution}) in all integrals.

In order to conclude from \eqref{eq:cvoncompacts} that
$$ \Pf\big[\mathsf{J}-\fkernel^{(T)}_{\kappa}\cdot  \kernel^{(T)}\big]_{\ell^2(\mathbb{D}_T)} \xrightarrow[T\to\infty]{} \Pf[\mathsf{J}-\kernel^{\infty}]_{\mathbb{L}^2(\kappa, \infty)},$$
one needs to estimate the entries of $\kernel^{(T)}$ and use Hadamard's bound (see Lemma \ref{lem:hadamard}) to conclude that the Fredholm Pfaffian expansion converges to the desired limit using  dominated convergence theorem. The following bounds are sufficient.
\begin{lemma}
 There exist positive constants $C, c, T_0$ such that for $T>T_0$ and $x, y>\kappa$,
$$\Big\vert T^{1/3} \kernel^{(T)}_{11}(x,y)\Big\vert < C ,\ \
\Big\vert T^{1/3}  \kernel^{(T)}_{12}(x,y)\Big\vert < C \exp\big(-c y \big),\ \
\Big\vert T^{1/3} \kernel^{(T)}_{22}(x,y)\Big\vert < C \exp\big(-c x - c y \big). $$
\label{lem:expobound}
\end{lemma}
\begin{proof}
 These bounds are obtained in a very similar way as in Lemma 5.11 and Lemma 6.4 in \cite{baik2017pfaffian} where the structure of the kernel considered is the same as ours.
Let us explain how the idea works for $\kernel^{(T)}_{22}$.
Using the definition of the rescaled kernel $\kernel^{(T)}$ in \eqref{eq:defKT} and  the formula for $\kernel^{\rm ASEP}_{22}$ in \eqref{eq:K22ASEP}, we have for some constant $C$,
\begin{multline*}
\big\vert  T^{1/3} \kernel^{(T)}_{22}(x,y) \big\vert   \leqslant \\ C T^{2/3}
\Bigg\vert \frac{1}{(2\I\pi)^2} \iint \frac{w-z}{zw(zw-1)}\exp\left(T(\G(z)-\G(w))+2^{-4/3}T^{1/3}(x\log(z) -y\log(w))\right) \mathrm{d}z\mathrm{d}w\Bigg\vert.
\end{multline*}
 The integrations can be restricted to a neighborhood of size $\eta$ around $-1$, where $\eta$ can be arbitrarily small, by doing an error bounded by  $e^{-c_1(\eta) T- c_2(\eta)T^{1/3} x}$ for some constants
$c_1(\eta), c_2(\eta)$. Now we make the change of variables $z=-1+\tilde z T^{-1/3}$ and likewise for $w$.
Using Taylor expansion of $\G$ and approximations of the other factors in the integrand, there exist a constant $C$ such that for $T$ large enough,
\begin{multline*}
\big\vert  T^{1/3} \kernel^{(T)}_{22}(x,y) \big\vert   \leqslant \\
 \frac{C}{(2\I\pi)^2} \iint \bigg\vert \frac{z-w}{z+w}\bigg\vert  \bigg\vert  \exp\left(   \frac{\sigma^3 z^3}{3} + \frac{\sigma^3 w^3}{3} - \sigma  z x -\sigma w y  +CT^{-1/3}(z^4+w^4)  \right)\bigg\vert \mathrm{d}z\mathrm{d}w.
\end{multline*}
Let $\mathcal{C}_a^{\pi/3}$ be the contour formed by two semi-infinite rays  departing $a$ with directions $\pm\pi/3$.
The integration contour in the integral above can be chosen as the intersection of $\mathcal{C}_a^{\pi/3} $ -- for some $a$ that can be freely chosen in $\R_{>0}$ -- with the ball of radius  $T^{1/3}\eta$ around $-1$. Note that the prefactor $T^{2/3}$ canceled with the Jacobian of the change of variables. On this contour,  $T^{-1/3}z^4$ is bounded by $\eta \vert z\vert^3$, so that $\exp\left( T^{-1/3}z^4\right)$  will not compensate the decay of $\exp\left(\Real[z^3]\right)$ for $\eta$ small enough. Moreover the factor $(z-w)/(z+w)$ stays bounded along the contour. Hence there exists a constant $C$ such that
\begin{equation}
\big\vert  T^{1/3} \kernel^{(T)}_{22}(x,y) \big\vert   \leqslant
 \frac{C}{(2\I\pi)^2} \iint \bigg\vert \exp\left(z^3/3+w^3/3 -z x - w y\right)\bigg\vert \mathrm{d}z\mathrm{d}w.
 \label{eq:tobound}
 \end{equation}
For $x$ and $y$ positive, it is easy to conclude from \eqref{eq:tobound} that $\big\vert  T^{1/3} \kernel^{(T)}_{22}(x,y) \big\vert$ has exponential decay in $x$ and $y$.
\end{proof}
To conclude the proof of Theorem \ref{th:GOElimit}, we have to show that
$$ \Pf[\mathsf{J}-\kernel^{\infty}]_{\mathbb{L}^2(\kappa, \infty)}=\Pf[\mathsf{J}-\kernel^{\rm GOE}]_{\mathbb{L}^2(\kappa, \infty)}.$$
Notice that the Fredholm Pfaffian of  $\kernel^{\infty}$ does not depend on $\sigma$ so that  one may take $\sigma=1$. In this case
$$ \kernel^{\infty}_{11}(x,y)=\kernel_{22}^{\rm GOE}(x,y),\ \  \kernel^{\infty}_{12}(x,y)=\kernel^{\rm GOE}_{12}(y,x)= -\kernel_{21}^{\rm GOE}(x,y), \ \ \kernel^{\infty}_{22}(x,y)=\kernel_{11}^{\rm GOE}(x,y).$$
Exchanging the even rows and columns with the odd rows and columns, one multiplies the Pfaffian of $\big(\kernel^{\infty}(x_i, x_k)\big)_{i,j=1}^k$ by $(-1)^k$. Multiplying even rows and even columns by $-1$, the Pfaffian gets multiplied by $(-1)^k$ another time, and one arrives at $\kernel^{\rm GOE}$.
%We conclude that
%$$\Pf\Big[\mathsf{J}-\fkernel_{\scaling{\kappa}}(\scaling(\cdot))\ \kernel^{(T)}\Big]_{\ell^2(\mathbb{D}_T)} \xrightarrow[T\to \infty]{} \Pf\left[  \mathsf{J}-\kernel^{\rm GOE}\right]_{\mathbb{L}^2(\kappa, \infty)}.$$
\end{proof}

\section{KPZ equation on $\R_{+}$}
\label{sec:KPZ}

In \cite{corwin2016open}, the half-space open ASEP with specially tuned weak asymmetry and weak boundary conditions is shown to converge to the half-space KPZ equation with Neumann type  boundary conditions.  That paper deals with a set of initial data which they call ``near equilibrium'', and  a boundary condition that makes the boundary repulsive, in a sense that we shall explain shortly. Both the initial data (``narrow wedge'' type) and boundary condition with which we work are not coverged by the results of \cite{corwin2016open}.
In fact the type of initial data we consider requires a slightly different scaling (a logarithmic correction at the level of the height function, see Remark \ref{rem:differentscaling}).

We show in this section that the observable of ASEP (Definition \ref{def:halflineASEP}) that is expected (in light of the \cite{corwin2016open} results) to converge to the solution of half-space KPZ equation has a weak limit whose distribution can be characterized. 
When the first version of this paper was posted, the identification of this limit with the KPZ equation solution had not yet been proved. However, subsequently, the work of \cite{corwin2016open} was extended by \cite{parekh2017kpz} to include more general initial data (including narrow wedge) and general boundary condition parameter $A\in \R$ (which includes the relevant case for us). We have not modified this section to reflect the work of \cite{parekh2017kpz} (besides adding a few parenthetical notes or footnotes). Combining our Corollary 7.7 with \cite[Theorem 1.2]{parekh2017kpz} yields \cite[Corollary 1.3]{parekh2017kpz} which is the one-point Laplace transform for the half-space KPZ equation from narrow wedge initial data with boundary parameter $A=-1/2$. 

%The KPZ equation can be defined in the a Hopf-Cole sense:  the solution to the KPZ equation is the logarithm of the solution to the stochastic heat equation with multiplicative noise. This is natural from the physical perspective, and it was shown in \cite{gerencser2017singular} that approximations of the KPZ equation (driven by a regularized noise) converge to the KPZ equation in the Hopf-Cole sense.  Let us define more precisely the stochastic heat equation we consider.
\begin{definition}[Half-space Stochastic heat equation]\label{def.mild}
We say that $\mathscr Z(T,X) $ is {\it a mild solution} to the SHE
\begin{equation} \label{e:SHE}
\partial_T \mathscr Z = \tfrac12 \Delta \mathscr Z + \mathscr Z \dot W\;,
\end{equation}
on $\R_+$ with delta initial data at the origin and Robin boundary condition with parameter $A\in \R$
\begin{equation}\label{e:SHERobin0}
\partial_X \mathscr Z(T, X)\Big\vert_{X=0} = A  \mathscr Z(T, 0)  \qquad (\forall T > 0)
\end{equation}
if $\mathscr Z(T, \cdot)$ is adapted to the filtration $\sigma \{\mathscr Z(0,\cdot), W|_{[0,T]}\}$ and
\begin{equation}\label{e:SHE-mild}
 \mathscr Z(T,X) = \mathscr P^R_T(X,0)	 + \int_0^T \!\!\! \int_0^\infty \!\!\!  \mathscr P^R_{T-S} (X,Y) \,\mathscr Z(S,Y) \, \mathrm{d}W_S(\mathrm{d}Y)
\end{equation}
where the last integral is the  It\^o integral with respect to the cylindrical Wiener process $W$,
and $ \mathscr P^R$ is the  heat kernel satisfying the Robin boundary condition
\begin{equation} % [e:PRobin0]
\partial_X \mathscr P^R_T (X,Y)\Big\vert_{X=0} =A  \mathscr P^R_T (0,Y) \qquad (\forall T > 0,Y>0) \;.
\end{equation}
\end{definition}
The Hopf-Cole solution to the half-space KPZ equation with Neumann boundary condition with parameter $A$ is defined to be $\log \mathscr Z$.  It was shown in \cite{corwin2016open} that for  $A\geqslant 0$ the half space SHE admits a unique solution. It has the property that
almost-surely, $\mathscr Z(T,X)>0$ for all $T>0$ and $X\geqslant 0$, so that the logarithm of $\mathscr Z(T,X)$ is well defined. When $A>0$, the heat kernel $\mathscr P^R_T$ corresponds to the transition kernel of a Brownian motion killed at the origin at rate $A$ times the local time. This is why the boundary condition with $A>0$ can be qualified as repulsive boundary condition. When $A<0$, the kernel  $\mathscr P^R_T$ corresponds to the transition kernel of Brownian motions duplicating at the origin at rate $\vert A\vert$ times the local time. When $A=0$, $\mathscr P^R_T$ is the transition kernel of a Brownian motion reflected at the origin.
\subsection{Convergence of ASEP to the SHE}
\label{sec:cvnotationsCS}
In order to relate ASEP with the SHE, we must perform a microscopic version of the Cole-Hopf transform (also called the G\"{a}rtner transform). We will work presently with the notation of \cite{corwin2016open} and then match to that of the present paper.
\begin{definition}[Microscopic Cole-Hopf / G\"{a}rtner transform]
Consider ASEP with left jump rate $\q$, right jump rate $\p$, input rate from the reservoir $\ratealpha$ and output rate to the reservoir $\rategamma$ (see Figure \ref{fig:generalASEP}). Define a height function $h_\tau(x)$ at time $\tau$ and position $x\in \Z_{\geqslant 0}$ by
$$
h_{\tau}(x) = h_{\tau}(0) + \sum_{y=1}^{x} \hat{\eta}_{\tau}(y),
$$
where $h_{\tau}(0)$ is 2 times the net number of particles that are removed (i.e. the number of particles that move into the source minus the number that move out of the source) from site $x=1$ during the time interval $[0,\tau]$ (in particular $h_0(0)=0$), and $\hat{\eta}_{\tau}(x) = 2\eta_{\tau}(x)-1$ is $+1$ if there is a particle at $x$ at time $\tau$ and $-1$ if there is not. For empty  initial data $h_0(x) = -x$.

For $\tau\geqslant 0$ and $x\in \Z_{\geqslant 0}$ define the microscopic Cole-Hopf transform of ASEP to be
$$
Z_\tau(x) = \exp\big(-\uplambda h_\tau(x) + \upnu \tau\big),
$$
where
$$
\uplambda = \frac{1}{2}\log\frac{\q}{\p},\qquad  \upnu = \p+\q-2\sqrt{\p\q}.
$$
\label{def:microHopfCole}
\end{definition}

The reason for the definition of $Z_{\tau}(x)$ is that inside $\Z_{\geqslant 0}$ (i.e. away from the boundary) this transformation of ASEP satisfies a discrete SHE. Assuming further that the boundary rates satisfy \eqref{eq:Liggettscondition}, it is possible to extend the discrete SHE to the boundary of $\Z_{\geqslant 0}$ in terms of a discrete Robin boundary condition (see \cite[Section 3]{corwin2016open}). In particular,
$$
Z_\tau(x) = \sum_{y\in \Z_{\geqslant 0}} p^{R}_{\tau}(x,y)Z_0(y) + \int_0^{\tau} \sum_{y\in \Z_{\geqslant 0}} p^{R}_{\tau-s} (x,y) dM_s(y)
$$
where, for each $y$, $M_s(y)$ are explicit martingales, and $p^R_{\tau}(x,y)$ is the half-line discrete heat kernel with Robin boundary condition $p^{R}_{\tau}(-1,y) = \mu \ p^{R}_{\tau}(0,y)$ (see \cite[Lemma 4.5]{corwin2016open} for an explicit formula for this). The parameter $\mu$ is related to $\ratealpha$ via the equality $\ratealpha = \p^{3/2}(\p^{1/2} - \mu \sqrt{\q}) (\p-\q)^{-1}$ where we also assume \eqref{eq:Liggettscondition}, i.e.  $\ratealpha/\p+\rategamma/\q=1$.

We now introduce weakly asymmetric scaling of half-line ASEP.
\begin{definition}\label{def:wasep}
Introduce a small parameter $\eps>0$ and then scale
$$
\p=\tfrac{1}{2}e^{\sqrt{\eps}}, \qquad \q=\tfrac{1}{2}e^{-\sqrt{\eps}}, \qquad \ratealpha = \frac{\p^{3/2}(\sqrt{\p}-(1-A\eps)\sqrt{\q})}{\p-\q},\qquad \rategamma = \frac{\q^{3/2}(\sqrt{\q}-(1-A\eps)\sqrt{\p})}{\q-\p}.
$$
Write $Z^{\eps}_\tau(x)$ to denote $Z_{\tau}(x)$ with parameters given in terms of the above $\eps$ parameterizations, and write the time-space rescaled version of $Z$ as
\begin{equation}
\mathscr Z^{\eps}(\hat \tau, \hat x)= \eps^{-1/2} Z^{\eps}_{\eps^{-2}\hat{\tau}}(\eps^{-1}\hat{x}).
\label{eq:weakscalingZ}
\end{equation}
\end{definition}
For small $\eps$, we have the approximations
$$
\ratealpha = \tfrac{1}{4}+(\tfrac{3}{8} + \tfrac{1}{4}A)\sqrt{\eps} + \mathcal{O}(\eps), \qquad \rategamma = \tfrac{1}{4}-(\tfrac{3}{8} + \tfrac{1}{4}A)\sqrt{\eps} + \mathcal{O}(\eps),
$$
and
$$
\uplambda = -\sqrt{\eps}, \qquad \upnu =  \frac{\eps}{2} + \frac{\eps^2}{24} + \mathcal{O}(\eps^3).
$$
%For $\bar{T}>0$ fixed, $\mathscr Z^{\e}\in D\big([0,\bar{T}],C(\R_{\geqslant 0})\big)$, the space of CADLAG functions of time, taking values in continuous functions of space $\R_{\geqslant 0}$ (for space, we linearly interpolate).

For ASEP with empty initial data (i.e. $h_0(x) = -x$ for $x\geqslant 0$), as $\eps\to 0$, we expect that $\mathscr Z^{\eps} \Rightarrow \mathscr Z$ as a space time process, where $\mathscr Z$ is the unique mild solution to the SHE on $\R_{\geqslant 0}$ with delta initial data and Robin boundary condition with parameter $A$ at the origin. Analogous result was proved when $A\geqslant 0$ and for near equilibrium initial data as Theorem 2.17 in \cite{corwin2016open}.

\begin{remark}
The scaling considered in \eqref{eq:weakscalingZ} differs from \cite[Definition 2.16]{corwin2016open} by an extra factor $\eps^{-1/2}$ multiplying $Z$. This comes from the fact that the initial condition that we consider is not ``near equilibrium'',  and this extra factor ensures that $Z^{\eps}_{\tau=0}(\cdot)$ converges to the delta function.  This is a very similar situation as in the works of \cite{bertini1997stochastic} and \cite{amir2011probability} regarding the full-space ASEP and KPZ equation. The first case requires a similar ``near equilibrium'' type of initial data, while the second extends it to include step initial data. The basic argument of \cite[Section 3]{amir2011probability} shows that for short time (that is, in the scaling in which ASEP converges to the KPZ equation) the step initial data height function becomes ``near equilibrium''. Then, applying the existing convergence results of \cite{bertini1997stochastic} one gets a consistent family of measures which can be extended back to time 0 and shown to coincide with the desired ``narrow wedge'' initial data KPZ equation. We leave a rigorous proof of this for future work (and hence we will not use the convergence result in the present paper).
\label{rem:differentscaling}
\end{remark}

\subsection{Matching notations}

We now translate the convergence stated in Section \ref{sec:cvnotationsCS} in terms of the notations used throughout the present paper. Let us consider the ASEP process considered in Definition \ref{def:halflineASEP}, that is half-line ASEP with rates $\p=1, \q=t, \ratealpha=1/2, \rategamma=t/2$. Let us set $t=e^{-2\sqrt{\eps}}$ and rescale the time by $2e^{-\sqrt{\eps}}$, so that effectively, the jump rates become
$$ \p = \frac{1}{2}e^{\sqrt{\eps}}, \ \q = \frac{1}{2}e^{-\sqrt{\eps}}, \  \ratealpha = \frac{1}{4}e^{\sqrt{\eps}}, \ \rategamma = \frac{1}{4}e^{-\sqrt{\eps}}.$$
If we match those rates with the rates from Definition \ref{def:wasep},
this corresponds to choosing the boundary parameter $A=-1/2+\mathcal{O}(\eps)$. This is why we cannot apply the results of \cite{corwin2016open}, since $A$ is assumed to be nonnegative there.
Since we have rescaled the time by $2e^{-\sqrt{\eps}}$, we have
$$ h_{\tau}(x) = -2 N_{x+1}(\tau e^{\sqrt{\eps}}/2)-x,$$
in the sense that the space-time processes have the same distribution,
where $N_x(\tau)$ is defined in Definition \ref{def:halflineASEP} and $h_{\tau}(x)$ is defined in Definition \ref{def:microHopfCole}. Then, under the scalings of Definition \ref{def:wasep},
$$ \mathscr Z^{\eps}(\hat \tau, \hat x) = \eps^{-1/2} \exp\Big( -2\sqrt{\eps} N_{1+\eps^{-1}\hat x}(\eps^{-2}\hat \tau e^{\sqrt{\eps}}/2 )  - \eps^{-1/2} \hat x + \eps^{-1}\hat{\tau}/2 + \hat{\tau}/24 + \mathcal{O}(\e)\Big).$$
At this point, it is more convenient to re-parametrize using $\tau=\hat \tau e^{\sqrt{\eps}}$ and  $\e=2\sqrt{\eps}$  so that $t=e^{-\e}$.
\begin{definition}
For $t=e^{-\e}$, under the notations of Definition \ref{def:halflineASEP}, we define the space-time process
$$ \mathcal Z^{\e}_{\tau}(x) := 2 \e^{-1} \exp\Big(- \e N_{1+4x\e^{-2}}(8\e^{-4}  \tau  )  - 2\e^{-1} x + 2\e^{-2}\tau - \tau \e^{-1} + 7\tau/24\Big).$$
\label{def:weaklyscalingclean}
\end{definition}
We expect that $\mathcal Z^{\eps} \Rightarrow \mathscr Z$ as a space time process, where $\mathscr Z$ is the unique mild solution to the SHE on $\R_{\geqslant 0}$ with delta initial data and Robin boundary condition with parameter $A=-1/2$ at the origin.

\subsection{Multiplicative functional of the GOE and KPZ equation on $\R_{+}$}

One can readily adapt Proposition \ref{prop:FredholmHL} and Proposition  \ref{prop:Fredholmcurrent}, changing $-t^x$ into $\zeta t^x$, so that for any
$y\in \R$ and  $\zeta\in \C\setminus \R_{>0}$,
\begin{equation}
\EE\left[\frac{1}{(\zeta t^{y+\lceil \NN(\tau)\rceil_2}, t^2)_{\infty}}\right] =  \Pf\left[  \mathsf{J}+\hkernel_y\cdot  \kernel^{\rm ASEP}\right]_{\ell^2(\Z_{\geqslant 0})},
\label{eq:Laplacecurrent}
\end{equation}
where
\begin{equation}
\hkernel_y(j) = \frac{(\zeta t^{y+j+1}; t^2)_{\infty}}{(\zeta t^{y+j}; t^2)_{\infty}}-1.
\end{equation}
We consider now the asymptotic behaviour of this identity in the weakly asymmetric regime.
\begin{theorem} Under the scalings
$$ t= e^{-\e}, \ \ \tau=\frac{\e^{-3}\hat{\tau}}{1-t}, \ \  $$
the random variable
\begin{equation}
\mathcal{U}_{\e}(\hat \tau)  = \frac{ t^{\big( \NN(\tau)- \frac{\e^{-3}\hat\tau}{4}\big) }}{1-t^2}
\label{eq:scalingxi}
\end{equation}
weakly converges to a random variable $\mathcal{U}(\hat \tau)$ such that for $\zeta<0$,
\begin{equation}
\EE\left[ \exp\left( \zeta\ \mathcal{U}(\hat \tau) \right)  \right]  = \Pf\left[  \mathsf{J}+  \gkernel\cdot  \kernel^{\rm GOE}\right]_{\mathbb{L}^2(\R)}, \ \ \
\label{eq:LaplaceKPZ}
\end{equation}
where
$$ \gkernel(x)  =  \frac{1}{\sqrt{1- \zeta e^{\hat{\sigma}x}}}-1, \quad \text{  and  }\quad \hat{\sigma}=2^{-4/3} (\hat{\tau})^{1/3}.$$
\label{th.LaplaceKPZ}
\end{theorem}

Before proving the Theorem, let us interpret further the result. Notice that we have
$$\bar{\mathcal{U}}_{\e}(T):= \mathcal{U}_{\e}\big((1-t)\e^{-1} T\big) = \frac{\e^{-1}}{2} \exp\Big( -\e N(\e^{-4} T) + \frac{T}{4}\e^{-2} - \frac{T}{8}\e^{-1} + \frac{T}{24} +\mathcal{O}(\e)\Big).$$
One the one hand, $\bar{\mathcal{U}}_{\e}(\tau)$ converges\footnote{The proof of Theorem \ref{th.LaplaceKPZ} is still valid when $\hat\tau $ is not a constant but converges to a constant as $\e$ goes to $0$. Hence by letting $\hat{\tau} = (1-t)\e^{-1} T$, we obtain that $\bar{\mathcal{U}}_{\e}(T)$ converges to $\mathcal{U}(T)$.} to $\mathcal{U}(\tau)$, a certain random variable having a complicated yet explicit distribution. On the other hand, one may check that
$$\mathcal{Z}^{\e}_{\tau}(0) = 4 \ \bar{\mathcal{U}}_{\e}(8 \tau) \exp\big(-\tau/24+\mathcal{O}(\e)\big).$$
Thus, if we define
$$ \mathcal{H}(\tau) = \log\big(4\  \mathcal{U}(8\tau) \big) -\frac{\tau}{24},$$
we expect\footnote{This is now proved in \cite[Theorem 1.2]{parekh2017kpz}.} that $\mathcal{H}(\tau)$ has the same distribution as $\mathscr H(\tau, 0)$ from Definition \ref{def.mild} with parameter $A=-1/2$. Furthermore, the Laplace transform of $e^{\mathcal{H}(\tau)}$ can be expressed as a multiplicative functional of the GOE process.
\begin{corollary}
For any $\tau>0$ and $z > 0$,
$$\EE\left[ \exp\left( \frac{-z}{4} \exp\left(\frac{\tau}{24} + \mathcal H (\tau)\right)\right) \right]   = \EE\left[ \prod_{i=1}^{+\infty} \frac{1}{\sqrt{1+z \exp\left((\tau/2)^{1/3}  \mathfrak{a}_i\right)}} \right],$$
where $\lbrace \mathfrak{a}_i\rbrace_{i=1}^{\infty} $ forms the GOE point process (see Definition \ref{def:GOEdistribution}).
\label{prop:multiplicativefunctional}
\end{corollary}
This corollary is proved in Section \ref{sec:proofcorollary}.

\subsection{Proof of Theorem \ref{th.LaplaceKPZ}}

Our proof proceeds via the following three steps:
\begin{enumerate}
\item If $\ \mathcal{U}_{\e}(\hat\tau)$ weakly converges to $ \mathcal{U}(\hat\tau)$, then the left-hand-side of \eqref{eq:Laplacecurrent} converges to  $\EE[\exp\left( \zeta \mathcal{U}(\hat\tau)\right)]$.
\item The right-hand-side of \eqref{eq:Laplacecurrent} converges to $\Pf\left[  \mathsf{J}+  \gkernel\cdot  \kernel^{\rm GOE}\right]_{\mathbb{L}^2(\R)}$.
\item The sequence $\mathcal{U}_{\e}(\hat\tau)$ indeed weakly converges to  $ \mathcal{U}(\hat\tau)$ whose distribution is determined by \eqref{eq:LaplaceKPZ}.
\end{enumerate}

\textit{Step (1):} Let us examine the scaling limit of the left-hand-side in \eqref{eq:Laplacecurrent}.
We need a Lemma about asymptotics of $q$-Pochhammer symbols.
\begin{lemma}
If two positive real numbers $\xi$ and $\theta$ are related by
$$ \theta = \frac{t^{\xi}}{1-t^2} $$
then as $t$ goes to $1$,
$$ \frac{1}{ \big(\zeta t^{\xi}; t^2\big)_{\infty}} \xrightarrow[]{} \exp( \zeta \theta),$$
uniformly for $\theta$ in a compact subset of $\R_{\geqslant 0}$ and for $\zeta$ in  $\R_{\leqslant 0}$.
\label{lem:limPochhammer}
\end{lemma}
\begin{proof}
Using the $q$-binomial theorem with $q=t^2$,
\begin{equation}
 \frac{1}{ \big(\zeta t^{\xi}; t^2\big)_{\infty}} = \sum_{k=0}^{\infty} \frac{\zeta^k \theta^k (1-t^2)^k}{(t^2; t^2)_k}.
 \label{eq:expansionPochhammer}
\end{equation}
This expansion is absolutely convergent for any $\zeta$ in a compact set and $t$ close enough to $1$.
We have
$$ \frac{(1-t^2)^k}{(t^2, t^2)_k} \xrightarrow[t \to 1]{} \frac{1}{k!}.$$
We may use dominated convergence to conclude that the right hand side of \eqref{eq:expansionPochhammer} converges to
$$ \sum_{k=0}^{\infty} \zeta^k\frac{\xi^k}{k!} = \exp(\zeta \xi)$$
as desired. It is easy to control the error made in order to show that the convergence is uniform as $\theta$ varies in a compact subset of $\R_{\geqslant 0}$ and $\zeta \in \R_{\leqslant 0}$.
\end{proof}

Let us scale $y$ in  \eqref{eq:Laplacecurrent} as $ y=-\e^{-3}\hat{\tau}/4 $ and apply Lemma \ref{lem:limPochhammer} with $\xi_{\e}= y+ \lceil \NN(\tau)\rceil_2$ and $\theta=\mathcal{U}_{\e}(\tilde \tau)$.
If $\mathcal{U}_{\e}(\hat\tau) \xrightarrow[\e \to 0]{} \mathcal{U}(\hat\tau)$, then
$$ \frac{1}{( \zeta t^{y+\lceil \NN(\tau)\rceil_2}, t^2)_{\infty}} \xrightarrow[\e \to 0]{} \exp\left( \zeta \mathcal{U}(\hat\tau)\right).$$
The left-hand-side belongs to the interval $(0,1)$, so that  if $\mathcal{U}_{\e}(\hat\tau)$ weakly converges to $ \mathcal{U}(\hat\tau)$, then
\begin{equation}
 \EE\left[\frac{1}{( \zeta t^{y+\lceil \NN(\tau)\rceil_2}, t^2)_{\infty}}\right] \xrightarrow[\e \to 0]{} \EE[\exp\left( \zeta \mathcal{U}(\hat\tau)\right)]
\label{eq:approxLaplace}
\end{equation}
uniformly for $\zeta \in \R_{\geqslant 0}$.

\textit{Step (2):} Now we examine the limit of the right-hand-side in \eqref{eq:Laplacecurrent}. We will rescale the kernel using $\scaling{x} = \e^{-3}\hat{\tau}/4 - 2^{-4/3}\e^{-1} \hat{\tau}^{1/3} x $. Let us first examine the scaling limit of $\hkernel_y$,
under the scalings of Theorem \ref{th.LaplaceKPZ}.
\begin{lemma} For a fixed $\zeta\in \C\setminus \R_{>0} $, and with $ y=-\e^{-3}\hat{\tau}/4$,
$$ \hkernel_{y}(\scaling{x})  \xrightarrow[\e \to 0]{}  \frac{1}{\sqrt{1-\zeta e^{\hat{\sigma}x}}}-1 = \gkernel(x).$$
\label{lem:limhkernel}
\end{lemma}
\begin{proof}
Let
$$ p(a)  = \frac{(\zeta t^{-\e^{-1}a+1}; t^2)_{\infty}}{(\zeta t^{-\e^{-1}a}; t^2)_{\infty}}.$$
We have
$$p(a)p(a+\e) =\frac{1}{1-\zeta t^{-\e^{-1}a}} \xrightarrow[\e \to 0]{} \frac{1}{1-\zeta e^a},$$
so that
$$ p(a)  \xrightarrow[\e \to 0]{} \frac{1}{\sqrt{1-\zeta e^a}}.$$
\end{proof}

The convergence of $\kernel^{\rm ASEP}$ to $\kernel^{\rm GOE}$ in the integrated sense of \eqref{eq:cvoncompacts} is already proved since the scalings are exactly the same as in the proof of Theorem \ref{th:GOElimit}, with $T= \epsilon^{-3}\hat{\tau}$. The only thing left to check is that dominated convergence theorem applies. More precisely, we need to show that
 for any nonnegative integer $k$, under the scalings $ y=-\e^{-3}\hat{\tau}/4$ , $T=\epsilon^{-3}\hat{\tau}$, the function
 \begin{equation}
 (x_1, \dots, x_k) \mapsto  \prod_{i=1}^k \hkernel_y(\scaling{x_i})   \cdot \Pf\left[ \kernel^{T}(x_i, x_j) \right]_{i,j=1}^k
 \label{eq:KPZquantitytobound}
 \end{equation}
is absolutely summable on the domain\footnote{that is the preimage of $\Z_{\geqslant 0}$ by the scaling function $\scaling{\cdot }$.} $\mathbb{D}_T$ of the Fredholm Pfaffian,  uniformly as $\e \to 0$.
Since  $\vert  \hkernel  \vert $ is bounded by $1$,  using Lemma \ref{lem:expobound} to control $\kernel^{T}(x, y)$ as $x,y$ approach $+\infty$, we know that \eqref{eq:KPZquantitytobound} is uniformly absolutely summable on sets of the form $(\kappa, +\infty)^k$. The two next lemmas provide bounds for $ \hkernel_y(\scaling{x}) $ and $\kernel^{(T)}(x,y)$ for $x,y$ near $-\infty$,  which prove that \eqref{eq:KPZquantitytobound} is absolutely summable on the whole domain $\mathbb{D}_T$.
\begin{lemma}For $ y=-\e^{-3}\hat{\tau}/4$,   $\zeta<0$ and any $x\in \R$,
	$$ \big\vert \hkernel_y(\scaling{x})\big\vert \leqslant \vert \zeta\vert e^x .$$
	\label{lem:expoboundh}
\end{lemma}
\begin{proof} When $\zeta<0$,  the function $ \hkernel_{y}(\scaling{x})$ is increasing in the variable $x$ and stays in
$(-1,0)$, so that
	$$ \big(\hkernel_{y}(\scaling{x})+1\big)^2 \geqslant \big(\hkernel_y(\scaling{x})+1\big)\big(\hkernel_y(\scaling{x}-\epsilon)+1\big) =\frac{1}{1-\zeta t^{-\e^{-1}x+1}}.$$
	It implies that
	$$ -\hkernel_y(\scaling{x}) \big(\hkernel_y(\scaling{x})+2\big)\leqslant \frac{-\zeta t^{-\e^{-1}x+1}}{1-\zeta t^{-\e^{-1}x+1}} .  $$	
Hence, since $\zeta<0$ and $\hkernel\in (-1,0)$,
$$ 0\leqslant -\hkernel_y(\scaling{x}) \leqslant -\zeta t^{-\e^{-1}x+1} \leqslant \vert \zeta \vert e^{x}.$$
\end{proof}

\begin{lemma}
There exist positive constants $c, C, T_0$ such that for $T>T_0$,
\begin{equation}
\Big\vert T^{1/3} \kernel^{(T)}_{11}(x,y)\Big\vert < \vert x y \vert ^c  ,\ \ \
\Big\vert T^{1/3}  \kernel^{(T)}_{12}(x,y)\Big\vert <  \vert x y \vert ^c\ \text{ and } \
\Big\vert T^{1/3} \kernel^{(T)}_{22}(x,y)\Big\vert <  \vert x y \vert ^c.
\label{eq:polyboundskernel}
\end{equation}
\label{lem:expoboundKPZ}
\end{lemma}
\begin{proof}
If these bounds hold uniformly as $T$ goes to infinity, they should in particular hold for $\kernel^{\rm GOE}$. Indeed, the entries in the kernel $\kernel^{\rm GOE}$ can be written as integrals of the Airy function and its derivative (see for instance the formulas in \cite{tracy2005matrix} which are equivalent to   \eqref{eq:KGOE}). The Airy function can be defined as
$$ \mathrm{Ai}(x) = \int_{\mathcal{C}_1^{\pi/3}} e^{z^3/3-zx}\mathrm{d}z,$$
where the contour $\mathcal{C}_1^{\pi/3}$ is formed by the union of two semi-infinite rays departing $1$ with angles $\pi/3$ and $-\pi/3$.
 Hence, the bounds \eqref{eq:polyboundskernel} when $T=\infty$ may be deduced from  the following bound on  the Airy function and its derivative,
$$ \vert \mathrm{Ai}(x) \vert \leqslant C \vert x\vert^{-1/4}, \ \ \vert \mathrm{Ai}'(x) \vert \leqslant C \vert x\vert^{1/4}, $$
which are classical  \cite[Chapter 11, (1.08), (1.09)]{olver1974asymptotics} and can be proved by saddle point analysis.

In order to prove the bounds  \eqref{eq:polyboundskernel} for the kernel $\kernel^{(T)}$ uniformly in $T$, we will likewise reduce the problem to the study of simpler functions that play an analogous role as  the Airy function and its derivative for the kernel $\kernel^{\rm ASEP}$.  The contours in the definition of  $\kernel^{\rm ASEP}$ in \eqref{eq:KASEP} can all be chosen as circles with radius less than $1$, so that using the expansions
$$ \frac{1}{1-zw} = \sum_{k=0}^{\infty}(zw)^k, \ \ \frac{1}{1-z^2} = \sum_{k=0}^{\infty}z^{2k}, \ \ \frac{1}{1-w^2} = \sum_{k=0}^{\infty}w^{2k}, $$
all entries of the kernel $(-1)^{u+v}\kernel^{\rm ASEP}(u,v)$ can be written using sums of the functions
$$ A(u) = \frac{(-1)^u}{2\I\pi} \int \frac{g(z)}{z^u}\mathrm{d}z,$$
and
$$ B(u) = \frac{(-1)^u}{2\I\pi} \int \frac{(z+1)g(z)}{z^u}\mathrm{d}z,$$
where
$$ g(z) = \exp\left(\frac{T}{2} \frac{z+1}{z-1}\right) $$
and the contour is a circle around the origin. More precisely, $A(u)$ plays an analogous role to the Airy function while $B(u)$ plays an analogous role to the derivative ot the Airy function. Then the kernel $\kernel^{\rm ASEP}$ has the same structure as $\kernel^{\rm GOE}$ provided Airy functions are replaced by their discrete analogues, and integrals are replaced by sums.
Since we are interested in the rescaled kernel $\kernel^{(T)}$ instead of $\kernel^{\rm ASEP}$, we are reduced to studying
\begin{equation}
 A_T(x):= T^{1/3}A(\scaling{x})  =   T^{1/3}\frac{(-1)^{\scaling{x}}}{2\I\pi} \int\exp\left(T\G(z)+2^{-4/3}T^{1/3}x\log(z) \right)\mathrm{d}z,\label{eq:defAT}
\end{equation}
where, as before,
$$ \G(z) = \frac{1}{2}\frac{z+1}{z-1} - \frac{1}{4}\log(z),$$
and $B_T(x):= T^{2/3}B(\scaling{x})$.
Both functions $A_T$ and $B_T$ have exponential decay at $+\infty$, for the same reasons as in Lemma \ref{lem:expobound}. Hence if we show polynomial bounds as $x$ approaches $-\infty$ for $A_T$ and $B_T$, then sums of the form
 $$T^{-1/3}\sum_{y=x}^{\infty }A_T(y), \ \ T^{-1/3}\sum_{y=x}^{\infty }B_T(y)$$
where the summations are over $y\in \mathbb{D}_T, y>x$,  will also satisfy polynomial bounds and $\kernel^{(T)}$ will satisfy \eqref{eq:polyboundskernel}. Thus, we are left with  showing that
$$ \vert A_T(x) \vert  \leqslant C \vert x\vert^{c_A}, \ \ \ \vert B_T(x) \vert  \leqslant C \vert x\vert^{c_B}$$
for some constants  $C,c_A, c_B>0$. The optimal exponents seem to be  $c_a= -1/4$ and $c_B=1/4$ -- as for the Airy function and its derivative -- but we will not need such precision.  Since the arguments are very similar for $A_T(x)$ and $B_T(x)$ we will only explain how the bound is obtained for $A_T$.

We will distinguish two cases, whether $T>r x^{6}$ or $T\leqslant r x^{6}$, where $r$ is a constant that we will optimize later to facilitate our analysis. First notice that the circle of radius $1$ centered at $0$ is a contour line for $\Real[\G(z)]$, this will be used in both cases.

\textbf{Case $T\leqslant r x^{6}$:}
In this case, we may integrate \eqref{eq:defAT} on a contour of radius $1-\eta$ for a small $\eta>0$. On such a contour, the real part of $\G(z)$ tends to $0$ as $\eta$ goes to $0$. We can bound $A_T(x)$ by
$$ \big\vert A_T(x) \big\vert \leqslant  T^{1/3} \frac{1}{2\I\pi} \int  \exp(T\ \Real[\G(z)]+ 2^{-4/3}T^{1/3}x\log(\vert z\vert)) \mathrm{d}z. $$
Since this holds for any $\eta>0$,  one can let $\eta$ go to zero and conclude that since the integral goes to $1$ as $\eta$ goes to $0$,
$$ \big\vert A_T(x) \big\vert \leqslant  T^{1/3} \leqslant r^{1/3} \vert x\vert^{2} .$$

\textbf{Case $T>r x^{6}$:}
  Let $\eta>0$ be a small parameter that we will chose more precisely later. In \eqref{eq:defAT}, the contour can be deformed to be the contour $\mathcal{C}$ depicted on Figure \ref{fig:contour}, i.e. a contour formed by two segments of length $\eta$ departing $-1$ at angles $\pm\pi/3$, and an arc of circle around $0$ joining the extremities of these two segments.

 We treat separately the contributions of the integral in a neighborhood of size $\eta$ around $-1$ and the contribution outside of that neighbourhood. Let us write $A_T(x) = I_1(x)+ I_2(x)$ according to this partition of the integration contour, that is, $I_1(x)$ is the integral in \eqref{eq:defAT} where the integration is restricted to a neighborhood of $-1$ of size $\eta$, while $I_2(x)$ equals the integration outside of that neighborhood (see Figure \ref{fig:contour}). Outside of a $\eta$-neighborhood along the contour $\mathcal{C}$, we may use the bounds
$$ \vert \exp(T\G(z))  \vert < \exp(- c_1 \eta^3 T)$$
for some fixed constant  $c_1$ (which does not depend on $\eta$), and
$$  \vert \exp(2^{-4/3}T^{1/3}x\log(z))  \vert < \exp( c_2 \eta T^{1/3} \vert x\vert) $$
for some fixed constant $c_2$. Taking into account the prefactor $T^{1/3}$, we hence find that
$$\big\vert I_2(x) \big\vert< T^{1/3} \exp\big( - c_1 \eta^3 T+ c_2 \eta T^{1/3} \vert x\vert \big).$$
%Taking the maximum over $T\in (0, \infty)$, we find a quantity
%$$ \big\vert I_2(x) \big\vert < C_1/\eta$$
%for some fixed constant $C_1$.

Now we analyze $I_1(x)$. We make the change of variables $z=-1+2^{4/3}\tilde z T^{-1/3}$.  Using Taylor approximations of the functions $\G$ and $\log$, we arrive at
\begin{equation*}
 I_1(x) = \frac{C_2}{2\I\pi} \int_{\mathcal{C}_0^{\pi/3}} \exp\left(   \frac{ \tilde z^3}{3} -  \tilde z x +E(\tilde{z}, T, x) \right) \mathrm{d}\tilde z,
\end{equation*}
where $C_2$ is a fixed constant , $E(\tilde{z}, T, x)$ is an error term and the new contour can be chosen as $\mathcal{C}_0^{\pi/3}$, restricted to $\tilde z$ such that $\vert \tilde z \vert <2^{-4/3}T^{1/3}\eta$. Using Taylor approximations, this error term can be bounded by
$$\big\vert  E(\tilde{z}, T, x) \big\vert < \eta \vert \tilde z\vert^3 + x \eta \vert \tilde z \vert.$$

At this point, we can adapt standard techniques used to estimate the Airy function $\Ai(x)$ on the negative reals when $x$ goes to $-\infty$. We refer for instance to Section 5 in \cite{corwin2013continuum}, and more specifically after Equation (5.6) therein.
Let us make the change of variables $\tilde z= \sqrt{\vert x\vert} z$. For negative $x$, the integral becomes
\begin{equation*}
I_1(x)= \frac{C_2\vert x \vert^{1/2}}{2\I\pi} \int \exp\left(  \vert x\vert^{3/2} \big(\frac{z^3}{3} + z\big) +F(z, T, x) \right) \mathrm{d}\tilde z,
\end{equation*}
where the integration contour has been changed according to the change of variables and the error term  can be bounded by
$$ \vert F(z, T, x) \vert < \eta \vert x\vert^{3/2} \big(  \vert z\vert^3 + \vert z\vert\big).$$
The function $z\mapsto \frac{z^3}{3} + z$ has two critical points at $z=\pm\I$ which yield contributions of the same order. One can make a new saddle point approximation around both these saddle points and control the approximations using the same steps as in \cite{corwin2013continuum}.
Without the error term $F(z, T, x)$, we could bound the integral $I_1(x)$ by a constant as in \cite{corwin2013continuum}. Taking into account the error term  which grows as $\vert x\vert $ goes to infinity, we can only arrive at
$$ \big\vert I_1(x) \big\vert< C_3 \exp\big(c_3\eta\vert x\vert^{3/2}\big),$$
for some constants\footnote{If we were bounding the function $B_T(x)$ instead of $A_T(x)$, the constant $C_3$ should be replaced by $C_3\vert x\vert^{c_4}$ for a certain exponent $c_4$. Although this term grows as $\vert x\vert$ goes to infinity, the rest of the argument can be adapted as for $A_T$.} $C_3, c_3>0$.

Summing the contributions of the integral in \eqref{eq:defAT} inside and outside of a $\eta$-neighborhood, we arrive at
$$  \big\vert A_T(x)  \big\vert = \big\vert I_1(x) + I_2(x) \big\vert < C_3 \exp\big(c_3 \eta\vert x\vert^{3/2}\big) + T^{1/3} \exp\big( - c_1 \eta^3 T+ c_2 \eta T^{1/3} \vert x\vert \big).$$
Recall that $\eta $ is a real number that can be taken arbitrarily small, and the constants $c_1, c_2$ and $C_3$ are independent from $\eta$. Letting $\eta=1/\vert x\vert^{3/2}$, the first term is bounded and the second term becomes
$$ T^{1/3} \exp\big( - c_1  T/\vert x\vert^{9/2}+ c_2 T^{1/3}/ \vert x\vert^{1/2} \big).$$
We need to maximize this quantity over $T\in ( r \vert x\vert^{6}, +\infty)$. Let us write $T=\vert x\vert^{6}U$ so that we are interested in
$$\max_{U\in (r, +\infty)} \Big\lbrace \vert x\vert^{2}U^{1/3} \exp\big( \vert x\vert^{3/2}(-c_1 U+ c_2 U^{1/3}) \big)  \Big\rbrace.$$
One may choose $r$ large enough so that for any $U>r$,
$$ -c_1 U+ c_2 U^{1/3}<-c_1 U/2, $$
in which case
\begin{align*}
\max_{U\in (r, +\infty)} \Big\lbrace \vert x\vert^{2}U^{1/3} \exp\big( \vert x\vert^{3/2}(-c_1 U+ c_2 U^{1/3}) \big)  \Big\rbrace &< \vert x\vert^{2}\ \max_{U\in (0, +\infty)} \Big\lbrace U^{1/3} \exp\big( - \vert x\vert^{3/2} c_1 U/2 \big)  \Big\rbrace,\\
&<C_4,
\end{align*}
for some constant $C_4$.

Finally we have found that in any case
$$  \big\vert A_T(x)  \big\vert < \max\Big\lbrace  r^{1/3}\vert x\vert^2,  C_3 e^{c_3},  C_4\Big\rbrace <C\vert x\vert^2$$
for some constant $C$, which is the desired polynomial growth estimate.
\end{proof}

\textit{Step (3):} We employ the same argument as in Step 3 of the proof of Theorem 11.6 in \cite{borodin2016asep}, that we reproduce here.\footnote{Although the proof of Theorem 11.6 in \cite{borodin2016asep} is only a sketch of proof, the step 3 is rigorous. One can also prove Step 3 by adapting the arguments in \cite[Example 5.5]{billingsley1968convergence}.}  Lemma \ref{lem:limPochhammer} shows that under the same scalings as above,
\begin{equation}
 \EE\left[\frac{1}{( \zeta t^{y+\lceil \NN(\tau)\rceil_2}, t^2)_{\infty}}\right] - \EE\left[ \exp(\zeta \mathcal{U}_{\e}(\hat \tau))\right]
 \label{eq:almostLaplace}
 \end{equation}
goes to $0$ as $\e\to 0$ uniformly for $\zeta \in \R_{\leqslant 0}$.
By the compactness of the space of positive measures in the weak topology, the distribution of $\mathcal{U}_{\e}(\hat \tau)$  must have limiting points as $\e\to 0$ in the space of positive measures on $\R_{\geqslant 0}$ of total mass at most $1$. Let $\mu$ be one such limiting point. We have seen in Step 2 that
$$\EE\left[\frac{1}{( \zeta t^{y+\lceil \NN(\tau)\rceil_2}, t^2)_{\infty}}\right] \xrightarrow[\e \to 0]{} \Pf\left[  \mathsf{J}+  \gkernel\cdot  \kernel^{\rm GOE}\right]_{\mathbb{L}^2(\R)},$$
which together with \eqref{eq:almostLaplace} implies that
\begin{equation}
 \int_{\R_{\geqslant 0}}\exp(\zeta y) \mu(\mathrm{d}y) =   \Pf\left[  \mathsf{J}+  \gkernel\cdot  \kernel^{\rm GOE}\right]_{\mathbb{L}^2(\R)}.
\label{eq:Laplacemu}
\end{equation}
Since positive measure with mass at most $1$ are uniquely determined by their Laplace transform, the limiting point $\mu$ is unique. Finally, since the right-hand-side of \eqref{eq:Laplacemu} goes to $1$ as $\zeta$ goes to zero, $\mu$ is actually a probability measure. Denoting $\mathcal{U}(\hat \tau)$ a random variables with distribution $\mu$,  $\mathcal{U}_{\e}(\hat \tau)$ weakly converges to $\mathcal{U}(\hat \tau)$ with,
for $\zeta\leqslant 0$,
\begin{equation*}
\EE\left[ \exp\left( \zeta\ \mathcal{U}(\hat \tau) \right)  \right]  = \Pf\left[  \mathsf{J}+  \gkernel\cdot  \kernel^{\rm GOE}\right]_{\mathbb{L}^2(\R)}, \ \ \
\end{equation*}
\subsection{Proof of Corollary \ref{prop:multiplicativefunctional}}
\label{sec:proofcorollary}
Since $1+\gkernel(x) = \frac{1}{\sqrt{1+\zeta e^{\sigma  x}}}$,
\begin{equation}
 \Pf\left[  \mathsf{J}+  \gkernel\cdot  \kernel^{\rm GOE}\right]_{\mathbb{L}^2(\R)} = \EE\left[ \prod_{i=1}^{+\infty} \frac{1}{\sqrt{1+\zeta  e^{\sigma  \mathfrak{a}_i}}} \right].
\label{eq:multiplicativefunctional}
\end{equation}
provided both sides are absolutely convergent. We know from the proof of Theorem \ref{th.LaplaceKPZ} that the left-hand-side of \eqref{eq:multiplicativefunctional} is absolutely convergent. One can show that the infinite product in the right-hand-side is absolutely convergent as well using the same estimates on $\kernel^{\rm GOE}$.

 Recall that $\sigma=2^{-4/3}\hat \tau^{1/3}$.  In light of \eqref{eq:multiplicativefunctional},  the statement of Corollary \ref{prop:multiplicativefunctional} is just a reformulation of Theorem \ref{th.LaplaceKPZ}.

%\bibliography{../mainbiblio.bib}

\begin{thebibliography}{DEHP93}
	
	\bibitem[Agg17]{aggarwal2017convergence}
	A.~Aggarwal.
	\newblock Convergence of the stochastic six-vertex model to the {ASEP}.
	\newblock {\em Math. Phys. Anal. Geom.}, 20(2):1--20, 2017.
	
	\bibitem[Agg18]{aggarwal2018current}
	A.~Aggarwal.
	\newblock Current fluctuations of the stationary {ASEP} and six-vertex model.
	\newblock {\em Duke Math. J.}, 167(2):269--384, 02 2018.
	
	\bibitem[AB16]{aggarwal2016phase}
	A.~Aggarwal and A.~Borodin.
	\newblock Phase transitions in the {ASEP} and stochastic six-vertex model.
	\newblock {\em arXiv preprint arXiv:1607.08684}, 2016.
	
	\bibitem[ACQ11]{amir2011probability}
	G.~Amir, I.~Corwin, and J.~Quastel.
	\newblock Probability distribution of the free energy of the continuum directed
	random polymer in 1 + 1 dimensions.
	\newblock {\em Comm. Pure Appl. Math.}, 64(4):466--537, 2011.
	
	\bibitem[BBCS16]{baik2017pfaffian}
	J.~Baik, G.~Barraquand, I.~Corwin, and T.~Suidan.
	\newblock {Pfaffian Schur processes and last passage percolation in a
		half-quadrant}.
	\newblock {\em to appear in Ann. Probab., arXiv:1606.00525}, 2016.
	
	\bibitem[BR01a]{baik2001algebraic}
	J.~Baik and E.~M. Rains.
	\newblock Algebraic aspects of increasing subsequences.
	\newblock {\em Duke Math. J.}, 109(1):1--65, 2001.
	
	\bibitem[BR01b]{baik2001asymptotics}
	J.~Baik and E.~M. Rains.
	\newblock The asymptotics of monotone subsequences of involutions.
	\newblock {\em Duke Math. J.}, 109(2):205--281, 2001.
	
	\bibitem[BBC]{barraquand2018half}
	G.~Barraquand, A.~Borodin, and I.~Corwin.
	\newblock Half-space {M}acdonald processes.
	\newblock {\em arXiv preprint arXiv:1802.08210}.
	
	\bibitem[Bax82]{baxter1982exactly}
	R.~J. Baxter.
	\newblock {\em Exactly solvable models in statistical mechanics}.
	\newblock Academic Press London, 1982.
	
	\bibitem[BG97]{bertini1997stochastic}
	L.~Bertini and G.~Giacomin.
	\newblock Stochastic {B}urgers and {KPZ} equations from particle systems.
	\newblock {\em Comm. Math. Phys.}, 183(3):571--607, 1997.
	
	\bibitem[BWZJ15]{betea2015refined}
	D.~Betea, M.~Wheeler, and P.~Zinn-Justin.
	\newblock {Refined Cauchy/Littlewood identities and six-vertex model partition
		functions: II. Proofs and new conjectures}.
	\newblock {\em J. Alg. Comb.}, 42(2):555--603, 2015.
	
	\bibitem[Bil68]{billingsley1968convergence}
	P.~Billingsley.
	\newblock {\em Convergence of probability measures}.
	\newblock John Wiley \& Sons, 1968.
	
	\bibitem[Bor16]{borodin2016stochastic}
	A.~Borodin.
	\newblock Stochastic higher-spin six-vertex model and {M}acdonald measures.
	\newblock {\em to appear in J. Math. Phys., arXiv:1608.01553}, 2016.
	
	\bibitem[Bor17]{borodin2017family}
	A.~Borodin.
	\newblock On a family of symmetric rational functions.
	\newblock {\em Adv. Math.}, 306:973--1018, 2017.
	
	\bibitem[BBC16]{borodin2016directed}
	A.~Borodin, A.~Bufetov, and I.~Corwin.
	\newblock Directed random polymers via nested contour integrals.
	\newblock {\em Ann. Phys.}, 368:191--247, 2016.
	
	\bibitem[BBW16]{borodin2016between}
	A.~Borodin, A.~Bufetov, and M.~Wheeler.
	\newblock Between the stochastic six vertex model and {Hall-Littlewood}
	processes.
	\newblock {\em J. Comb. Th. A, arXiv:1611.09486}, 2016.
	
	\bibitem[BC14]{borodin2014macdonald}
	A.~Borodin and I.~Corwin.
	\newblock Macdonald processes.
	\newblock {\em Probab. Theory and Rel. Fields}, 158(1-2):225--400, 2014.
	
	\bibitem[BCF14]{borodin2012free}
	A.~Borodin, I.~Corwin, and P.~Ferrari.
	\newblock Free energy fluctuations for directed polymers in random media in 1+
	1 dimension.
	\newblock {\em Comm. Pure Appl. Math.}, 67(7):1129--1214, 2014.
	
	\bibitem[BCG16]{borodin2016stochasticsix}
	A.~Borodin, I.~Corwin, and V.~Gorin.
	\newblock Stochastic six-vertex model.
	\newblock {\em Duke Math. J.}, 165(3):563--624, 2016.
	
	\bibitem[BCPS15]{borodin2014spectral}
	A.~Borodin, I.~Corwin, L.~Petrov, and T.~Sasamoto.
	\newblock Spectral theory for interacting particle systems solvable by
	coordinate {B}ethe ansatz.
	\newblock {\em Comm. Math. Phys.}, 339(3):1167--1245, 2015.
	
	\bibitem[BO17]{borodin2016asep}
	A.~Borodin and G.~Olshanski.
	\newblock The {ASEP} and determinantal point processes.
	\newblock {\em to appear in Commun. Math, Phys. arXiv:1608.01564},
	353(2):853--903, Jul 2017.
	
	\bibitem[BP16]{borodin2016higher}
	A.~Borodin and L.~Petrov.
	\newblock Higher spin six vertex model and symmetric rational functions.
	\newblock {\em to appear in Selecta Mathematica, arXiv:1601.05770}, 2016.
	
	\bibitem[BR05]{borodin2005eynard}
	A.~Borodin and E.~M. Rains.
	\newblock {Eynard--Mehta theorem, Schur process, and their Pfaffian analogs}.
	\newblock {\em J. Stat. Phys.}, 121(3-4):291--317, 2005.
	
	\bibitem[CDR10]{calabrese2010free}
	P.~Calabrese, P.~Le Doussal, and A.~Rosso.
	\newblock Free-energy distribution of the directed polymer at high temperature.
	\newblock {\em EPL (Europhysics Letters)}, 90(2):20002, 2010.
	
	\bibitem[CW07]{corteel2007tableaux}
	S.~Corteel and L.~K. Williams.
	\newblock Tableaux combinatorics for the asymmetric exclusion process.
	\newblock {\em Adv. appl. math.}, 39(3):293--310, 2007.
	
	\bibitem[CW10]{corteel2010staircase}
	S.~Corteel and L.~K. Williams.
	\newblock Staircase tableaux, the asymmetric exclusion process, and
	{Askey--Wilson} polynomials.
	\newblock {\em Proc. Natl. Acad. Sci.}, 107(15):6726--6730, 2010.
	
	\bibitem[CW11]{corteel2011tableaux}
	S.~Corteel and L.~K. Williams.
	\newblock Tableaux combinatorics for the asymmetric exclusion process and
	{Askey--Wilson} polynomials.
	\newblock {\em Duke Math. J.}, 159(3):385--415, 2011.
	
	\bibitem[Cor12]{corwin2012kardar}
	I.~Corwin.
	\newblock The {K}ardar--{P}arisi--{Z}hang equation and universality class.
	\newblock {\em Random Matrices Theory Appl.}, 1(01), 2012.
	
	\bibitem[CD17]{corwin2017transversal}
	I.~Corwin and E.~Dimitrov.
	\newblock Transversal fluctuations of the {ASEP}, stochastic six vertex model,
	and {Hall-Littlewood Gibbsian} line ensembles.
	\newblock {\em arXiv preprint arXiv:1703.07180}, 2017.
	
	\bibitem[CQR13]{corwin2013continuum}
	I.~Corwin, J.~Quastel, and D.~Remenik.
	\newblock Continuum statistics of the {A}iry 2 process.
	\newblock {\em Comm. Math. Phys.}, pages 1--16, 2013.
	
	\bibitem[CS16]{corwin2016open}
	I.~Corwin and H.~Shen.
	\newblock Open {ASEP} in the weakly asymmetric regime.
	\newblock {\em Comm. Pure Appl. Math., arXiv:1610.04931}, 2016.
	
	\bibitem[DEHP93]{derrida1993exact}
	B.~Derrida, M.~R. Evans, V.~Hakim, and V.~Pasquier.
	\newblock Exact solution of a {1D} asymmetric exclusion model using a matrix
	formulation.
	\newblock {\em J. Phys. A}, 26(7):1493, 1993.
	
	\bibitem[Dot10]{dotsenko2010replica}
	V.~Dotsenko.
	\newblock Replica {B}ethe ansatz derivation of the {T}racy–{W}idom
	distribution of the free energy fluctuations in one-dimensional directed
	polymers.
	\newblock {\em J. Stat. Mech.}, 2010(07):P07010, 2010.
	
	\bibitem[DMZ14]{duhart2014semi}
	H.~G. Duhart, P.~M{\"o}rters, and J.~Zimmer.
	\newblock The semi-infinite asymmetric exclusion process: Large deviations via
	matrix products.
	\newblock {\em arXiv preprint arXiv:1411.3270}, 2014.
	
	\bibitem[Gro04]{grosskinsky2004phase}
	S.~Grosskinsky.
	\newblock {\em Phase transitions in nonequilibrium stochastic particle systems
		with local conservation laws}.
	\newblock PhD thesis, PhD thesis, TU Munich, 2004.
	
	\bibitem[GLD12]{gueudre2012directed}
	T.~Gueudr{\'e} and P.~Le~Doussal.
	\newblock Directed polymer near a hard wall and {KPZ} equation in the
	half-space.
	\newblock {\em Europhysics Letters}, 100(2):26006, 2012.
	
	\bibitem[HHT15]{halpin2015kpz}
	T.~Halpin-Healy and K.~A. Takeuchi.
	\newblock A {KPZ} cocktail-shaken, not stirred... toasting 30 years of
	kinetically roughened surfaces.
	\newblock {\em J. Stat. Phys.}, 160(4):794--814, 2015.
	
	\bibitem[KPZ86]{kardar1986dynamic}
	M.~Kardar, G.~Parisi, and Y.~Zhang.
	\newblock Dynamic scaling of growing interfaces.
	\newblock {\em Phys. Rev. Lett.}, 56:889--892, Mar 1986.
	
	\bibitem[KKM{\etalchar{+}}09]{kitanine2009algebraic}
	N.~Kitanine, K.~K. Kozlowski, J.-M. Maillet, N.~A. Slavnov, and V.~Terras.
	\newblock Algebraic {B}ethe ansatz approach to the asymptotic behavior of
	correlation functions.
	\newblock {\em J. Stat. Mech. Theory Exp.}, 2009(04):P04003, 2009.
	
	\bibitem[Kup02]{kuperberg2002symmetry}
	G.~Kuperberg.
	\newblock Symmetry classes of alternating-sign matrices under one roof.
	\newblock {\em Ann. Math.}, 156:835--866, 2002.
	
	\bibitem[Lig75]{liggett1975ergodic}
	T.~M. Liggett.
	\newblock Ergodic theorems for the asymmetric simple exclusion process.
	\newblock {\em Trans. Amer. Math. Soc.}, 213:237--261, 1975.
	
	\bibitem[Mac95]{macdonald1995symmetric}
	I.~G. Macdonald.
	\newblock {\em Symmetric functions and Hall polynomials}, volume 354.
	\newblock Clarendon press Oxford, 1995.
	
	\bibitem[OSZ14]{o2014geometric}
	N.~O'Connell, T.~Sepp{\"a}l{\"a}inen, and N.~Zygouras.
	\newblock Geometric {RSK} correspondence, {W}hittaker functions and symmetrized
	random polymers.
	\newblock {\em Invent. Math.}, 197(2):361--416, 2014.
	
	\bibitem[OR03]{okounkov2003correlation}
	A.~Okounkov and N.~Reshetikhin.
	\newblock Correlation function of {S}chur process with application to local
	geometry of a random 3-dimensional {Y}oung diagram.
	\newblock {\em J. Am. Math. Soc.}, 16(3):581--603, 2003.
	
	\bibitem[Olv74]{olver1974asymptotics}
	F.~W.~J. Olver.
	\newblock {\em Asymptotics and special functions}.
	\newblock Academic press, 1974.
	
	\bibitem[Par17]{parekh2017kpz}
	S.~Parekh.
	\newblock The {KPZ} limit of {ASEP} with boundary.
	\newblock {\em arXiv preprint arXiv:1711.05297}, 2017.
	
	\bibitem[Rai00]{rains2000correlation}
	E.~M. Rains.
	\newblock Correlation functions for symmetrized increasing subsequences.
	\newblock {\em arXiv preprint math/0006097}, 2000.
	
	\bibitem[Rai18]{rains2014multivariate}
	E.~M. Rains.
	\newblock Multivariate quadratic transformations and the interpolation kernel.
	\newblock {\em SIGMA, arXiv:1408.0305}, 14(019), 2018.
	
	\bibitem[SI04]{sasamoto2004fluctuations}
	T.~Sasamoto and T.~Imamura.
	\newblock Fluctuations of the one-dimensional polynuclear growth model in
	half-space.
	\newblock {\em J. Stat. Phys.}, 115(3-4):749--803, 2004.
	
	\bibitem[SS10]{sasamoto2010exact}
	T.~Sasamoto and H.~Spohn.
	\newblock Exact height distributions for the {KPZ} equation with narrow wedge
	initial condition.
	\newblock {\em Nuclear Phys. B}, 834(3):523--542, 2010.
	
	\bibitem[SW14]{sasamoto2012combinatorics}
	T.~Sasamoto and L.~Williams.
	\newblock Combinatorics of the asymmetric exclusion process on a semi-infinite
	lattice.
	\newblock {\em Journal of Combinatorics}, 5(4):419--434, 2014.
	
	\bibitem[Skl88]{sklyanin1988boundary}
	E.~K. Sklyanin.
	\newblock Boundary conditions for integrable quantum systems.
	\newblock {\em J. Phys. A.}, 21(10):2375, 1988.
	
	\bibitem[Spo83]{spohn1983long}
	H.~Spohn.
	\newblock Long range correlations for stochastic lattice gases in a
	non-equilibrium steady state.
	\newblock {\em J. Phys. A}, 16(18):4275, 1983.
	
	\bibitem[Spo16]{spohn2016kardar}
	H.~Spohn.
	\newblock The {Kardar-Parisi-Zhang} equation -- a statistical physics
	perspective.
	\newblock {\em lecture notes for Les Houches summer school, july 2015,
		arXiv:1601.00499}, 2016.
	
	\bibitem[TW94]{tracy1994level}
	C.~A. Tracy and H.~Widom.
	\newblock Level-spacing distributions and the {A}iry kernel.
	\newblock {\em Comm. Math. Phys.}, 159(1):151--174, 1994.
	
	\bibitem[TW96]{tracy1996orthogonal}
	C.~A. Tracy and H.~Widom.
	\newblock On orthogonal and symplectic matrix ensembles.
	\newblock {\em Comm. Math. Phys.}, 177(3):727--754, 1996.
	
	\bibitem[TW05]{tracy2005matrix}
	C.~A. Tracy and H.~Widom.
	\newblock Matrix kernels for the {G}aussian orthogonal and symplectic
	ensembles.
	\newblock {\em Ann. Inst. Fourier}, 55(6):2197--2207, 2005.
	
	\bibitem[TW09]{tracy2009asymptotics}
	C.~A. Tracy and H.~Widom.
	\newblock Asymptotics in {ASEP} with step initial condition.
	\newblock {\em Comm. Math. Phys.}, 290(1):129--154, 2009.
	
	\bibitem[TW13a]{tracy2013asymmetric}
	C.~A. Tracy and H.~Widom.
	\newblock The asymmetric simple exclusion process with an open boundary.
	\newblock {\em J. Math. Phys.}, 54(10):103301, 2013.
	
	\bibitem[TW13b]{tracy2013bose}
	C.~A. Tracy and H.~Widom.
	\newblock The {B}ose gas and asymmetric simple exclusion process on the
	half-line.
	\newblock {\em J. Stat. Phys.}, 150(1):1--12, 2013.
	
	\bibitem[Tsi06]{tsilevich2006quantum}
	N.~V. Tsilevich.
	\newblock Quantum inverse scattering method for the $q$-boson model and
	symmetric functions.
	\newblock {\em Funct. Anal. Appl.}, 40(3):207--217, 2006.
	
	\bibitem[USW04]{uchiyama2004asymmetric}
	M.~Uchiyama, T.~Sasamoto, and M.~Wadati.
	\newblock Asymmetric simple exclusion process with open boundaries and
	{Askey--Wilson} polynomials.
	\newblock {\em J. Phys. A}, 37(18):4985, 2004.
	
	\bibitem[WZJ16]{wheeler2016refined}
	M.~Wheeler and P.~Zinn-Justin.
	\newblock {Refined Cauchy/Littlewood identities and six-vertex model partition
		functions: III. Deformed bosons}.
	\newblock {\em Adv. Math.}, 299:543--600, 2016.
	
\end{thebibliography}
%\bibliographystyle{mystyle-alpha}

\newcommand{\etalchar}[1]{$^{#1}$}

\end{document}